\documentclass[11pt,a4paper]{scrartcl}

\makeatletter
\DeclareOldFontCommand{\rm}{\normalfont\rmfamily}{\mathrm}
\DeclareOldFontCommand{\sf}{\normalfont\sffamily}{\mathsf}
\DeclareOldFontCommand{\tt}{\normalfont\ttfamily}{\mathtt}
\DeclareOldFontCommand{\bf}{\normalfont\bfseries}{\mathbf}
\DeclareOldFontCommand{\it}{\normalfont\itshape}{\mathit}
\DeclareOldFontCommand{\sl}{\normalfont\slshape}{\@nomath\sl}
\DeclareOldFontCommand{\sc}{\normalfont\scshape}{\@nomath\sc}
\makeatother

\usepackage{a4wide}
\usepackage{latexsym}
\usepackage[USenglish]{babel}
\usepackage[utf8]{inputenc}
\usepackage[T1]{fontenc}
\usepackage{amssymb}
\usepackage{hyperref}

\usepackage[amsthm,thmmarks]{ntheorem}
\usepackage{amsmath}
\usepackage{mathrsfs}
\usepackage{graphicx}
\usepackage{mathdots}
\usepackage{amsfonts}
\usepackage{amsxtra}
\usepackage{url}
\usepackage{mathptmx}
\usepackage{helvet}
\usepackage{stmaryrd}
\usepackage{dsfont}
\usepackage{setspace}
\usepackage{enumerate}
\usepackage{MathDef}
\usepackage{verbatim}

\usepackage{color}
\usepackage{cleveref}
\usepackage{cancel}

\def\dif{\mathrm{d}}

\def\Pas{\P\text{-a.s.}}
\def\diag{\text{diag}}
\renewcommand{\P}{\ensuremath{\mathbb{P}}}
\DeclareMathOperator*{\argmin}{arg\,min}
\newcommand{\vt}{\ensuremath{\vartheta}}

\DeclareMathOperator{\tr}{tr}
\def\cin{\xrightarrow{n\to\infty}}
\def\cip{\xrightarrow[\hphantom{\longrightarrow}]{p}}
\def\cas{\xrightarrow[\hphantom{\longrightarrow}]{a.s.}}
\def\cid{\xrightarrow[\hphantom{\longrightarrow}]{w}}
\def\cil2{\xrightarrow[\hphantom{\longrightarrow}]{L^2}}
\def\ccip{\xrightarrow{ ~p~ }}

\def\ccid{\xrightarrow{ ~w~ }}
\def\ccil2{\xrightarrow[c]{ \,L^2\, }}
\def\limn{\lim_{n\to\infty}}
\DeclareMathAlphabet\mathbfcal{OMS}{cmsy}{b}{n} %Bold

\theoremstyle{plain}
\newtheorem{theorem}{Theorem}[section]
\newtheorem{lemma}[theorem]{Lemma}
\newtheorem{proposition}[theorem]{Proposition}
\newtheorem{definition}[theorem]{Definition}
\newtheorem{corollary}[theorem]{Corollary}
\newtheorem{assumptionletter}{{\textbf{Assumption}}}

\theoremstyle{definition}

\newtheorem{remark}[theorem]{Remark}

\numberwithin{equation}{section}
\allowdisplaybreaks

\title{Quasi-maximum likelihood estimation for cointegrated  continuous-time state space models observed at low frequencies}

\author{Vicky Fasen-Hartmann \setcounter{footnote}{1}\thanks{Institute of Stochastics, Englerstra{\ss}e 2,
D-76131 Karlsruhe, Germany, email: vicky.fasen@kit.edu} \label{fnote}%\emph{Email:}
\thanks{Financial
support by the Deutsche Forschungsgemeinschaft through the research
grant FA 809/2-2 is gratefully acknowledged.} \and Markus Scholz \thanks{Allianz Lebensversichung-AG, Reinsburgstra{\ss}e 19, D-70197 Stuttgart, Germany. }}

\date{}

\begin{document}
%
%%%%%%%%%%%%%%%%%  Titel und Autoreninformationen  %%%%%%%%%%%%%%%%%%%%%%%%%%%%%%%%
\maketitle
%
%%%%%%%%%%%%%%%%%  Abstract und AMS Classification  %%%%%%%%%%%%%%%%%%%%%%%%%%%%%%%
\begin{abstract} \vspace*{-1cm}
In this paper, we investigate quasi-maximum likelihood (QML) estimation  for the parameters of a cointegrated solution of a continuous-time linear state space model  observed at discrete time points. The class of cointegrated solutions of continuous-time linear state space models is equivalent to the class of cointegrated continuous-time ARMA (MCARMA) processes.
As a start, some pseudo-innovations are constructed to be able to define a
QML-function. Moreover,
 the parameter vector is divided appropriately in long-run and short-run parameters using a
representation for cointegrated solutions of continuous-time linear state space models  as a sum of a  Lévy process plus a stationary solution of a linear state space model.
 Then, we establish the consistency of our estimator in three steps.
First, we show  the consistency for the  QML estimator of the long-run parameters. In the next step, we calculate its consistency rate.
Finally, we use these results to prove the consistency for the QML estimator of the short-run parameters. After all, we derive the limiting distributions of the estimators. The long-run parameters are asymptotically mixed normally distributed, whereas the short-run parameters are asymptotically  normally distributed.
The performance of the QML estimator is demonstrated by a simulation study.
\end{abstract}

\noindent
\begin{tabbing}
%\emph{AMS Subject Classification 2010: }\=Primary:  62B10,  62F12, 62M86
%\\ \> Secondary:  62F10, 62M10
\\ \> Secondary:  62M10, 60F05
\end{tabbing}

\vspace{0.2cm}\noindent\emph{Keywords:} %Asymptotically normal,
Cointegration, (super-)consistency, identifiability, Kalman filter, MCARMA process, %mixed normal distribution,
pseudo-innovation, quasi-maximum likelihood estimation, state space model.

%%%%%%%%%%%%%%%%%  Einleitung  %%%%%%%%%%%%%%%%%%%%%%%%%%%%%%%%%%%%%%%%%%%%%%%%%
\section{Introduction}
%
%
%%%%%%%%%%%%%%%%%  Literaturverzeichnis  %%%%%%%%%%%%%%%%%%%%%%%%%%%

This paper deals with quasi-maximum likelihood (QML) estimation  for the parameters of a  cointegrated solution of a continuous-time linear state space model.
The source of randomness in our  model is a L\'evy process, i.e., an $\R^m$-valued stochastic process $L=(L(t))_{t\geq 0}$ with $L(0)=0_m$ $\mathbb{P}$-a.s., stationary and independent increments, and c\`adl\`ag sample paths.
A typical example of a Lévy process is a Brownian motion. More details on L\'evy processes can be found, e.g., in the monograph of Sato \cite{sato1999levy}.
For deterministic matrices $A\in \R^{N\times N}$, $B\in  \R^{N\times m}$, $C\in \R^{d\times N}$ and an $\R^m$-valued Lévy process $L$,
an $\R^d$-valued continuous-time \textsl{linear state space model} $(A,B,C,L)$  is defined by the state and observation equation
\begin{align} \label{SPR}
\begin{array}{rcl}
    \dif X(t)& = & A X(t) \dif t + B \dif L(t),\\
%\intertext{and the observation equation}
    Y(t) &=& C X(t).% \quad \hspace*{1.5cm}\text{ for } t \geq0.
\end{array}
\end{align}
The state vector process $X=(X(t))_{t \geq 0}$ is an $\R^N$-valued process and the output process $Y=(Y(t))_{t \geq 0}$ is an $\R^d$-valued process. Since the driving noise in this model is a Lévy process
the model allows flexible margins. In particular, the margins can be Gaussian if we use a Brownian motion as Lévy process.

The topic of this paper are \textsl{cointegrated} solutions $Y$ of linear state space models. Cointegrated means that $Y$ is non stationary but has stationary increments, and
there exist linear combinations of $Y$ which are stationary. The cointegration space is the space spanned by all vectors $\beta$ so that $\beta^{\mathsf{T}}Y$ is stationary.
Without any transformation of the state space model \eqref{SPR} it is impossible to see clearly if there exists a cointegrated solution, not to mention the form of
the  cointegration space.
In the case of a minimal state-space model (see Bernstein~\cite{Bernstein2009} for a definition), the eigenvalues of $A$ determine whether a solution $Y$ may be  stationary or cointegrated. If the
 eigenvalue $0$ of $A$ has the same geometric and algebraic multiplicity $0<c<\min(d,m)$, and all other eigenvalues of $A$ have negative real parts, then
there exists a cointegrated solution $Y$.
%The parameters   determining  the cointegration space are the long-run parameters
%whereas the remaining parameters are the short-run parameters.
%It is well-known in econometrics that in a cointegrated model the asymptotic behavior of the QML estimator for the long-run parameters and the
%QML estimator for the short-run parameters differ.  Therefore we need a representation for a cointegrated solution of a linear state space model
%where we can see clearly the cointegration space to get a separation in long-run and short-run parameters.
In that case
 $Y$ has the form
\begin{align}
\label{eqRepContCointSSMSepar}
    Y(t)=C_1Z +C_1B_1L(t) +Y_{st}(t),% \quad \text{ for } t\geq 0,
\end{align}
where $B_1\in \R^{c\times m}$ and $C_1\in \R^{d\times c}$ have rank $c$ (see Fasen-Hartmann and Scholz \cite[Theorem 3.2]{FasenScholz1}). The starting vector $Z$ is a $c$-dimensional random  vector. The process $Y_{st}=(Y_{st}(t))_{t\geq 0}$ is a stationary  solution of the state space model
 \begin{align} \label{SPstationaer}
\begin{array}{rcl}
        \dif X_{st}(t)& = & A_2 X(t) \dif t + B_2 \dif L(t),\\
%\intertext{and the observation equation}
    Y_{st}(t) &=& C_2 X_{st}(t),% \quad \hspace*{1.5cm}\text{ for } t \geq0.
\end{array}
\end{align}
driven by the Lévy process $L$ with $A_2\in\R^{(N-c)\times(N-c)}$, $B_2\in\R^{(N-c)\times m}$ and $C_2\in\R^{d\times(N-c)}$.
The matrices $A,A_1,A_2$, $B,B_1,B_2$, $C,C_1,C_2$ and $C_3$ are related through an invertible transformation matrix $T\in\R^{N\times N}$ such that
\beao
    TA{T}^{-1}=\begin{pmatrix}0_{c\times c} & 0_{c\times (N-c)}\\ 0_{ (N-c)\times c} & A_2\end{pmatrix}=:A', \quad
    TB=\begin{pmatrix} {B_1}^\mathsf{T},  {B_2}^\mathsf{T}\end{pmatrix}^\mathsf{T}=:B' \quad
 \mbox{and} \quad  C{T}^{-1}=\begin{pmatrix}C_1, & C_2 \end{pmatrix}=:C',
\eeao
where ${B_i}^\mathsf{T}$ denotes the transpose of $B_i$ ($i=1,2$) and $0_{ (N-c)\times c}\in\R^{(N-c)\times c}$ denotes a matrix with only zero components.
% On the other hand, $Y$ is the solution of the linear state space model $(A,B,C,L)$ with  $A=\diag(0,A_2)$, $B=(B_1^\mathsf{T},B_2^\mathsf{T})^\mathsf{T}$, and $C=(C_1^{\mathsf{T}}, C_2^{\mathsf{T}})^{\mathsf{T}}$.
The process $Y$  in \eqref{eqRepContCointSSMSepar} is obviously cointegrated with cointegration space spanned by the orthogonal of $C_1$
if the covariance matrix $\Cov(L(1))$ is non-singular. % Because then $Y$ is non-stationary  with stationary increments and $ C_{1}^{\perp\,\mathsf{T}} Y= C_{1}^{\perp \,\mathsf{T} }Y_{st}$
%is stationary but $C_{1}^{\mathsf{T}} Y$ is in any component non-stationary.
The probabilistic properties of $Y$ are analyzed in detail in Fasen-Hartmann and Scholz~\cite{FasenScholz1} and lay the
groundwork for the present paper. Remarkable is that $Y$ is a solution of the state space model
 $(A',B',C',L)$ as well.

The class
of cointegrated solutions of linear state space models is huge. They are equal to the class of cointegrated multivariate continuous-time  ARMA (MCARMA) processes
(see Fasen-Hartmann and Scholz~\cite{FasenScholz1}). As the name suggests, MCARMA processes are the continuous-time versions of the popular and well-established
ARMA processes in discrete-time. In finance and economics continuous-time models
provide the basis for option pricing, asset allocation and term structure theory. The underlying observations
of asset prices, exchange rates, and interest rates are often irregularly spaced, in particular,
in the context of high frequency data. Consequently, one often works with continuous-time models
which infer the implied dynamics and properties of the estimated model at different frequencies (see Chen et al.~\cite{CHEN2017293}).
Fitting discrete-time models to such kind of data have the drawback that the model parameters are not
time-invariant: If the sampling frequency changes, then the  parameters of the discrete-time model change as well.
The advantages of continuous-time modelling over discrete-time modelling in economics and finance
 are described in detail, i.a., in
the distinguished papers of Bergstrom~\cite{Bergstrom1990}, Phillips~\cite{Phillips1991}; Chambers, McCrorie and
 Thornton~\cite{ChambersMcCrorieThornton2017} and in   signal processing, systems and control they are described
 in Sinha and Rao~\cite{SinhaRao}. \label{pageref4}
In particular,  MCARMA models are applied in diversified fields as signal processing, systems and control
(see Garnier and Wang~\cite{GarnierWang2008}, Sinha and Rao~\cite{SinhaRaoBuch}),
high-frequency financial econometrics (see Todorov \cite{Todorov2009})  and financial mathematics (see Benth et al. \cite{benth2014}, Andresen et al. \cite{andresen}).
Thornton and Chambers~\cite{Thornton:Chambers:2012} use them as well for modelling sunspot data.
Empirical relevance of non-stationary
MCARMA processes in economics and in finance is shown, i.a., in Thornton and Chambers~\cite{Thornton:Chambers:2012,Thornton:Chambers:2016,Thornton:Chambers:2017}.

%Thornton and Chambers~\cite{Thornton:Chambers:2017} apply MCARMA models to the yield curve,  and to
% the relationship of US GPD and oil prices taking into account different frequencies. \marginpar{\textcolor{red}{überarbeiten}} Further, Thornton and %Chambers~\cite{Thornton:Chambers:2016} model short-term
%interest rates, the term structure of interest rates, and a present value model of stock prices and dividends.
%Empirical relevance of MCARMA models with MA component in sunspot, interest rate and consumption data
%is shown in Chambers and  Thornton~\cite{Thornton:Chambers:2012}.

%Bauer and Wagner \cite{BauerWagner2002} consider QML estimation of cointegrated state space models in discrete-time.
%Although our continuous-time cointegrated state space model sampled equidistantly is a cointegrated state space model in discrete-time
%it has a different representation as the model assumed in Bauer and Wagner \cite{BauerWagner2002} so that we are not able to use their results.

There is not much known about the statistical inference of cointegrated  Lévy driven MCARMA models.
 %which allow a moving average part and
% go away from the Gaussian assumption.
In the context of non-stationary MCARMA processes most attention is paid to Gaussian MCAR$(p)$ (multivariate continuous-time AR) processes: %, in particular, Gaussian MCAR$(1)$ processes which are multivariate Ornstein-Uhlenbeck processes.
 An algorithm to estimate the structural parameters in a  Gaussian
MCAR$(p)$ model by maximum-likelihood started already by Harvey and Stock~\cite{Harvey:Stock:1985,Harvey:Stock:1988,HarveyStock1989} and
were further explored in the well-known paper of Bergstrom~\cite{Bergstrom:1997}.  Zadrozny~\cite{Zadrozny1988} investigates continuous-time Brownian motion driven ARMAX models
allowing stocks and flows at different frequencies and higher order integration.
These papers use the state space representation of the MCARMA process and Kalman filtering techniques to compute the Gaussian likelihood function.  In a recent paper Thornton and Chambers~\cite{Thornton:Chambers:2017} extend the results to MCARMA
processes with mixed stock-flow data using an exact discrete-time ARMA representation of the low-frequency observed MCARMA process. However, all of the papers have in common on the one hand,
that they do not analyze the asymptotic properties of the estimators. On the other hand, they are not able to
estimate the cointegration space directly or rather relate their results to cointegrated models.
%This is important because it is well-known in econometrics that in a cointegrated model the convergence rate of the QML
%estimator for the long-run parameters and the
%QML estimator for the short-run parameters differ.
 %Even for the much easier Gaussian model properties of the estimators are not known.
%An alternative
%spectral likelihood–based method was suggested by Robinson (1993) for that may include exogenous variables
%and mixtures of stocks and flows

Besides, statistical inference and identification of  continuously and discretely observed cointegrated Gaussian MCAR$(1)$ processes, which are  homogeneous Gaussian diffusions, are considered  in
Kessler and Rahbek \cite{KesslerRahbek2001,Kessler}; Stockmarr and
Jacobsen~\cite{litstockmarrjacobsen1994} %  do analogous things by observing the process on a discrete time-grid, which is getting
%finer and finer.
and frequency domain estimators for cointegrated Gaussian MCAR$(p)$ models
are topic of %Robinson~\cite{Robinson:1993},
Chambers and McCrorie~\cite{Chambers:McCrorie:2007}.
There are only a few papers investigating non-Gaussian cointegrated MCARMA processes. % and hence, the whole class of cointegrated linear state space models.
For example, Fasen \cite{Fasen2013} treats a multiple regression model in continuous-time. There the stationary part is a
 multivariate Ornstein-Uhlenbeck process and the process is observed on an equidistant time-grid.
 The model in  Fasen \cite{Fasen2014}
  is similar but the stationary part is an MCARMA process  and the process is observed on a high-frequency time grid.

The aim of this paper is to investigate  QML estimators  for $C_1,B_1$ and the parameters of the stationary process $Y_{st}$ from the discrete-time observations $Y(h),\ldots,Y(nh)$ where $h>0$ is fixed.
The parameters of $C_1$ are the long-run parameters, whereas the other parameters are the short-run parameters.
Although there exist results on QML for discrete-time processes they can unfortunately not directly  applied to the sampled process
for the following reasons.

MCARMA processes sampled equidistantly belong to the class
of  ARMA processes  (see Thornton and Chambers~\cite{Thornton:Chambers:2017} and
 Chambers, McCrorie and Thornton \cite{ChambersMcCrorieThornton2017}). But identification problems arise from employing the
 ARMA structure for the estimation of MCARMA parameters. That is until now an unsolved problem (see as well the overview article Chambers, McCrorie and Thornton \cite{ChambersMcCrorieThornton2017}).
  %Among them is the aliasing problem.
  %Even stationary solutions of state space models have
%semi-explicit ARMA representations (see Schlemm and Stelzer \cite[Theorem 4.2]{SchlemmStelzer2012}).
 Moreover, in this representation the innovations are only uncorrelated and not iid (independent and identically distributed). However, statistical inference for cointegrated ARMA models is
done only for an iid noise elsewise even
a Gaussian white noise, see, e.g., the monographs of Johansen \cite{Johansen1991},  Lütkepohl~\cite{Luetkepohl2005} and Reinsel~\cite{Reinsel},
and cannot be used for estimation of Lévy driven MCARMA processes.

Another attempt is to use the representation of the sampled continuous-time state space model as discrete-time state space model (see Zadrozny~\cite{Zadrozny1988}).
That is what we do in this paper.
Sampling $Y$ with distance $h>0$ results in
 $Y^{(h)}:=(Y(kh))_{k\in\N_0}=:(Y_k^{(h)})_{k\in\N_0}$, a cointegrated solution of the discrete time state-space model
 \begin{align} \label{e1}
\begin{array}{rcl}
    X_k^{(h)}& = & \e^{Ah}X_k^{(h)} + \xi_k^{(h)},\\
%\intertext{and the observation equation}
    Y^{(h)}_k &=& C X^{(h)}_k,% \quad \hspace*{1.5cm}\text{ for } k\in\N_0,
\end{array}
\end{align}
where $(\xi_k^{(h)})_{k\in\N_0}:=(\int_{(k-1)h}^{kh}\e^{A(kh-t)}B \dif L(t))_{k\in\N_0}$ is an iid  sequence.
%Obviously $Y^{(h)}$ is cointegrated as well.
For cointegrated solutions of discrete-time state space models of the form
 \begin{align} \label{SS}
\begin{array}{rcl}
    X_k& = & AX_k + \epsilon_k,\\
%\intertext{and the observation equation}
    Y_k &=& C X_k+ \epsilon_k,% \quad \hspace*{1.5cm}\text{ for } k\in\N_0,
\end{array}
\end{align}
where $(\epsilon_k)_{k\in\N_0}$ is a white noise, asymptotic properties of the QML estimator
were investigated in   the unpublished work of Bauer and Wagner \cite{BauerWagner2002}.
An essential difference between the state space model \eqref{e1} and \eqref{SS} is that in \eqref{e1}
the noise is only going into the state equation, whereas in \eqref{SS} it is going into both the state
and the observation equation. That is an essential difference. Because an advantage of model \eqref{SS} over our state space model is that it is already in innovation form, i.e.,
 the white noise $(\epsilon_k)_{k\in\N_0}$ can be represented
 by finitely many past values of $(Y_k)_{k\in\N_0}$ due to
 \beam \label{stern}
    \epsilon_k=Y_k-C(A-BC)^kX_0-C\sum_{j=1}^k(A-BC)^{j-1}BY_{k-j}.
 \eeam
%  However, their model excludes
%not only our model but  also aggregated linear processes.
But in our model  \eqref{e1} it is not possible to write the noise $(\xi_k^{(h)})_{k\in\N_0}$
by the past of $(Y_k^{(h)})_{k\in\N_0}$.  Therefore, we are not able to apply the asymptotic results of  Bauer and Wagner~\cite{BauerWagner2002}
to the setting of this paper.

%Ahn and Reinsel~\cite{AhnReinsel} and Yap and Reinsel~\cite{YapReinsel}.

%As estimation method we use a QML method.
 % The groundwork for this paper is laid in Fasen-Hartmann and Scholz~\cite{FasenScholz1}.
We use the Kalman-filter to calculate the linear innovations
and to construct an error correction form  (see Fasen-Hartmann and Scholz~\cite[Proposition 5.4 and Theorem 5.7]{FasenScholz1}). However,  the linear innovations and the error correction form use infinitely many past values
in contrast to the usual finite order form for VARMA models and discrete-time state space models  as, e.g., in~ Lütkepohl and Claessen~\cite{LuetkepohlClaessen}, Saikkonen~\cite{Saikkonen92}, Yap and Reinsel~\cite{YapReinsel95a}  and respectively, Aoki~\cite{Aoki90}, Bauer and Wagner~\cite{BauerWagner2002} (see \eqref{stern}).
% making calculations more involved.
 %Infinite means that we use infinitely many past values. % although $Y_t$ is only defined for $t\geq t_0$.
%However, in practice we do not have infinitely many past values of $Y^{(h)}$ for our QML method. Thus, we have
%to approximate  the  innovations without using the past values. %Not only from the practical point of view problems arise by the infinite order representation
%but also from the theoretical point of view since $Y(t)$ is only defined for $t\geq 0$ excluding the infinite past as well. %Since $Y^{(h)}$
%is non-stationary it is not obvious to define $Y(-kh)$ for negative values.
%Therefore we have to develop a method to construct $Y(-kh)$ for $k\in\N$ making it possible to interpret this infinite error correction form appropriate.
Indeed,
the linear innovations are stationary, but in general it is not possible to say anything about their mixing properties. Hence, standard limit results for stationary mixing processes cannot be applied.
 For more details in the case of stationary MCARMA models we refer to
Schlemm and Stelzer~\cite{SchlemmStelzer2012}.

 The representation of the innovations motivates the definition of the pseudo-innovations
and hence, the pseudo-Gaussian likelihood function.
The term pseudo reflects in the first case that we do not use the real innovations and in the second case that we do not have a Gaussian model.
This approach is standard for stationary models (see~Schlemm and Stelzer~\cite{SchlemmStelzer2012a}) but not so well investigated for non-stationary models.
In  our model, the pseudo-innovations   are  as well non-stationary  and hence,  classical methods for QML estimation for stationary models do not work, e.g., the convergence of the quasi-maximum-likelihood function by a law of large numbers or an ergodic theorem.

Well-known achievements on ML estimation for integrated and cointegrated processes in discrete time
are Saikkonen \cite{Saikkonen1993,Saikkonen1995}. Under the constraint that the ML estimator is consistent and the long-run parameter estimator
satisfies some appropriate order of consistency condition, the papers present stochastic equicontinuity criteria for the standardized score vector and the standardized Hessian matrix such
that the asymptotic distribution of the ML estimator can be calculated. The main contributions of these papers are the derivation of stochastic equicontinuity and
weak convergence results of various first and second order sample moments from integrated processes. The concepts are applied to a ML estimator in a simple
regression model with integrated and stationary regressors.

In this paper, we follow the ideas of Saikkonen \cite{Saikkonen1995} to derive the asymptotic distribution of the QML estimator by providing evidence that these three criteria are satisfied.
However, our model does not satisfy the stochastic equicontinuity conditions of Saikkonen \cite{Saikkonen1993,Saikkonen1995} such that the weak convergence results
 of these papers
cannot be applied directly. But we use a similar approach. In the derivation of the consistency of the QML estimator   we even require local Lipschitz continuity
  for some parts of the likelihood-function which is stronger than local stochastic equicontinuity. For this reason we pay our attention in this paper to local Lipschitz continuity instead of stochastic equicontinuity.
%
%In econometrics it is well known that for cointegrated models
%proving the consistency and the asymptotic behavior of the QML estimator cannot be done by standard techniques.

Although Saikkonen \cite{Saikkonen1993,Saikkonen1995} presents no general conditions for the analysis of the consistency and the order of consistency of a ML estimator in an integrated or cointegrated
model,  the verification of the consistency of the ML estimator in the regression example of Saikkonen~\cite{Saikkonen1995} suggests, how to proceed in more general models. That is done by a stepwise approach:
In the first step, we prove the consistency of the long-run parameter estimator and in the second step its consistency rate; the
long-run parameter estimator is super-consistent. In the third step, we are
able to prove the consistency of the short-run parameter estimator.
 However,   important for the  proofs is, as in Saikkonen~\cite{Saikkonen1995}, the appropriate division of the likelihood-function where
 one part of the likelihood-function  depends only on the short-run parameters and is based on stationary processes. This decomposition is not obvious
and presumes as well a  splitting of the pseudo-innovations in a non-stationary and a stationary part  depending only on the
short-run parameters.

The paper is structured on the following way.
An introduction into  QML estimation for cointegrated continuous-time linear state space models is given in \Cref{Section:3}.
 First, we state in \Cref{sec:parametric model} the assumptions about our parametric family of cointegrated output processes $Y$. Then, we define the pseudo-innovations for the QML estimation by the Kalman filter in \Cref{Section:3:subsection:0}. Based on the pseudo-innovations we calculate the pseudo-Gaussian log-likelihood function in \Cref{Section:3:subsection:1}. % A major problem arising in the estimation of multivariate continuous-time models observed discretely is the identifiability. The models cannot only have many redundancies due to the  multidimensionality  but they also may be indistinguishable due to the equidistant sampling. We overcome these problems in
 In \Cref{Section:3:subsection:2} we introduce some identifiability conditions to get a unique minimum of the likelihood function.
The main results of this paper are given  in \Cref{sec:4} and \Cref{sec:5}.
First, we show  the consistency of the QML estimator in \Cref{sec:4}.
%After proving the consistency of the long-run QML estimator in \Cref{sec:4.1} we calculate its  consistency rate
 %in \Cref{Sec:cons:longrun}. This knowledge we use to show the consistency of the short-run QML estimator in \Cref{sec4:sub2}. %Hence, we show the consistency result in three steps:
%First we show the consistency of the long-run parameter estimator, then determine the order of consistency and lastly show the consistency of the short-run parameter estimator.
Next, we calculate the asymptotic distribution of the QML estimator in \Cref{sec:5}. The short-run QML estimator is asymptotically normally distributed
 and mimics the properties of QML estimators for stationary models. In contrast,  the long-run QML estimator is asymptotically mixed normally distributed
 with a convergence rate of $n$ instead of $\sqrt{n}$ as occurring in stationary models.
 Finally, in Section~\ref{Section: Simulation study} we show the performance of our estimator in a simulation study,
 and in \Cref{sec:conclusion} we give some conclusions. Eventually,
   in \Cref{Appendix: Auxiliary results} we present some asymptotic results and local Lipschitz continuity conditions which we use throughout the paper. Because of their technicality and to keep the paper readable,  they are moved to the appendix.

\subsubsection*{Notation} \vspace*{-0.3cm}

%We  use the notation $\stackrel{\DD}{\to}$ for weak convergence and
%$\stackrel{\p}{\to}$ for convergence in probability.
%For two random vectors $Z_1,\,Z_2$ the notation $Z_1 \stackrel{\DD}{=} Z_2$ means equality in distribution.
We use as norms the Euclidean norm $ \lVert\cdot \rVert$ in $\R^d$
and the Frobenius norm $ \lVert\cdot \rVert$ for
matrices, which is submultiplicative. $0_{d\times s}$ denotes the  zero matrix in $\R^{d\times s}$ and $I_{d}$ is the
identity matrix in $\R^{d\times d}$.
 For a matrix $A\in\R^{d\times d}$ we denote by $A^T$ its transpose, $\tr(A)$ its trace,  $\det(A)$ its determinant,  $\rank~A$ its rank,  $\lambda_{\text{min}}(A)$ its smallest eigenvalue and  $\sigma_{\text{min}}(A)$ its smallest singular value. If $A$ is symmetric and positive semi-definite, we write $A^{\frac12}$ for the principal square root, i.e., $A^{\frac12}$ is a symmetric, positive semi-definite matrix satisfying $A^\frac12 A^\frac12 = A$.
For a matrix  $A \in \R^{d \times s}$ with $\rank~A=s$, $A^{\perp}$ is a
$d\times (d-s)$-dimensional matrix with rank $(d-s)$ satisfying $A^\mathsf{T}A^{\perp}=0_{s\times(d-s)}$ and $A^{\perp \mathsf{T}}A=0_{(d-s)\times s}$.
%For two matrices $A,B$ we write $\diag(A,B)$ for a block diagonal matrix whose first block is the matrix $A$ and the second
%block is the matrix $B$.
%The space of all $m\times n$-dimensional real-valued matrices is  $M_{m,n}(\R)$.
%,  the set of $m$-dimensional symmetric positive-definite matrices is denoted by $\Sympp_m$.
%and $\text{GL}(N):=\{A\in\R^{N\times N}:\det~A\not=0\}$ for some $N\in\N$.
For two matrices $A \in \R^{d \times s}$ and $B \in \R^{r \times n}$, we denote by $A \otimes B$ the Kronecker product which is an element of $\R^{dr \times sn}$, by
$\text{vec}(A)$ the operator which converts the matrix $A$ into a column vector and by $\text{vech}(A)$  the operator which converts a symmetric matrix $A$ into a
column vector by vectorizing only the lower triangular part of $A$.
We write $\partial_i$ for the partial derivative operator with respect to the $i^{th}$ coordinate and $\partial_{i,j}$
for the second partial derivative operator
with respect to the $i^{th}$ and $j^{th}$ coordinate.
Further, for a matrix function $f(\vt)$ in $\R^{d\times m}$ with $\vt\in\R^s$ the gradient with respect to the parameter vector $\vt$
 is denoted by $\nabla_{\vartheta} f(\vt)=\frac{\partial\text{vec}(f(\vt))}{\partial\vartheta^T}\in\R^{dm\times s}$.
Let $\xi=(\xi_k)_{k\in\N}$ and $\eta=(\eta_k)_{k\in\N}$ be $d$-dimensional stochastic processes then
$\Gamma_{\xi,\eta}(l)=\Cov(\xi_1,\eta_{1+l})$ and $\Gamma_{\xi}(l)=\Cov(\xi_1,\xi_{1+l})$, $l\in\N_0$, are the covariance functions.
Finally,
we denote with $\cid$ weak convergence and with $\cip$ convergence in probability. In general $\mathfrak{C}$ denotes a constant which may change from line to line.

\section{Step-wise quasi-maximum likelihood estimation} \label{Section:3}
%% ==============================

\subsection{Parametric model} \label{sec:parametric model}

Let $\Theta\subset\R^s$, $s\in\N$, be a parameter space.
We assume that we have a parametric family $(Y_\vt)_{\vt\in\Theta}$  of solutions of continuous-time cointegrated linear state space models of the form
\begin{align}
\label{eqRepContCointSSMSepartheta}
    Y_{\vt}(t)=C_{1,\vt}Z +C_{1,\vt}B_{1,\vt}L_{\vt}(t) +Y_{st,\vt}(t), \quad t\geq 0,
\end{align}
where $Z$ is a random starting vector, $L_\vt=(L_\vt(t))_{t\geq 0}$ is a Lévy process and  $Y_{st,\vt}=(Y_{st,\vt}(t))_{t\geq 0}$ is a  stationary  solution of the  state-space model %$(A_{2,\vartheta},B_{2,\vartheta},C_{2,\vartheta},L_\vartheta)$ so that
\begin{align}   \label{statespacestationary}
\begin{array}{rl}
    \dif X_{st,\vt}(t)& =  A_{2,\vt} X_{st,\vt}(t) \dif t + B_{2,\vt} \dif L_{\vt}(t), \\
%\intertext{and the observation equation}
    Y_{st,\vt}(t) &= C_{2,\vt} X_{st,\vt}(t),
    \end{array}
\end{align}
with $ A_{2,\vartheta}\in \R^{(N-c)\times (N-c)} $, $B_{1,\vartheta}\in \R^{c\times m}$, $B_{2,\vartheta}\in \R^{(N-c)\times m}$, $C_{1,\vartheta}\in \R^{d\times c}$ and $C_{2,\vartheta}\in \R^{d\times (N-c)}$ where \linebreak $c\leq  \min(d,m)\leq N$. In the  parameterization of the L\'evy process $L_\vartheta$
only the covariance matrix $\Sigma_\vt^L$ of  $L_\vartheta$ is parameterized.

The parameter vector of the underlying process $Y$ is denoted by $\vartheta^0$, i.e., \linebreak
$(A_{2},B_{1},B_{2},C_{1},C_{2},L)=(A_{2,\vt^0},B_{1,\vt^0},B_{2,\vt^0},C_{1,\vt^0},C_{2,\vt^0},L_{\vt^0})$ where $Y_{st}$
is a stationary solution of the  state space model $(A_2,B_2,C_2,L)$. Throughout the paper, we
shortly write $(A_{2,\vt},B_{1,\vt},B_{2,\vt},C_{1,\vt},C_{2,\vt},L_\vt)$ for the  cointegrated state space model with solution $Y_\vt$ as defined in \eqref{eqRepContCointSSMSepartheta}. To be more precise we have the following assumptions on our model.

%\begin{remark}
% Usually  $C_{1,\vartheta_1}^{\perp}\in\R^{d\times(d-c)}$ is not uniquely defined. However, in this paper we take a unique map $\vt_1\mapsto C_{1,\vartheta_1}^{\perp}$
% by assuming  that
%      %Since the $\rank(C_{1,\vt}B_{1,\vt})=c$ and $C_{1,\vt}B_{1,\vt}\in\R^{d\times m}$ there exists a unique lower triangular matrix $\wt C_{1,\vt}\in\R^{d\times c}$ and
% $\wt B_{1,\vt}\in\R^{c\times m}$ with $C_{1,\vt}B_{1,\vt}=\wt C_{1,\vt}\wt B_{1,\vt}$ and $\wt C_{1,\vt}^{\mathsf{T}}\wt C_{1,\vt}=I_c$. Similarly then there exists a unique
%    Moreover,  $C_{1,\vt}^{\perp\,\mathsf{T}} Y_\vt= C_{1,\vt}^{\perp \,\mathsf{T} }Y_{st,\vt}$
%is stationary but $C_{1,\vt}^{\mathsf{T}} Y_\vt$ is in any component non-stationary so that $Y_\vt$ is indeed cointegrated.
%\end{remark}

\begin{assumptionletter}
\label{AssMBrownMot}
For any $\vartheta \in\Theta$ the cointegrated state space model  $(A_{2,\vartheta},B_{1,\vartheta},B_{2,\vartheta},C_{1,\vartheta},C_{2,\vartheta},L_\vartheta)$ satisfies the following conditions:
\begin{enumerate}
\item[(A1)]
The parameter space $\Theta$ is a compact subset of $\R^s$.
\item[(A2)]
The true parameter vector $\vartheta^0$ lies in the interior of the parameter space $\Theta$.
\item[(A3)]   The L\'evy process $L_\vartheta$ has mean zero and non-singular covariance matrix $\Sigma^L_\vartheta=\E [ L_\vartheta(1) L_\vartheta(1)^\mathsf{T}]$.
Moreover, there exists a $\delta>0$ such that $\E \Vert L_{\vartheta}(1)\Vert^{4+\delta}<\infty$ for any $\vartheta\in\Theta$.

\item[(A4)]
The eigenvalues of $A_{2,\vartheta}$ have strictly negative real parts.
\item[(A5)] The triplet $(A_{2,\vt},B_{2,\vt},C_{2,\vt})$ is minimal with McMillan degree $N-c$  (see Hannan and Deistler~\cite[Chapter 4.2]{hannandeistler2012} for the definition of McMillan degree).
\item[(A6)] The matrices $B_{1,\vartheta}\in\R^{c\times m}$ and $C_{1,\vartheta}\R^{d\times c}$ have full rank $c\leq \min(d,m)$.
%\item[(A7)] $C_{1,\vartheta_1}$ is a positive lower triangular matrix   satisfying $C_{1,\vartheta_1}^{\mathsf{T}}C_{1,\vartheta_1}=I_c$.
\item[(A7)] The $c$-dimensional starting random vector $Z$  does not depend on $\vartheta$, $\E\|Z\|^2<\infty$ and $Z$ is independent of $L_\vartheta$.
\item[(A8)] The functions $\vt\mapsto A_{2,\vt}$, $\vt\mapsto B_{i,\vt}$, $\vt\mapsto C_{i,\vt}$ for $i\in\{1,2\}$, $\vt\mapsto \Sigma^L_{\vt}$
 and  $\vt_1\mapsto C_{1,\vartheta_1}^{\perp}$ are three times continuously differentiable, where $C_{1,\vt}^{\perp}$ is the unique
  lower triangular matrix with  $C_{1,\vt}^{\perp \mathsf{T}} C_{1,\vt}^{\perp}=I_{d-c}$ and
   $C_{1,\vt}^{\perp \mathsf{T}} C_{1,\vt}=0_{(d-c)\times c}$.
\item[(A9)] $A_{\vt}:=\diag(0_{c\times c},A_{2,\vt})\in\R^{N\times N}$, $B_{\vt}:=(B_{1,\vt}^\mathsf{T},B_{2,\vt}^\mathsf{T})^\mathsf{T}\in\R^{N\times m}$, $C_{\vt}:=(C_{1,\vt}, C_{2,\vt})\in\R^{d\times N}$. Moreover, $C_{\vt}$
    has full rank $d\leq N$.
\item[(A10)] For any $\lambda,\lambda'\in\sigma(A_\vt)=\sigma(A_{2,\vt})\cup\{0\}$ and any $k\in\Z\backslash\{0\}$: $\lambda-\lambda'\not=2\pi k/h$ (Kalman-Bertram criterion).
\end{enumerate}
\end{assumptionletter}

%As a consequence of (ii)-(iii), all matrix functions arising in the estimation procedure are continuously differentiable and together with the compact parameter space (i) this implies Lipschitz continuity.

\begin{remark} $\mbox{}$
\begin{enumerate}
    \item[(i)] (A1) and (A2) are standard assumptions for QML estimation.
    \item[(ii)] Assumption (A3)-(A4) are sufficient assumptions to guarantee that there exists a stationary solution $Y_{st,\vt}$ of
    the state space model \eqref{statespacestationary} (see Marquardt and Stelzer~\cite{MarquardtStelzer2007}).
    \item[(iii)] Due to the assumption (A5) the state space representation of $Y_{st,\vt}$ in \eqref{statespacestationary}
    with  $A_{2,\vartheta}\in \R^{(N-c)\times (N-c)} $, $B_{2,\vartheta}\in \R^{(d-c)\times m}$  and $C_{2,\vartheta}\in \R^{d\times (N-c)}$ is unique  up to a change of basis.
%    \item   Since $C_{1}^\perp Y=C_{1}^\perp Y_{st}$,
%    due Assumption (A6) the space spanned by $C_{1}^\perp$ is the cointegration space and $r=\rank C_{1}^\perp=d-c$  is the cointegration rank.
    \item[(iv)] We require that $c$ respectively the cointegration rank $r=d-c$ is known in advance to be able to estimate the model adequately. In reality, it is necessary to estimate first the cointegration rank $r$ and obtain from this $c=d-r$. Possibilities to do this is via information criteria.
    \item[(v)] Using the notation in (A9) it is possible to show that $Y_\vt$ is the solution of the state space model $(A_\vt,B_\vt,C_\vt,L_\vt)$.
    Furthermore, on account of (A5) and (A6), the state space model $(A_\vt,B_\vt,C_\vt)$ is minimal with McMillan degree $N$ (see Fasen-Hartmann and Scholz~\cite[Lemma 2.4]{FasenScholz1})
    and hence, as well unique up to a change of basis. That in combination with (A10) is sufficient that
    $Y_\vt^{(h)}:=(Y_\vt^{(h)}(k))_{k\in\N_0}:=(Y_\vt(kh))_{k\in\N_0}$ is a solution of a discrete-time state space model with McMillan degree $N$ as well.
\end{enumerate}
\end{remark}

Furthermore, we assume that the parameter space $\Theta$ is a product space of the form $\Theta=\Theta_1\times\Theta_2$ with $\Theta_1\subset \R^{s_1}$ and $\Theta_2\subset \R^{s_2}$, $s=s_1+s_2$. The vector  $\vartheta=(\vartheta_1^\mathsf{T},\vartheta_2^\mathsf{T})^\mathsf{T}\in\Theta$ is a  $s$-dimensional parameter vector  where $\vartheta_1\in\Theta_1$ and $\vartheta_2\in\Theta_2$.
The idea is that $\vt_1$ is the $s_1$-dimensional  vector of long-run parameters  modelling the cointegration space and hence, responsible for the cointegration of $Y_\vt$. Whereas $\vartheta_2$ is the $s_2$-dimensional vector of short-run parameters which has no influence on the cointegration of the model.
 Since the matrix $C_{1,\vartheta}$ is responsible for the cointegration property (see Fasen-Hartmann and Scholz~\cite[Theorem 3.2]{FasenScholz1}) %, thus $C_{1,\vt}^\perp$ must span the same cointegration space.
 we parameterize $C_{1,\vartheta}$ with the sub-vector $\vt_1$ and use for all the other matrices $\vt_2$.  In summary, we parameterize the matrices with the following sub-vectors
$(A_{2,\vartheta_2},B_{1,\vartheta_2},B_{2,\vartheta_2},C_{1,\vartheta_1},C_{2,\vartheta_2},L_{\vartheta_2})$ for $(\vt_1,\vt_2)\in\Theta_1\times\Theta_2=\Theta$.
%This partitioning transfers immediately to the matrix $\Pi(\vt)=\alpha(\vt)\beta(\vt)^\mathsf{T}$ because we know that
% we can take for $\beta(\vt)$ the orthogonal complement of $C_{1,\vt}=C_{1,\vt_1}$ (see \Cref{Remark3.3}), the matrix $\beta(\vt)=\beta(\vt_1)$ depends only on $\vt_1$
% and the special form of the adjustment matrix $\alpha(\vt)$  which has rank $r$ is not of importance. If necessary for the further considerations, we write $\alpha(\vt)\beta(\vt_1)^\mathsf{T}$ otherwise we remain with the shorter notation $\Pi(\vt)$.

\subsection{Linear and pseudo-innovations}
\label{Section:3:subsection:0}

% Under \autoref{AssMBrownMot} we have the canonical form as given in \cite[Theorem 2.9]{FasenScholz1}.
In this section, we define the pseudo-innovations which are essential to define the QML function.
Sampling at distance $h>0$ maps the class of continuous-time  state space models to discrete-time  state space models.
That class of state space models is not in innovation form and hence, we use a result from Fasen-Hartmann and Scholz \cite{FasenScholz1} to calculate the linear innovations %of the observations %$(Y_1^{(h)},\ldots,Y_n^{(h)})$ of the output process $Y^{(h)}:=(Y_k^{(h)})_{k\in\N}:=(Y(kh))_{k\in\N}$
by the Kalman filter. The Kalman  filter constructs the linear innovations as
$\varepsilon_\vt^*(k)=Y_{\vt}(kh)-P_{k-1}Y_\vt(kh)$
where $P_k$ is the orthogonal projection onto $\overline{\text{span}}\{Y_\vt(lh):-\infty<l\leq k\}$  where the closure is taken in the Hilbert space of square-integrable random variables with inner product $(Z_1,Z_2)\mapsto \E (Z_1^{\mathsf{T}}Z_2)$.
Thus, $\varepsilon_\vt^*(k)$ is orthogonal to the Hilbert space generated by $\overline{\text{span}}\{Y_\vt(lh),-\infty<l<k\}$. In our setting, the linear innovations are as follows.

\begin{proposition}[Fasen-Hartmann and Scholz~\cite{FasenScholz1}]
\label{psin}
Let $\Omega_\vt^{(h)}$ be the unique solution of the discrete-time algebraic Riccati equation \begin{align*}
\Omega_\vt^{(h)}=&\mathrm{e}^{A_\vt h}\Omega_\vt^{(h)} \mathrm{e}^{A_\vt^\mathsf{T} h}-\e^{A_\vt h}\Omega_\vt^{(h)} C_\vt^\mathsf{T}\big(C_\vt\Omega_\vt^{(h)} C_\vt^\mathsf{T}\big)^{-1}C_\vt\Omega_\vt^{(h)} \mathrm{e}^{A_\vt^\mathsf{T} h}+{\Sigma}_\vt^{(h)},
\end{align*}
where
\begin{eqnarray*}
    \Sigma^{(h)}_\vt=\int_0^h\left(\begin{array}{cc}
        B_{1,\vt}\Sigma_\vt^LB_{1,\vt}^T & \e^{A_{2,\vt}u}B_{2,\vt}\Sigma_\vt^LB_{1,\vt}^T \\
        B_{1,\vt}\Sigma_\vt^{L}B_{2,\vt}^T\e^{A_{2,\vt}^Tu} & \e^{A_{2,\vt}u}B_{2,\vt}\Sigma_\vt^L B_{2,\vt}^T\e^{A_{2,\vt}^Tu}
    \end{array}\right)\,du,
\end{eqnarray*}
and $K_\vt^{(h)}=\mathrm{e}^{A_\vt h}\Omega_\vt^{(h)} C_\vt^\mathsf{T}\big(C_\vt\Omega_\vt^{(h)} C_\vt^\mathsf{T}\big)^{-1}
$ be the steady-state Kalman gain matrix. Then, the \textbf{linear innovations} $\varepsilon_\vt^*=(\varepsilon_\vt^*(k))_{k\in\N}$ of $Y_\vt^{(h)}:=(Y_\vt^{(h)}(k))_{k\in\N}:=(Y_\vt(kh))_{k\in\N}$ are the unique stationary solution
of the state space equation
\beam \label{2.3}
\begin{array}{rll}
     \varepsilon_{\vt}^*(k)&=Y_{\vt}^{(h)}(k)-C_\vt X^{*}_\vt(k),  & \text{ where }\\
     X_{\vt}^{*}(k)&=(\e^{A_\vt h}-K_\vt^{(h)} C_\vt) X^{*}_{\vt}(k-1)+K_\vt^{(h)}Y_{\vt}^{(h)}(k-1).
 %       &=&Y _{\vt}^{(h)}(k)-C_\vt(\e^{A_\vt h}-K_\vt^{(h)} C_\vt)^{k-1}K_\vt^{(h)} X_\vt(1)- \sum_{i=1}^{k-1} C_\vt(\e^{A_\vt h}-K_\vt^{(h)} C_\vt)K_\vt^{(h)}
 %    Y_{\vt}^{(h)}(k-i).
\end{array}
\eeam
Moreover, $
V_\vartheta^{(h)}=\E(\varepsilon_\vt^*(1)\varepsilon_\vt^*(1)^{\mathsf{T}})=C_\vartheta\Omega_\vartheta^{(h)} C_\vartheta^\mathsf{T}$ is the prediction covariance matrix of the Kalman filter.
\end{proposition}
%More details on the well-definedness of this definition can be found in .
%and they can also be represented in the error correction form
%\begin{align}
%\label{eqinnovationsequence}
%\varepsilon_k^{(h)}(\vartheta)
%=&-\Pi(\vartheta)Y_{k-1}^{(h)}+\overline{k}(B,\vartheta)\Delta Y_k^{(h)},\quad k\in\N
%\end{align}
%where $\Pi(\vartheta):=\alpha(\vartheta)\beta^\mathsf{T}(\vartheta)$ and the transfer function $\overline{k}(z,\vartheta):=I_d-\overline{k}(z,\vartheta)$ is given similarly as in \cite{FasenScholz1}.
%
We obtain recursively from \eqref{2.3}
\beao
   \varepsilon_{\vt}^*(k)=Y _{\vt}^{(h)}(k)-C_\vt(\e^{A_\vt h}-K_\vt^{(h)} C_\vt)^{k-1} X_\vt^{*}(1)- \sum_{j=1}^{k-1} C_\vt(\e^{A_\vt h}-K_\vt^{(h)} C_\vt)^{j-1}K_\vt^{(h)}
     Y_{\vt}^{(h)}(k-j).
\eeao
%In our case of a cointegrated process $Y_\vt$ the process $(X_\vt^{*}(k))_{k\in\N}$ is non-stationary in contrast to the stationary process $(\varepsilon_{\vt}^*(k))_{k\in\N}$.
However, the question arises which choice of $X_\vt^{*}(1)$ of the Kalman recursion results in the stationary $(\varepsilon_{\vt}^*(k))_{k\in\N}$. This we want to elaborate in the
following.

Since all eigenvalues of $(\e^{A_\vt h}-K_\vt^{(h)} C_\vt)$ lie inside the unit circle
(see Scholz~\cite[Lemma 4.6.7]{Scholz}) the matrix function $\mathsf{l}(z,\vartheta):=I_d-C_\vt\sum_{j=1}^\infty\big(\mathrm{e}^{A_\vt h}-K_\vt^{(h)}C_\vt\big)^{j-1}K_\vt^{(h)}z^j$ for $z\in\C$ is well-defined and
due to Fasen-Hartmann and Scholz \cite[Lemma 5.6]{FasenScholz1}  has the representation
as
\beao
    \mathsf{l}(z,\vartheta)=-\alpha(\vt)C_{1,\vt}^{\perp\,\mathsf{T}}z+\mathsf{k}(z,\vartheta)(1-z)
\eeao
for the linear filter $\mathsf{k}(z,\vartheta):=I_d-\sum_{j=1}^\infty {\mathsf{k}}_j(\vartheta)z^{j}$ with ${\mathsf{k}}_j(\vartheta):=\sum_{i=j}^{\infty} C_\vt(\e^{A_\vt h}-K_\vt^{(h)} C_\vt)^{i}K_\vt^{(h)}\in\R^{d\times d}$ and a matrix $\alpha(\vt)\in \R^{d\times(d-c)}$ with full rank $d-c$. This representation of $\mathsf{l}(z,\vartheta)$
helps  us to choose the initial condition  $X_\vt^{*}(1)$ in the Kalman recursion appropriate so that the linear innovations $(\varepsilon_{\vt}^*(k))_{k\in\N}$ are really stationary.
Therefore, it is important to know that the stationary process $Y_{st,\vt}$ can be defined on $\R$ as $Y_{st,\vt}(t)=\int_{-\infty}^tf_{st,\vt}(t-s)\,dL_\vt(s)$, $t\in\R$, with
$f_{st,\vt}(u)=C_{2,\vt}\e^{A_{2,\vt}u}B_{2,\vt}\1_{\left[0,\infty\right)}(u)$ and the Levy process
 $(L_\vt(t))_{t\in\R}$ is defined on the negative real-line as $L_\vt(t)=\wt L_\vt(-t-)$ for $t<0$ with an independent copy $(\wt L_\vt(t))_{t\geq 0}$ of $(L_\vt(t))_{t\geq 0}$.
Then, we have an adequate definition of $\Delta Y_\vt^{(h)}(k):=Y_{\vt}^{(h)}(k)-Y_{\vt}^{(h)}(k-1)$ for negative values as well as
 $\Delta Y_\vt^{(h)}(k)=\int_{-\infty}^{kh}f_{\Delta,\vt}(kh-s)\,d L_\vt(s)$, $k\in\Z$,
with $f_{\Delta,\vt}(u)=f_{st,\vt}(u)-f_{st,\vt}(u-h)+C_{1,\vt}B_{1,\vt}\1_{\left[0,h\right)}(u)$.
As notation, we use $\mathsf{B}$ for the backshift operator satisfying $\mathsf{B}Y_\vt^{(h)}(k)=Y_\vt^{(h)}({k-1})$.

\begin{lemma}
\label{LemInnovaSeqProp}
Let \autoref{AssMBrownMot} hold.
Then, \begin{align*}
%\label{eqInnovaSeqPropii}
    \varepsilon_{\vt}^*(k)&=-\Pi(\vartheta)Y_{\vt}^{(h)}(k-1)+ {k}(\mathsf{B},\vt)\Delta Y_{\vt}^{(h)}(k),
    %+\bigg(\Delta Y_k^{(h)}-\sum_{i=1}^\infty \overline{K}_i(\vartheta) \Delta Y_{k-i}^{(h)}\bigg)
    \quad k\in\N,
    \end{align*}
     where $\Pi(\vt)=\alpha(\vt)C_{1,\vt}^{{\perp\,\mathsf{T}}}$ and  $\mathsf{k}(\mathsf{B},\vt)\Delta Y_{\vt}^{(h)}(k)=\Delta Y_{\vt}^{(h)}(k)-\sum_{j=1}^\infty {\mathsf{k}}_j(\vt)\Delta Y_{\vt}^{(h)}(k-j)$.
    The matrix sequence $(\mathsf{k}_j(\vartheta))_{j\in\N }$ is uniformly exponentially bounded, i.e. there exist   constants $\mathfrak{C}>0$ and $0<\rho<1$ such that $\sup_{\vartheta\in\Theta}\|\mathsf{k}_j(\vartheta)\|\leq \mathfrak{C}\rho^j$, $j\in\N$.
\end{lemma}
\begin{proof}  It remains to show that $({\mathsf{k}}_j(\vartheta))_{j\in\N }$ is uniformly exponentially bounded.
The proof follows in the same line as Schlemm and Stelzer \cite[Lemma 2.6]{SchlemmStelzer2012a} using that all eigenvalues
of $(\e^{A_\vt h}-K_\vt^{(h)} C_\vt)$ lie inside the unit circle (see Scholz~\cite[Lemma 4.6.7]{Scholz}).
\end{proof}
Due to $\Pi(\vartheta)Y_{\vt}^{(h)}(k-1)=\Pi(\vartheta)Y_{\vt,st}^{(h)}(k-1)$ we receive from Lemma~\ref{LemInnovaSeqProp}
\beao
\varepsilon_{\vt}^*(k)=-\Pi(\vartheta)Y_{\vt,st}^{(h)}(k-1)+ {\mathsf{k}}(\mathsf{B},\vt)\Delta Y_{\vt}^{(h)}(k).
\eeao
From this representation  we see nicely that $(\varepsilon_{\vt}^*(k))_{k\in\N}$ is indeed a stationary process.
Defining $Y_\vt^{(h)}$ on the negative integers as
\beao
    Y_\vt^{(h)}(-k)=C_{1,\vt}Z+Y_{st,\vt}(0)-\sum_{j=0}^{k-1}\Delta Y_{\vt}^{(h)}(-j)
        =C_{1,\vt}Z+L_\vt(-kh)+Y_{st,\vt}^{(h)}(-k), \quad  k\in\N_0,
\eeao
the initial condition in the Kalman recursion is $X^{*}_\vt(1):=\sum_{j=0}^{\infty} (\e^{A_\vt h}-K_\vt^{(h)} C_\vt)^{j}K_\vt^{(h)}
     Y_{\vt}^{(h)}(-j)$ so that
\beao
\varepsilon_{\vt}^*(k)&= Y_{\vt}^{(h)}(k)-\sum_{j=1}^{\infty} C_\vt(\e^{A_\vt h}-K_\vt^{(h)} C_\vt)^{j-1}K_\vt^{(h)}
     Y_{\vt}^{(h)}(k-j).
\eeao

The representation of the linear innovations in Lemma~\ref{LemInnovaSeqProp} motivates the definition of the pseudo-innovations which are going in the likelihood function.

\begin{definition}
  The \textbf{pseudo-innovations} are defined for $k\in\N$ as
  \begin{align*}
%\label{eqInnovaSeqPropii}
    \varepsilon_k^{(h)}({\vt})=-\Pi(\vartheta)Y_{k-1}^{(h)}+ {k}(\mathsf{B},\vt)\Delta Y_{ k}^{(h)}
        =Y_k^{(h)}-\sum_{j=1}^{\infty} C_\vt(\e^{A_\vt h}-K_\vt^{(h)} C_\vt)^{j-1}K_\vt^{(h)}
     Y_{k-j}^{(h)}.
    \end{align*}
\end{definition}
The main difference of the linear innovations and the pseudo-innovations is that in the linear innovation $Y_\vt^{(h)}$ is going in, whereas
in the pseudo-innovations $Y^{(h)}$ is going in.
For $\vt=\vt^0$ the pseudo-innovations $(\varepsilon_k^{(h)}({\vt^0}))_{k\in\N}$ are the linear-innovations $(\varepsilon_{\vt^0}^*(k))_{k\in\N}$.
In Appendix~\ref{Section: Properties linear innovation} we present some probabilistic properties of the pseudo-innovations which we use in the paper.
In particular, we see that the pseudo-innovations are three times differentiable.

\subsection{Quasi-maximum likelihood estimation}
\label{Section:3:subsection:1}
We estimate the model parameters via an adapted quasi-maximum likelihood estimation method.
%In contrast to the classical quasi-maximum likelihood approach we use a step-wise approach. In particular, we use a separation of the parameter space in long-run and short-run parameters and accordingly separate the log-likelihood function.
Minus two over $n$ times the logarithm of the pseudo-Gaussian likelihood function is given by
 \begin{align*}
%\label{eqLikelFunc}
\mathcal{L}_n^{(h)}(\vartheta)
= \frac{1}{n}\sum_{k=1}^n\left[d\log 2\pi +\log \det V_\vt^{(h)}+\varepsilon_k^{(h)}(\vartheta)^\mathsf{T}\big(V_\vt^{(h)}\big)^{-1}\varepsilon_k^{(h)}(\vartheta)\right].
\end{align*}
The pseudo-innovations $\varepsilon_k^{(h)}(\vartheta)$ are constructed by the infinite past $\{Y^{(h)}(l):-\infty<l<k\}$. However, the infinite past is not known, we only have the finite observations
$Y_1^{(h)},\ldots,Y_n^{(h)}$. Therefore, we have to approximate the pseudo-innovations and the likelihood-function.
For a starting value $\widehat X_1^{(h)}(\vt)$, which is usually a deterministic constant, we define recursively based on \eqref{2.3} the approximate pseudo-innovations
as
\beao
     \widehat X^{(h)}_k(\vt)&=&(\e^{A_\vt h}-K_\vt^{(h)} C_\vt) \widehat X_{k-1}^{(h)}(\vt)+K_\vt^{(h)}Y_{k-1}^{(h)},\quad \\
     \widehat \varepsilon_{k}^{(h)}(\vt)&=&Y^{(h)}_k-C_\vt \widehat X_k^{(h)}(\vt),
 %       &=&Y _{\vt}^{(h)}(k)-C_\vt(\e^{A_\vt h}-K_\vt^{(h)} C_\vt)^{k-1}K_\vt^{(h)} X_\vt(1)- \sum_{i=1}^{k-1} C_\vt(\e^{A_\vt h}-K_\vt^{(h)} C_\vt)K_\vt^{(h)}
 %    Y_{\vt}^{(h)}(k-i).
\eeao
and the approximate likelihood-function as
\begin{align*}
%\label{eqLikelFunc}
\mathcal{\widehat{L}}_n^{(h)}(\vartheta)
= \frac{1}{n}\sum_{k=1}^n\left[d\log 2\pi +\log \det V_\vt^{(h)}+\widehat\varepsilon_k^{(h)}(\vartheta)^\mathsf{T}\big(V_\vt^{(h)}\big)^{-1}\widehat\varepsilon_k^{(h)}(\vartheta)\right].
\end{align*}
Then, the QML estimator %$\widehat{\vartheta}_n$ %. Hence, the QMLE based on the sample $(Y_1^{(h)},\ldots,Y_n^{(h)})$ gives
$$
%\label{eqLikelEsti}
\widehat{\vartheta}_n:=(\widehat{\vt}_{n,1}^{\mathsf{T}},\widehat{\vt}_{n,2}^{\mathsf{T}})^{\mathsf{T}}:
=\text{argmin}_{\vartheta\in\Theta}{\mathcal{\widehat{L}}_{n}^{(h)}}(\vartheta)
$$
is defined as the minimizer of the pseudo-Gaussian log-likelihood function $\mathcal{\widehat{L}}_{n}^{(h)}(\vartheta)$.
The estimator $\widehat{\vt}_{n,1}$ estimates the long-run parameter $\vt_1$ and the estimator $\widehat{\vt}_{n,2}$ estimates the short-run parameter $\vt_2$.
However, for our asymptotic results it does not matter if we use $\widehat{\mathcal{L}}_n^{(h)}(\vartheta)$ or ${\mathcal{L}}_n^{(h)}(\vartheta)$ as a conclusion of the next proposition. However, for that proposition to hold, we require the following \autoref{Ass:int} which assumes
a uniform bound on the second moment of the starting value $\widehat X_1^{(h)}(\vt)$ of the Kalman recursion and its partial derivatives.

\begin{assumptionletter}\label{Ass:int}
For every $u,v\in\{1,\ldots,s\}$  we assume that
\beao
  \E\left(\sup_{\vt\in\Theta}\|\widehat X_1^{(h)}(\vt)\|^2\right)<\infty, \quad  \E\left(\sup_{\vt\in\Theta}\|\partial_u\widehat X_1^{(h)}(\vt)\|^2\right)<\infty \quad \text{ and }\quad \E\left(\sup_{\vt\in\Theta}\|\partial_{u,v}\widehat X_1^{(h)}(\vt)\|^2\right)<\infty
\eeao
and $\widehat X_1^{(h)}(\vt)$ is independent of $(L_\vt(t))_{t\geq0}$.
\end{assumptionletter}

This assumption is not very restrictive, e.g., if $\widehat X_1^{(h)}(\vt)=\widehat X_1^{(h)}(\vt^0)$ for any $\vt\in\Theta$ and $\widehat X_1^{(h)}(\vt^0)$ is a deterministic vector, which we usually have in practice, \autoref{Ass:int}
is automatically satisfied.

\begin{proposition}  \label{Lemma 2.8}
Let \autoref{AssMBrownMot} and \ref{Ass:int} hold. Moreover, let $\gamma<1$ and $u,v\in\{1,\ldots,s\}$. Then, \begin{itemize}
    \item[(a)] $n^{\gamma}\sup_{\vt\in\Theta}|\mathcal{\widehat{L}}_n^{(h)}(\vartheta)-\mathcal{{L}}_n^{(h)}(\vartheta)|\cip 0$,
    \item[(b)] $n^{\gamma}\sup_{\vt\in\Theta}|\partial_u\mathcal{\widehat{L}}_n^{(h)}(\vartheta)-\partial_u\mathcal{{L}}_n^{(h)}(\vartheta)|\cip 0$,
    \item[(c)] $n^{\gamma}\sup_{\vt\in\Theta}|\partial_{u,v}\mathcal{\widehat{L}}_n^{(h)}(\vartheta)-\partial_{u,v}\mathcal{{L}}_n^{(h)}(\vartheta)|\cip 0$.
\end{itemize}
\end{proposition}
The proof of this proposition is similarly to the proof of Schlemm and Stelzer~\cite[Lemma 2.7 and Lemma 2.15]{SchlemmStelzer2012}. However, they are some essential differences since
in their paper $(Y_k^{(h)})_{k\in\N}$ and $(\varepsilon_k^{(h)}(\vt))_{k\in\N}$ are stationary sequences where in our setup they are non-stationary. Furthermore, we require
different convergence rates. A detailed proof can be found in \Cref{Appendix B}.

 We split now the pseudo-innovation sequence based on the decomposition  $\vartheta=(\vt_1^{\mathsf{T}},\vt_2^{\mathsf{T}})^{\mathsf{T}}$ so that
 one part is stationary and depends only on $\vt_2$:
\begin{align}
\label{eqInnovSequSeperated}
\varepsilon_k^{(h)}(\vartheta)=&~\varepsilon_{k,1}^{(h)}(\vartheta)+\varepsilon_{k,2}^{(h)}(\vartheta), \notag
\\\text{where }\qquad
\varepsilon_{k,1}^{(h)}(\vartheta):=&-\left[\Pi(\vartheta_1,\vartheta_2) -\Pi(\vartheta_1^0,\vartheta_2)\right]Y_{k-1}^{(h)}+\left[{\mathsf{k}}(\mathsf{B},\vartheta_1,\vartheta_2)-{\mathsf{k}}(\mathsf{B},\vartheta_1^0,\vartheta_2)\right]\Delta Y_k^{(h)}
\\\text{and }\qquad
\notag
\varepsilon_{k,2}^{(h)}(\vartheta):=&~\varepsilon_{k,2}^{(h)}(\vartheta_2)
=-\Pi(\vartheta_1^0,\vartheta_2)Y_{k-1}^{(h)}+{\mathsf{k}}(\mathsf{B},\vartheta_1^0,\vartheta_2)\Delta Y_k^{(h)}.
\end{align}
Due to similar calculations as in \eqref{vareps} we receive that $\Pi(\vartheta_1^0,\vartheta_2)Y_{k-1}^{(h)}=\Pi(\vartheta_1^0,\vartheta_2)Y_{st,k-1}^{(h)}$  and hence,
\beam \label{3.3b}
    \varepsilon_{k,2}^{(h)}(\vartheta_2)
=-\Pi(\vartheta_1^0,\vartheta_2)Y_{st,k-1}^{(h)}+{\mathsf{k}}(\mathsf{B},\vartheta_1^0,\vartheta_2)\Delta Y_k^{(h)}, \quad k\in\N,
\eeam
is indeed stationary. Moreover,  $\varepsilon_{k,1}^{(h)}(\vartheta_1^0,\vartheta_2)=0$ for any $\vartheta_2\in\Theta_2$ and $k\in\N$. \label{page7}
Finally, we separate the log-likelihood function $\mathcal{L}_n^{(h)}(\vt)$  in  \label{pageref1}
 \begin{align*}
%\label{eqLikelFuncSeperated}
\notag
\mathcal{L}_n^{(h)}(\vt)&=\mathcal{L}_{n,1}^{(h)}(\vartheta)+\mathcal{L}_{n,2}^{(h)}(\vartheta_2),
\intertext{where}
\notag
\mathcal{L}_{n,1}^{(h)}(\vartheta)&:=\mathcal{L}_{n}^{(h)}(\vartheta_1,\vartheta_2)-\mathcal{L}_{n}^{(h)}(\vartheta_1^0,\vartheta_2) \\ \notag
&=\log\det V_{\vt}^{(h)}-\log \det V_{\vt_1^0,\vt_2}^{(h)}+
\frac{1}{n}\sum_{k=1}^n \varepsilon_{k,1}^{(h)}(\vartheta)^\mathsf{T}\big(V_{\vt}^{(h)}\big)^{-1}\varepsilon_{k,1}^{(h)}(\vartheta)
\\\notag
&\quad+\frac{2}{n}\sum_{k=1}^n\varepsilon_{k,1}^{(h)}(\vartheta)^\mathsf{T}\big(V_{\vt}^{(h)}\big)^{-1}\varepsilon_{k,2}^{(h)}(\vartheta_2) + \frac{1}{n}\sum_{k=1}^n \varepsilon_{k,2}^{(h)}(\vartheta_2)^\mathsf{T}\big(V_{\vt}^{(h)}\big)^{-1}\varepsilon_{k,2}^{(h)}(\vartheta_2)
\\&\quad-\frac{1}{n}\sum_{k=1}^n \varepsilon_{k,2}^{(h)}(\vartheta_2)^\mathsf{T}\big(V_{\vt_1^0,\vt_2}^{(h)}\big)^{-1}\varepsilon_{k,2}^{(h)}(\vartheta_2),\\
%=:\frac{1}{n}\sum_{k=1}^n \ell_{k,1}^{(h)}(\vt)
%\label{eqLikelFuncSeperated:b}
\mathcal{L}_{n,2}^{(h)}(\vartheta_2)
&:=\mathcal{L}_{n}^{(h)}(\vartheta_1^0,\vartheta_2) =d\log 2\pi +\log \det V_{\vt_1^0,\vt_2}^{(h)}+\frac{1}{n}\sum_{k=1}^n\varepsilon_{k,2}^{(h)}(\vartheta_2)^\mathsf{T} \big(V_{\vt_1^0,\vt_2}^{(h)}\big)^{-1}\varepsilon_{k,2}^{(h)}(\vartheta_2).
%=:\frac{1}{n}\sum_{k=1}^n \ell_{k,2}^{(h)}(\vt_2).
%\label{eqLikelFuncSeperated:c}
\end{align*}
Obviously, $\mathcal{L}_{n,2}^{(h)}(\vartheta_2)$ depends only on the short-run parameters, whereas $\mathcal{L}_{n,1}^{(h)}(\vartheta)$ depends on all parameters. Furthermore, we have the following relations:
\beam \label{3.6a}
 \mathcal{L}_{n,1}^{(h)}(\vartheta_1^0,\vartheta_2)=0~\quad\text{ and }~\quad \mathcal{L}_{n}^{(h)}(\vartheta_1^0,\vartheta_2)=\mathcal{L}_{n,2}^{(h)}(\vartheta_2) \quad ~\text{ for any }~\vt_2\in\Theta_2.
\eeam
This immediately implies $\mathcal{L}_{n}^{(h)}(\vartheta^0)=\mathcal{L}_{n,2}^{(h)}(\vartheta^0_2)$.
In the remaining of the paper, we will see that the asymptotic properties of $\widehat \vt_{n,1}$  are determined by
$\mathcal{L}_{n,1}^{(h)}(\vartheta)$, whereas the asymptotic properties of  $\widehat \vt_{n,2}$  are completely determined by
$\mathcal{L}_{n,2}^{(h)}(\vartheta_2)$. Since $\mathcal{L}_{n,2}^{(h)}(\vartheta_2)$ is based only on stationary processes it is not surprising that
 $\widehat \vt_{n,2}$ exhibits the same asymptotic properties as QML estimators for stationary processes.

%
%In the following we write shortly $\varepsilon_k$ for $\varepsilon_k(\vartheta^0)$ whenever we insert the true parameter $\vartheta^0$.

%% ==============================
\subsection{Identifiability}
\label{Section:3:subsection:2}
%% ==============================

In order to properly estimate our model, we need a unique minimum of the likelihood function and therefore we need some identifiability criteria
for the family of stochastic processes $(Y_\vartheta,\vartheta\in\Theta)$. % A challenge
%is that we do not observe the process continuously instead we observe it at an equidistant time-grid and hence,  we have
%   to avoid the aliasing effect.
The first assumption guarantees the uniqueness of the long-run parameter $\vartheta_1^0$.

\begin{assumptionletter}
\label{AssMUniquePosTriForm} \label{AssMPosDefN1}
%For any $\vartheta_1\in\Theta_1$ the following conditions hold:
 There exists a constant $\mathfrak{C}^*>0$ so that $\|C_{1,\vartheta_1}^{\perp\mathsf{T}}C_1\|\geq \mathfrak{C}^*\|\vt_1-\vt_1^0\|$\, for $\vt\in\Theta$.
\end{assumptionletter}

\begin{remark} \label{Remark:id} $\mbox{}$
\begin{enumerate}
   % \item Assumption (B1) guarantees that the cointegration space  has a unique representation for any $\vt\in\Theta$.
          \item[(i)] Without \autoref{AssMUniquePosTriForm} we have only that $\|C_{1,\vartheta_1}^{\perp\mathsf{T}}C_1\|$ has a zero in $\vt_1^0$ but not  that $\|C_{1,\vartheta_1}^{\perp\mathsf{T}}C_1\|\not=0$
            for $\vt_1\not=\vt_1^0$.
     In particular, $\|C_{1,\vartheta_1}^{\perp\mathsf{T}}C_1\|\not=0$ for $\vt_1\not=\vt_1^0$
           implies that the space spanned by $C_1$ and $C_{1,\vartheta_1}$ are not the same.
      \item[(ii)] Due to the Lipschitz-continuity of $C_{1,\vartheta_1}^{\perp\mathsf{T}}$  and $C_{1,\vartheta_1^0}^{\perp\mathsf{T}}C_1=0_{(d-c)\times c}$ the upper bound $\|C_{1,\vartheta_1}^{\perp\mathsf{T}}C_1\|\leq \mathfrak{C}\|\vt_1-\vt_1^0\|$ for some constant $\mathfrak{C}>0$ is valid as well.
    \item[(iii)] Note that \autoref{AssMUniquePosTriForm} implies that $\Vert \Pi(\vartheta) C_1B_1\Vert=\Vert \alpha(\vartheta)C_{1,\vartheta_1}^{\perp\mathsf{T}} C_1B_1\Vert >0$ for $\vt_1^0\not=\vt_1$ since  $\alpha(\vt)$ and $B_1$ have full rank,
%As a consequence of \autoref{AssMUniquePosTriForm} the span of $\beta(\vt_1^0)$ matches the cointegration space
and thus, the process $\big(\varepsilon_{k,1}^{(h)}(\vartheta)\big)_{k\in\N}$ is indeed non-stationary for all long-run parameters $\vartheta_1\not=\vt_1^0$.
    \item[(iv)] The matrix function $\alpha(\vt)$ is continuous and has full column rank $d-c$ so that necessarily $\inf_{\vt\in\Theta}\sigma_{\min}(\alpha(\vt))>0$.
    Applying Bernstein~\cite[Corollary 9.6.7]{Bernstein2009}  gives for some constant $\mathfrak{C}>0$:
    \beao
        \|\Pi(\vt)C_1\|\geq \inf_{\vt\in\Theta}\left\{\sigma_{\min}(\alpha(\vt))\right\}\|C_{1,\vartheta_1}^{\perp\mathsf{T}}C_1\|\geq \mathfrak{C}\|\vt_1-\vt_1^0\|.
    \eeao
\end{enumerate}
\end{remark}

The next assumption guarantees the uniqueness of the short-run parameter $\vartheta_2^0$.

\begin{assumptionletter}
\label{Identifiability} \label{AssMIdentSpectDens}
For any $\vartheta_2^0\not=\vartheta_2\in\Theta_2$ there exists a $z\in\C$ such that either
\begin{align*}
C_{\vt_1^0,\vt_2}\left[I_N-\big(\e^{A_{\vt_2} h}-K_{\vt_1^0,\vt_2}^{(h)} C_{\vt_1^0,\vt_2}\big)z\right]^{-1}K_{\vt_1^0,\vt_2}^{(h)}\not=&~C\left[I_N-\big(\e^{A h}-K^{(h)} C\big)z\right]^{-1}K^{(h)}
 ~\text{ or }~
 V_{\vt_1^0,\vt_2}^{(h)}\not=~V^{(h)}.
 \end{align*}
\end{assumptionletter}

\begin{lemma}
\label{LemUniqueMin}
Let \autoref{AssMBrownMot}  and \ref{Identifiability} hold. The function $\mathbfcal{L}_{2}^{(h)}:\Theta_2\to\R$ defined by
\begin{align}
\label{l2h}
\mathbfcal{L}_{2}^{(h)}(\vartheta_2):=d\log(2\pi)+\log\det V_{\vt_1^0,\vt_2}^{(h)}+\E\left( \varepsilon_{1,2}^{(h)}(\vartheta_2)^\mathsf{T}\big(V_{\vt_1^0,\vt_2}^{(h)}\big)^{-1} \varepsilon_{1,2}^{(h)}(\vartheta_2)\right)
\end{align} has a unique global minimum at $\vt_2^0$.
\end{lemma}
\begin{proof}
The proof is analogous to the proof of Lemma 2.10 in Schlemm and Stelzer \cite{SchlemmStelzer2012a}.
\end{proof}

%\begin{remark}
 Without the additional  \autoref{Identifiability} we obtain only that $\mathbfcal{L}_{2}^{(h)}(\vartheta_2)$ has  a minimum in $\vt_2^0$
    but not that the minimum is unique. %Under special conditions on $C_{1,\vt_1}$ and $B_{1,\vt_2}$ and under the assumption that the family of stochastic processes
%$(Y_{\vt_2,st}:\vt_2\in\Theta_2)$ is identifiable from its spectral density we receive a model
%that satisfies \autoref{AssMPosDefN1} and \ref{Identifiability}; details can be found in Fasen-Hartmann and Scholz \cite{FasenScholz17}.
 %   \item[(b)] If ??? is given in canonical form and $B_{2,\vt_2}\not=B_{2,\vt_2^0}$ for $\vt_2\not=\vt_2^0$ then \autoref{Identifiability} holds. ???? \marginpar{überarbeiten}
%\end{itemize}
%\end{remark}

Due to Fasen-Hartmann and Scholz~\cite[Theorem 3.1]{FasenScholz1} a canonical form for cointegrated state space processes already exists and can be used to construct
a model class satisfying \autoref{AssMPosDefN1} and \autoref{AssMIdentSpectDens}. Further details are presented in Fasen-Hartmann and Scholz~\cite{FasenScholz17}.
Moreover, criteria to overcome the aliasing effect
(see Blevins~\cite{Blevins2017}, Hansen and Sargent~\cite{HansenSargent1983},  McCrorie~\cite{McCrorie2003,McCrorie2009}, Phillips~\cite{Phillips1973,Phillips1991},
Schlemm and Stelzer~\cite{SchlemmStelzer2012a}) are given there.

%% ==============================
\section{Consistency of the QML estimator}
\label{sec:4}
%% ==============================

\label{pageref2}
 In order to show the consistency of the QML estimator, we follow the ideas of Saikkonen \cite{Saikkonen1995} in the simple regression model. Thus, we prove the consistency in three steps. In the first step, we prove the consistency of the long-run QML estimator $\widehat{\vt}_{n,1}$ and next we determine its consistency rate. Thirdly, we prove the consistency of the short-run QML estimator $\widehat{\vt}_{n,2}$  by making use of the consistency rate of the long-run QML estimator.
Throughout the rest of this paper, we assume that \autoref{AssMBrownMot} - \ref{AssMIdentSpectDens} always hold. Furthermore, we denote by
 $(W(r))_{0\leq r\leq 1}=((W_1(r)^{\mathsf{T}},W_2(r)^{\mathsf{T}},W_3(r)^{\mathsf{T}})^{\mathsf{T}})_{0\leq r\leq 1}$  a $(2d+m)$-dimensional Brownian motion  with
covariance matrix
\begin{eqnarray} \label{Cov B}
    \Sigma_W=\psi(1)\int_0^h\left(\begin{array}{cc}
        \Sigma_L &  \Sigma_L\e^{A_2^Tu} \\
        \e^{A_2u}B_2\Sigma_L & \e^{A_2u}B_2\Sigma_L B_2^T\e^{A_2^Tu}
    \end{array}\right)\,du\psi(1)^{\mathsf{T}},
\end{eqnarray}
where
\beam \label{def psi}
       && \psi_0:= \left( \begin{array}{cc}
        C_1B_1 & C_2\\
        0_{d\times m}      & C_2\\
        I_{m\times m} & 0_{m\times N-c}
    \end{array}
    \right),  \quad\quad
    \psi_j= \left( \begin{array}{cc}
      0_{d\times m}   & C_2(\e^{A_2hj}-\e^{A_2h(j-1)})\\
        0_{d\times m}      & C_2\e^{A_2hj}\\
        I_{m\times m} & 0_{m\times N-c}
    \end{array}
    \right), \quad j\geq 1, \quad
    \text{  } \notag\\
    &&  \psi(z)=\sum_{j=0}^\infty\psi_jz^j, \quad z\in\C,
\eeam
 $(W_i(r))_{0\leq r\leq 1}$, $i=1,2$,
are $d$-dimensional Brownian motions and $(W_3(r))_{0\leq r\leq 1}$ is an $m$-dimensional Brownian motion. \label{Brownian motion}

\subsection{Consistency of the long-run QML estimator} \label{sec:4.1}
% We begin with the QMLE $\widehat{\vartheta}_{n,1}$ of the long-run parameter vector $\vt_1^0$. As in Saikkonen \cite{Saikkonen1995}, we want to show that $\widehat{\vt}_{n,1}-\vt_1^0=o_p(n^{-\gamma})$ holds for all $0\leq \gamma<1$.  For this purpose,
To show the consistency for the long-run parameter, Saikkonen~\cite[p. 903]{Saikkonen1995} suggests in his example that it is sufficient to show the following theorem, where
 $\mathcal{B}(\vartheta_1^0,\delta):= \{\vartheta_1\in\Theta_1:\|\vartheta_1-\vartheta_1^0\|\leq\delta\}$ denotes
 the closed ball with radius $\delta$ around $\vt_1^0$, and
$\overline{\mathcal{B}}(\vartheta_{1}^0,\delta):=\Theta_1\backslash {\mathcal{B}}(\vartheta_{1}^0,\delta)$ denotes its complement.

\begin{theorem}
\label{eqSuffCondforConsa}
 For any $\delta>0$ we have
\begin{align*}
\lim_{n\to\infty}\P\bigg(\inf_{\vartheta\in\overline{\mathcal{B}}(\vartheta_{1}^0,\delta)\times\Theta_2}
\hspace{-3pt}
\mathcal{\widehat{L}}_n^{(h)}(\vartheta)-\mathcal{\widehat{L}}_n^{(h)}(\vartheta^0)>0\bigg)=1.
\end{align*}
\end{theorem}

\begin{corollary}
In particular, $%\label{eqconsistency:a}
\widehat{\vt}_{n,1}-\vt_1^0=o_p(1)$.
\end{corollary}

\subsubsection{Proof of \Cref{eqSuffCondforConsa}}

The following lemmata are important for the proof of the theorem.

\begin{lemma} \label{Lemma 3.5a}
 Let $L^{(h)}:=(L_{k}^{(h)})_{k\in\Z}:=(L(kh))_{k\in\Z}$ and define
\begin{eqnarray*}
    \mathcal{L}_{n,1,1}^{(h)}(\vartheta)&:=&\frac{1}{n}\sum_{k=1}^n\left[\Pi(\vartheta)C_1B_1L_{k-1}^{(h)}\right]^\mathsf{T}\big(V_{\vt}^{(h)}\big)^{-1}
     \Pi(\vartheta)C_1B_1L_{k-1}^{(h)},\\
    \mathcal{L}_{n,1,2}^{(h)}(\vartheta)&:=&\mathcal{L}_{n,1}^{(h)}(\vartheta)-\mathcal{L}_{n,1,1}^{(h)}(\vartheta).
\end{eqnarray*}
Then,  $ |\mathcal{L}_{n,1,2}^{(h)}(\vartheta)|\leq \mathfrak{C}\|\vt_1-\vt_1^0\| U_n $ for $ \vt\in\Theta$
with  $U_n=1+V_n+Q_n=\mathcal{O}_p(1)$, and $V_n$ and $Q_n$ are defined as in Proposition~\ref{Proposition3.3}.
\end{lemma}

To conclude $\mathcal{L}_{n,1,2}^{(h)}(\cdot,\vartheta_2)$ is local Lipschitz continuous in $\vartheta_1^0$.

\begin{proof}[Proof of Lemma~\ref{Lemma 3.5a}]
Define
\begin{eqnarray*}
    \varepsilon_{k,1,1}^{(h)}(\vartheta)&:=&-\left[\Pi(\vartheta_1,\vartheta_2) -\Pi(\vartheta_1^0,\vartheta_2)\right]C_1B_1L_{k-1}^{(h)}=-\Pi(\vartheta_1,\vartheta_2)C_1B_1L_{k-1}^{(h)},\\
    \varepsilon_{k,1,2}^{(h)}(\vartheta)&:=&\varepsilon_{k,1}^{(h)}(\vartheta)-\varepsilon_{k,1,1}^{(h)}(\vartheta)\\
        &=&-\left[\Pi(\vartheta_1,\vartheta_2) -\Pi(\vartheta_1^0,\vartheta_2)\right]Y_{st,k-1}^{(h)}+\left[{\mathsf{k}}(\mathsf{B},\vartheta_1,\vartheta_2)-{\mathsf{k}}(\mathsf{B},\vartheta_1^0,\vartheta_2)\right]\Delta Y_k^{(h)}.
\end{eqnarray*}
Then, $(\varepsilon_{k,1,2}^{(h)}(\vartheta))_{k\in\N}$ is a stationary sequence and $\varepsilon_{k,1}^{(h)}(\vartheta)=\varepsilon_{k,1,1}^{(h)}(\vartheta)+\varepsilon_{k,1,2}^{(h)}(\vartheta)$.
First, note that
\begin{eqnarray*}
\mathcal{L}_{n,1,2}^{(h)}(\vartheta)&=&\log\det V_{\vt_1,\vt_2}^{(h)}-\log \det V_{\vt_1^0,\vt_2}^{(h)}
  + \frac{2}{n}\sum_{k=1}^n\varepsilon_{k,1,1}^{(h)}(\vartheta)^\mathsf{T}\big(V_{\vt}^{(h)}\big)^{-1}[\varepsilon_{k,2}^{(h)}(\vartheta_2)
 +\varepsilon_{k,1,2}^{(h)}(\vartheta)]\\
&& + \frac{1}{n}\sum_{k=1}^n\varepsilon_{k,1,2}^{(h)}(\vartheta)^\mathsf{T}\big(V_{\vt}^{(h)}\big)^{-1}[2\varepsilon_{k,2}^{(h)}(\vartheta_2)
 + \varepsilon_{k,1,2}^{(h)}(\vartheta)]\\
 &&+ \frac{1}{n}\sum_{k=1}^n \varepsilon_{k,2}^{(h)}(\vartheta_2)^\mathsf{T}\left(\big(V_{\vt}^{(h)}\big)^{-1}-\big(V_{\vt_1^0,\vt_2}^{(h)}\big)^{-1}\right)\varepsilon_{k,2}^{(h)}(\vartheta_2).
\end{eqnarray*}
In the following, we use Bernstein \cite[(2.2.27) and Corollary 9.3.9]{Bernstein2009}  to get the upper bound
\begin{eqnarray*}
\lefteqn{\left|\frac{1}{n}\sum_{k=1}^n\varepsilon_{k,1,1}^{(h)}(\vartheta)^\mathsf{T}\big(V_{\vt}^{(h)}\big)^{-1}[\varepsilon_{k,2}^{(h)}(\vartheta)+\varepsilon_{k,1,2}^{(h)}(\vartheta)]\right|}\\
    &&=\left|\frac{1}{n}\sum_{k=1}^n\tr\left(\varepsilon_{k,1,1}^{(h)}(\vartheta)^\mathsf{T}\big(V_{\vt}^{(h)}\big)^{-1}[\varepsilon_{k,2}^{(h)}(\vartheta)+\varepsilon_{k,1,2}^{(h)}(\vartheta)]\right)\right|\\
    &&=\left|\tr\left(\big(V_{\vt}^{(h)}\big)^{-1}\frac{1}{n}\sum_{k=1}^n[\varepsilon_{k,2}^{(h)}(\vartheta)+\varepsilon_{k,1,2}^{(h)}(\vartheta)]\varepsilon_{k,1,1}^{(h)}(\vartheta)^\mathsf{T}\right)\right|\\
    &&\leq\|(V_{\vt}^{(h)}\big)^{-1}\|\left\|\frac{1}{n}\sum_{k=1}^n[\varepsilon_{k,2}^{(h)}(\vartheta)+\varepsilon_{k,1,2}^{(h)}(\vartheta)]\varepsilon_{k,1,1}^{(h)}(\vartheta)^\mathsf{T}\right\|.
\end{eqnarray*}
Similarly we find upper bounds for the other terms.
Moreover, due to Lemma~\ref{Lemma 2.3}~(b)
\begin{eqnarray*}
|\mathcal{L}_{n,1,2}^{(h)}(\vartheta)|&\leq&|\log\det V_{\vt_1,\vt_2}^{(h)}-\log \det V_{\vt_1^0,\vt_2}^{(h)}|
    +\mathfrak{C}\left\|\frac{1}{n}\sum_{k=1}^n[\varepsilon_{k,2}^{(h)}(\vartheta)+\varepsilon_{k,1,2}^{(h)}(\vartheta)]\varepsilon_{k,1,1}^{(h)}(\vartheta)^\mathsf{T}\right\|\\
&&  + \mathfrak{C}\left\|\frac{1}{n}\sum_{k=1}^n[2\varepsilon_{k,2}^{(h)}(\vartheta)+\varepsilon_{k,1,2}^{(h)}(\vartheta)]\varepsilon_{k,1,2}^{(h)}(\vartheta)^\mathsf{T}\right\|\\
  &&+ \mathfrak{C}\left\|\big(V_{\vt_1,\vt_2}^{(h)}\big)^{-1}-\big(V_{\vt_1^0,\vt_2}^{(h)}\big)^{-1}\right\|\left\|\frac{1}{n}\sum_{k=1}^n \varepsilon_{k,2}^{(h)}(\vartheta_2)\varepsilon_{k,2}^{(h)}(\vartheta_2)^{\mathsf{T}}\right\|.
\end{eqnarray*}
Since $V_\vt^{-1}$ and $\log \det V_\vt$ are Lipschitz continuous by Lemma~\ref{Lemma 2.3}(a),
we obtain
\begin{eqnarray} \label{C.7}
|\mathcal{L}_{n,1,2}^{(h)}(\vartheta)|&\leq&\mathfrak{C}\left(\|\vt_1-\vt_1^0\|  +\left\|\frac{1}{n}\sum_{k=1}^n[\varepsilon_{k,2}^{(h)}(\vartheta)+\varepsilon_{k,1,2}^{(h)}(\vartheta)]\varepsilon_{k,1,1}^{(h)}(\vartheta)^\mathsf{T}\right\|\right.\\
&& \hspace*{-4cm}\left. + \left\|\frac{1}{n}\sum_{k=1}^n[2\varepsilon_{k,2}^{(h)}(\vartheta)+\varepsilon_{k,1,2}^{(h)}(\vartheta)]\varepsilon_{k,1,2}^{(h)}(\vartheta)^\mathsf{T}\right\|
+ \|\vt_1-\vt_1^0\|\left\|\frac{1}{n}\sum_{k=1}^n \varepsilon_{k,2}^{(h)}(\vartheta_2)\varepsilon_{k,2}^{(h)}(\vartheta_2)^{\mathsf{T}}\right\|\right).\nonumber
\end{eqnarray}
Moreover, $\Pi(\vt)$ is Lipschitz continuous as well (see  Lemma~\ref{Lemma 2.3}(a)) and
the sequence of matrix functions  $(\mathsf{k}_j(\vt))_{j\in\N}$ and $(\nabla_\vt \mathsf{k}_j(\vt))_{j\in\N}$ are exponentially bounded (see Lemma~\ref{LemInnovaSeqProp} and Lemma~\ref{LemDerInnovaSeqProp}).
Due to \eqref{EQ3} and $\varepsilon_{k,1,1}^{(h)}(\vartheta_1^0,\vt_2)=0$ we receive
\beam
    \left\|\frac{1}{n}\sum_{k=1}^n[\varepsilon_{k,2}^{(h)}(\vartheta)+\varepsilon_{k,1,2}^{(h)}(\vartheta)]\varepsilon_{k,1,1}^{(h)}(\vartheta)^\mathsf{T}\right\|\leq \mathfrak{C}\|\vt_1-\vt_1^0\|V_n.
\eeam
Due to \eqref{EQ5} and $\varepsilon_{k,1,2}^{(h)}(\vartheta_1^0,\vt_2)=0$ we get
\beam
    \left\|\frac{1}{n}\sum_{k=1}^n[2\varepsilon_{k,2}^{(h)}(\vartheta)+\varepsilon_{k,1,2}^{(h)}(\vartheta)]\varepsilon_{k,1,2}^{(h)}(\vartheta)^\mathsf{T}\right\|
        \leq \mathfrak{C}\|\vt_1-\vt_1^0\|Q_n.
\eeam
Finally,
\beam \label{C.9}
    \left\|\frac{1}{n}\sum_{k=1}^n \varepsilon_{k,2}^{(h)}(\vartheta_2)\varepsilon_{k,2}^{(h)}(\vartheta_2)^{\mathsf{T}}\right\|\leq \mathfrak{C} Q_n
\eeam
as well. Then, \eqref{C.7}-\eqref{C.9} result in the upper bound $|\mathcal{L}_{n,1,2}^{(h)}(\vartheta)|\leq \mathfrak{C}\|\vt_1-\vt_1^0\|(1+V_n+Q_n)$.
A direct consequence of Proposition~\ref{Proposition3.3} is $U_n=1+V_n+Q_n=\mathcal{O}_p(1)$.
\end{proof}

\begin{lemma}
\label{propConvEps2} $\mbox{}$
%Assume that \autoref{AssMBrownMot} and \ref{AssMContDiffPi}  hold.
\begin{itemize}
\item[(a)] %Let $\big(\mathbfcal{L}_{2}^{(h)}(\vartheta_2)\big)_{\vt_2\in\Theta_2}$ be  given as in \eqref{l2h}. Then,
$
\sup_{\vt_2\in\Theta_2}|\mathcal{L}_{n,2}^{(h)}(\vartheta_2)- \mathbfcal{L}_{2}^{(h)}(\vartheta_2)|\ccip 0 $   as  $n\to\infty
$.
\item[(b)]
$
    \frac{1}{n}\mathcal{L}_{n,1}^{(h)}(\vartheta)    \ccid %\int_0^1 (\Pi(\vartheta)C_1W(r))^{\mathsf{T}}(V_\vt^{(h)})^{-1}(\Pi(\vartheta)C_1W(r))\,\dif r \\
     %&=&\tr\left((V_\vt^{(h)})^{-1}(\Pi(\vartheta)C_1)\int_0^1 W(r)W(r)^{\mathsf{T}}\,\dif r (\Pi(\vartheta)C_1)^{\mathsf{T}}\right)\,\\
     \int_0^1 \|(V_\vt^{(h)})^{-1/2}\Pi(\vartheta)C_1B_1W_3(r)\|^2\,\dif r
$
and the convergence holds in the space of continuous functions on $\Theta$ with the supremum norm.
\end{itemize}
\end{lemma}
\begin{proof} $\mbox{}$\\
(a) \,
%Note that    $\big\Vert(V_{\vt_1^0,\vt_2}^{(h)}\big)^{-1}\big\Vert \leq c$  on $\Theta$ by \cite[Lemma 5.9.2]{Scholz} and $\big(V_{\vt_1^0,\vt_2}^{(h)}\big)^{-1}$
%is continuous by  \cite[Lemma 5.9.3]{Scholz}.  Moreover, $\big(\varepsilon_{k,2}^{(h)}(\vartheta_2)\big)$ given in \eqref{eqInnovSequSeperated} has uniformly
%exponentially bounded matrix coefficients (see \Cref{LemInnovaSeqProp}).
%Hence, due to the representation of $\mathcal{L}_{n,2}^{(h)}(\vartheta_2)$
is a consequence of  Proposition~\ref{PropUniformConvRes1}(a) and the continuous mapping theorem.  \\[2mm]
(b) \, First, $\sup_{\vt\in\Theta}|\frac{1}{n}\mathcal{L}_{n,1,2}(\vt)|=o_p(1)$ due to Lemma~\ref{Lemma 3.5a} and $\Theta$ compact.
Second, a   conclusion of Proposition~\ref{PropUniformConvRes1}(b) and the continuous mapping theorem is that
\beao
    \frac{1}{n}\mathcal{L}_{n,1,1}^{(h)}(\vartheta)&=&\tr\left(\big(V_{\vt}^{(h)}\big)^{-1/2}\Pi(\vartheta)C_1B_1\left(\frac{1}{n^2}\sum_{k=1}^n
    L_{k-1}^{(h)}L_{k-1}^{(h)^\mathsf{T}}\right)B_1^\mathsf{T}C_1^\mathsf{T}\Pi(\vartheta)^\mathsf{T}\big(V_{\vt}^{(h)}\big)^{-1/2}\right) \\
       &\ccid& \tr\left(\big(V_{\vt}^{(h)}\big)^{-1/2}\Pi(\vartheta)C_1B_1\left(\int_0^1W_3(r)W_3(r)^\mathsf{T}\,dr\right) B_1^\mathsf{T}C_1^\mathsf{T}\Pi(\vartheta)^\mathsf{T}\big(V_{\vt}^{(h)}\big)^{-1/2}\right)\\ %\int_0^1 (\Pi(\vartheta)C_1W(r))^{\mathsf{T}}(V_\vt^{(h)})^{-1}(\Pi(\vartheta)C_1W(r))\,\dif r \\
     %&=&\tr\left((V_\vt^{(h)})^{-1}(\Pi(\vartheta)C_1)\int_0^1 W(r)W(r)^{\mathsf{T}}\,\dif r (\Pi(\vartheta)C_1)^{\mathsf{T}}\right)\,\\
     &=&\int_0^1 \|(V_\vt^{(h)})^{-1/2}\Pi(\vartheta)C_1B_1W_3(r)\|^2\,\dif r,
\eeao
and the convergence holds in the space of continuous functions on $\Theta$ with the supremum norm due to the continuity of $\Pi(\vt)$ and
$(V_\vt^{(h)})^{-1}$ (see Lemma~\ref{Lemma 2.3}(a)).
In the first and in the last equality we used  Bernstein \cite[2.2.27]{Bernstein2009} which allows us to permutate matrices in the trace.
\end{proof}

\begin{proof}[Proof of \Cref{eqSuffCondforConsa}]
On the one hand,
 due to Proposition~\ref{Lemma 2.8}
 \beao
 \lefteqn{\inf_{\vartheta\in\overline{\mathcal{B}}(\vartheta_{1}^0,\delta)\times\Theta_2}
        \mathcal{\widehat{L}}_n^{(h)}(\vartheta)-\mathcal{\widehat{L}}_n^{(h)}(\vartheta^0) }\\
        &&\geq \inf_{\vartheta\in\overline{\mathcal{B}}(\vartheta_{1}^0,\delta)\times\Theta_2}
        (\mathcal{{L}}_n^{(h)}(\vartheta)-\mathcal{{L}}_n^{(h)}(\vartheta^0))
         -2\sup_{\vt\in\Theta}|\mathcal{\widehat{L}}_n^{(h)}(\vartheta)-\mathcal{{L}}_n^{(h)}(\vartheta)|\\
         &&=\inf_{\vartheta\in\overline{\mathcal{B}}(\vartheta_{1}^0,\delta)\times\Theta_2}
        \left(\mathcal{{L}}_n^{(h)}(\vartheta)-\mathcal{{L}}_n^{(h)}(\vartheta^0)\right)+o_p(1).
 \eeao
 On the other hand,  due to Lemma~\ref{LemUniqueMin} and  Lemma~\ref{propConvEps2}(a)
\beao
    \lefteqn{\left|\inf_{\vartheta_2\in\Theta_2}\mathcal{L}_{n,2}^{(h)}(\vartheta_2) -\mathcal{L}_{n,2}^{(h)}(\vartheta_2^0)\right|}\\
    &\leq& \sup_{\vt_2\in\Theta_2}|\mathcal{L}_{n,2}^{(h)}(\vartheta_2)-\mathbfcal{L}_{2}^{(h)}(\vartheta_2)|+
    \left|\inf_{\vt_2\in\Theta_2}\mathbfcal{L}_{2}^{(h)}(\vartheta_2)-\mathbfcal{L}_{2}^{(h)}(\vartheta_2^0)\right|
    +|\mathbfcal{L}_{2}^{(h)}(\vartheta_2^0)-\mathcal{L}_{n,2}^{(h)}(\vartheta_2^0)|
    =o_p(1).
\eeao
 Using \eqref{3.6a} and the above results we receive
\begin{align*}
\inf_{\vartheta\in\overline{\mathcal{B}}(\vartheta_{1}^0,\delta)\times\Theta_2}
\hspace{-3pt}
\mathcal{\widehat{L}}_n^{(h)}(\vartheta)-\mathcal{\widehat{L}}_n^{(h)}(\vartheta^0) &\geq \inf_{\vartheta\in\overline{\mathcal{B}}(\vartheta_{1}^0,\delta)\times\Theta_2}
        \left(\mathcal{{L}}_n^{(h)}(\vartheta)-\mathcal{{L}}_n^{(h)}(\vartheta^0)\right)+o_p(1)\\
&\geq\inf_{\vartheta\in\overline{\mathcal{B}}(\vartheta_{1}^0,\delta)\times\Theta_2}
\hspace{-3pt}
\mathcal{L}_{n,1}^{(h)}(\vartheta)
+\inf_{\vartheta\in\Theta_2}\big(\mathcal{L}_{n,2}^{(h)}(\vartheta_2)-\mathcal{L}_{n,2}^{(h)}(\vartheta_2^0)\big)+o_p(1)
\notag\\
&=\inf_{\vartheta\in\overline{\mathcal{B}}(\vartheta_{1}^0,\delta)\times\Theta_2}
\hspace{-3pt}
\mathcal{L}_{n,1}^{(h)}(\vartheta)+o_p(1).
\end{align*}
Hence,  it suffices to show that for any $\tau>0$
\begin{eqnarray}  \label{eqQgeqc}
\limn \P\bigg(\inf_{\vartheta\in\overline{\mathcal{B}}(\vartheta_{1}^0,\delta)\times\Theta_2}
\hspace{-3pt}
\mathcal{L}_{n,1}^{(h)}(\vartheta)> \tau\bigg)=1.
\end{eqnarray}
%We use the decomposition $\mathcal{L}_{n,1}^{(h)}(\vartheta)=\mathcal{L}_{n,1,1}^{(h)}(\vartheta)+\mathcal{L}_{n,1,2}^{(h)}(\vartheta)$ as given in \Cref{Lemma 3.5}.
An application of Lemma~\ref{propConvEps2}(b) and the continuous mapping theorem yield
\begin{eqnarray} \label{C.1}
    \inf_{\vartheta\in\overline{\mathcal{B}}(\vartheta_{1}^0,\delta)\times\Theta_2}\frac{1}{n}\mathcal{L}_{n,1}^{(h)}(\vartheta) &\xrightarrow[]{ ~w~ }&
   % \inf_{\vartheta\in\overline{\mathcal{B}}(\vartheta_{1}^0,\delta)\times\Theta_2}\int_0^1 \|(V_\vt^{(h)})^{-1/2}\Pi(\vartheta)W_1(r)\|^2\,\dif r \notag\\
    %&\stackrel{d}{=}&
    \inf_{\vartheta\in\overline{\mathcal{B}}(\vartheta_{1}^0,\delta)\times\Theta_2}\int_0^1 \|(V_\vt^{(h)})^{-1/2}\Pi(\vartheta)C_1B_1W_3(r)\|^2\,\dif r.
\end{eqnarray}
%where $(W_3(r))_{0\leq r\leq 1}$ is a Brownian motion with covariance matrix $\Sigma_L$.
%In the following we find an upper bound for the integrand. Due to Bernstein \cite[Corollary 9.6.7]{Bernstein2009}
%\begin{eqnarray} \label{C.2}
%    |(V_\vt^{(h)})^{-1/2}\Pi(\vartheta)C_1W(r)\|^2\geq \sigma_{\min}(V_\vt^{(h)})^{-1})\|\Pi(\vartheta)C_1W(r)\|^2.
%\end{eqnarray}
% The random matrix $\int_0^1W(r)W(r)^T\dif r$ is $\P$-a.s. positive definite. Hence, there exists a $c\times c$-dimensional symmetric positive random matrix $W_*$
%with $\int_0^1W(r)W(r)^T\dif r=W_*W_*^{\mathsf{T}}$. Then, %\begin{eqnarray*}
%        \int_0^1 \|(\Pi(\vartheta)C_1W(r)\|^2\,\dif r&=&
%        \int_0^1 \tr(W(r)^{\mathsf{T}}C_1^{\mathsf{T}}\Pi(\vt)^{\mathsf{T}}\Pi(\vartheta)C_1W(r))\dif r\\
%        &=&\int_0^1 \tr(C_1^{\mathsf{T}}\Pi(\vt)^{\mathsf{T}}\Pi(\vartheta)\Pi(\vartheta)C_1W(r)W(r)^{\mathsf{T}})\dif r\\
%        &=&\tr(W_*^{\mathsf{T}}C_1^{\mathsf{T}}\Pi(\vt)^{\mathsf{T}}\Pi(\vartheta)C_1W_*)\\
%        &=&\|\Pi(\vt)^{\mathsf{T}}C_1W_*\|^2
%\end{eqnarray*}
%where we used Bernstein \cite[2.2.25]{Bernstein2009} which allows us to permutate matrices in the trace.
%Again an application of  Bernstein \cite[Corollary 9.6.7]{Bernstein2009} and \eqref{C.2} yields
Due to Bernstein \cite[Corollary 9.6.7]{Bernstein2009}
\begin{eqnarray} \label{C.2}
   \int_0^1\|(V_\vt^{(h)})^{-1/2}\Pi(\vartheta)C_1B_1W_3(r)\|^2\,dr \geq \sigma_{\min}((V_\vt^{(h)})^{-1})\int_0^1\|\Pi(\vartheta)C_1B_1W_3(r)\|^2\,dr.
\end{eqnarray}
Moreover,
\begin{eqnarray*}
        \int_0^1 \|\Pi(\vartheta)C_1B_1W_3(r)\|^2\,dr\,&=&
       \int_0^1  \tr\left([B_1W_3(r)]^{\mathsf{T}}[\Pi(\vt)C_1]^{\mathsf{T}}[\Pi(\vartheta)C_1][B_1W_3(r)]\right)\,dr\\
        &=&  \tr\left([\Pi(\vt)C_1]^{\mathsf{T}}[\Pi(\vartheta)C_1]\int_0^1[B_1W_3(r)][B_1W_3(r)]^{\mathsf{T}}\,dr\right)
\end{eqnarray*}
where we used Bernstein \cite[2.2.27]{Bernstein2009} to permutate the matrices in the trace. %which allows us to permutate matrices in the trace.
 The random matrix $\int_0^1[B_1W_3(r)]^{\mathsf{T}}[B_1W_3(r)]\,dr$ is $\P$-a.s. positive definite since $B_1$ and the covariance matrix
 of $W_3$ have full rank.
 Hence, there exists an $m\times m$-dimensional symmetric positive random matrix $W^*$
with $\int_0^1[B_1W_3(r)][B_1W_3(r)]^{\mathsf{T}}\,dr=W^{*}W^{*\mathsf{T}}$. Then, we obtain similarly as above with Bernstein \cite[2.2.27]{Bernstein2009}
\begin{eqnarray*}
        \int_0^1 \|\Pi(\vartheta)C_1B_1W_3(r)\|^2\,dr=\tr\left([W^*]^{\mathsf{T}}[\Pi(\vt)C_1]^{\mathsf{T}}[\Pi(\vartheta)C_1]W^*\right)
        =\|\Pi(\vt)C_1W^*\|^2.
\end{eqnarray*}
Again an application of  Bernstein \cite[Corollary 9.6.7]{Bernstein2009} and \eqref{C.2} yields
\begin{eqnarray} \label{C11}
        \int_0^1 \|(V_\vt^{(h)})^{-1/2}\Pi(\vartheta)C_1B_1W_3(r)\|^2\,\dif r
       \hspace*{-0.2cm}&\geq&\hspace*{-0.2cm} \sigma_{\min}((V_\vt^{(h)})^{-1})\int_0^1\|\Pi(\vartheta)C_1B_1W_3(r)\|^2\,dr\\
        &=&\hspace*{-0.2cm} \sigma_{\min}((V_\vt^{(h)})^{-1})\|\Pi(\vt)C_1W^{*}\|^2 \nonumber\\
        &\geq&\hspace*{-0.2cm} \sigma_{\min}((V_\vt^{(h)})^{-1})
        \sigma_{\min}\left(W^{*}W^{*\mathsf{T}}\right)\|\Pi(\vt)C_1\|^2 \nonumber\\
        &=&\hspace*{-0.2cm} \sigma_{\min}((V_\vt^{(h)})^{-1})
        \sigma_{\min}\left(\int_0^1B_1W_3(r)[B_1W_3(r)]^{\mathsf{T}}\,dr\right)\|\Pi(\vt)C_1\|^2. \nonumber
\end{eqnarray}
Since $B_1\int_0^1W_3(r)W_3(r)^{\mathsf{T}}\dif r B_1^{\mathsf{T}}$ is $\P$-a.s. positive definite  $\sigma_{\min}\left(B_1\int_0^1W_3(r)W_3(r)^T\dif rB_1^{\mathsf{T}}\right)>0$ $\p$-a.s.
On the one hand, $\inf_{\vartheta\in\overline{\mathcal{B}}(\vartheta_{1}^0,\delta)\times\Theta_2}\sigma_{\min}((V_\vt^{(h)})^{-1})>0$ due Lemma~\ref{Lemma 2.3}(c).
On the other hand, \autoref{AssMPosDefN1} (see Remark~\ref{Remark:id}) implies that
$\inf_{\vartheta\in\overline{\mathcal{B}}(\vartheta_{1}^0,\delta)\times\Theta_2}\Vert \Pi(\vartheta) C_1 \Vert^2 >\mathfrak{C}^2\delta^2>0$. To conclude
\begin{eqnarray*}
    \inf_{\vartheta\in\overline{\mathcal{B}}(\vartheta_{1}^0,\delta)\times\Theta_2}\int_0^1 \|(V_\vt^{(h)})^{-1/2}\Pi(\vartheta)C_1B_1W_3(r)\|^2\,\dif r>0 \quad \P\text{-a.s.,}
\end{eqnarray*}
which finally gives with \eqref{C.1} that
$
    \inf_{\vartheta\in\overline{\mathcal{B}}(\vartheta_{1}^0,\delta)\times\Theta_2}\mathcal{L}_{n,1}^{(h)}(\vartheta) \xrightarrow[]{ ~p~ }\infty
$ and thus, \eqref{eqQgeqc} is proven.
\end{proof}
%Therefore, we complete the proof of \Cref{eqSuffCondforCons} for $\gamma=0$  and hence have consistency of the long-run parameter estimator.

\subsection{Super-consistency of the long-run QML estimator} \label{Sec:cons:longrun}

From the previous section we already know that the QML estimator $\widehat\vt_{n,1}$ for the long-run parameter is consistent. In the following,
we will calculate its consistency rate.
For $0\leq \gamma< 1$  define the set
\begin{align} \label{N}
N_{n,\gamma}(\vartheta_{1}^0,\delta):=\left\{\vartheta_1\in\Theta_1:\|\vartheta_1-\vartheta_1^0\|\leq\delta n^{-\gamma}\right\},\quad  n\in\N,
\end{align}
 and $\overline{N}_{n,\gamma}(\vartheta_{1}^0,\delta):=\Theta_1\backslash N_{n,\gamma}(\vartheta_{1}^0,\delta)$ as its complement.
 As Saikkonen \cite[eq. (26)]{Saikkonen1995} we receive the consistency rate from the next statement.

\begin{theorem}
\label{eqSuffCondforCons}
Let $0\leq \gamma< 1$.  For any $\delta>0$ we have
\begin{align*}
\lim_{n\to\infty}\P\bigg(\inf_{\vartheta\in\overline{N}_{n,\gamma}(\vartheta_{1}^0,\delta)\times\Theta_2}
\hspace{-3pt}
\mathcal{\widehat{L}}_n^{(h)}(\vartheta)-\mathcal{\widehat{L}}_n^{(h)}(\vartheta^0)>0\bigg)=1.
\end{align*}
\end{theorem}

\begin{corollary}  \label{super-consist}
In particular, $\widehat{\vartheta}_{n,1}-\vartheta_1^0=o_p(n^{-\gamma})$ for $0\leq \gamma <1$.
\end{corollary}

\subsubsection{Proof of \Cref{eqSuffCondforCons}}

The proof uses the next lemma.

\begin{lemma} \label{Lemma 3.5b}
Let the notation of Lemma~\ref{Lemma 3.5a} hold. Then, \begin{itemize}
\item[(a)]
$\mathcal{L}_{n,1,1}^{(h)}(\vartheta)\geq \mathfrak{C} \sigma_{\min}((V_\vt^{(h)})^{-1})\|\vt_1-\vt_1^0\|^2\sigma_{\min}\left(\frac{1}{n}\sum_{k=1}^nB_1L_{k-1}^{(h)}
[B_1L_{k-1}^{(h)}]^{\mathsf{T}}\right).
$
\item[(b)] $\mathcal{L}_{n,1,1}^{(h)}(\vartheta)
\leq  \mathfrak{C}\|(V_\vt^{(h)})^{-1}\|\|\vt_1-\vt_1^0\|^2\tr\left(\frac{1}{n}\sum_{k=1}^nB_1L_{k-1}^{(h)}[B_1L_{k-1}^{(h)}]^{\mathsf{T}}\right)$.
\end{itemize}
\end{lemma}
\begin{proof}$\mbox{}$\\
(a) \, Several applications of Bernstein \cite[Corollary 9.6.7]{Bernstein2009} give similarly  as in \eqref{C11}
\begin{eqnarray*}
    \mathcal{L}_{n,1,1}^{(h)}(\vartheta)&=&\frac{1}{n}\sum_{k=1}^n\|(V_\vt^{(h)})^{-1/2}\Pi(\vartheta)C_1B_1L_{k-1}^{(h)}\|^2 \\
    &\geq& \sigma_{\min}((V_\vt^{(h)})^{-1}) \sigma_{\min}\left(\frac{1}{n}\sum_{k=1}^nB_1L_{k-1}^{(h)}[B_1L_{k-1}^{(h)}]^{\mathsf{T}}\right)\|\Pi(\vt)C_1\|^2.
\end{eqnarray*}
An application of \autoref{AssMPosDefN1} (see Remark~\ref{Remark:id}) yields (a).\\[2mm]
(b) \, The submultiplicativity of the norm gives
\begin{eqnarray*}
     \mathcal{L}_{n,1,1}^{(h)}(\vartheta)&=&\frac{1}{n}\sum_{k=1}^n\|(V_\vt^{(h)})^{-1/2}\Pi(\vartheta)C_1B_1L_{k-1}^{(h)}\|^2 \\
     &\leq& \|(V_\vt^{(h)})^{-1/2}\|^2\|\Pi(\vt)C_1\|^2\frac{1}{n}\sum_{k=1}^n\|B_1L_{k-1}^{(h)}\|^2\\
        &=&\|(V_\vt^{(h)})^{-1/2}\|^2\|\Pi(\vt)C_1\|^2\tr\left(\frac{1}{n}\sum_{k=1}^n [B_1L_{k-1}^{(h)}][B_1L_{k-1}^{(h)}]^{\mathsf{T}}\right).
\end{eqnarray*}
In the last line we applied Bernstein \cite[2.2.27]{Bernstein2009}. Due to  $\Pi(\vt_1^0,\vt_2)C_1=0_{d\times c}$ we have
\beao
    \mathcal{L}_{n,1,1}^{(h)}(\vartheta)\leq    \|(V_\vt^{(h)})^{-1}\|\|\Pi(\vt)C_1-\Pi(\vt_1^0,\vt_2)C_1\|^2\tr\left(\frac{1}{n}\sum_{k=1}^nB_1L_{k-1}^{(h)}[B_1L_{k-1}^{(h)}]^{\mathsf{T}}\right).
\eeao
Finally, the Lipschitz continuity of $\Pi(\vt)$ and hence, of $\Pi(\vt)C_1$ yield the statement.
\end{proof}

A conclusion of Lemma~\ref{Lemma 3.5a} and Lemma~\ref{Lemma 3.5b} is the local Lipschitz continuity of $\mathcal{L}_{n,1}^{(h)}(\cdot,\vartheta_2)$ in $\vartheta_1^0$.
Essential for the proof of Theorem~\ref{eqSuffCondforCons} is as well the local Lipschitz continuity of
$\mathcal{L}_{n,1,2}^{(h)}(\cdot,\vartheta_2)$  in $\vartheta_1^0$.

\begin{proof}[Proof of \Cref{eqSuffCondforCons}]
Due to Proposition~\ref{Lemma 2.8} the lower bound
\begin{align*}
\inf_{\vartheta\in\overline{N}_{n,\gamma}(\vartheta_{1}^0,\delta)\times\Theta_2}
\hspace{-3pt}
n \big(\mathcal{\widehat{L}}_{n}^{(h)}(\vartheta)-\mathcal{\widehat{L}}_{n}^{(h)}(\vartheta^0) \big)
&\geq \inf_{\vartheta\in\overline{N}_{n,\gamma}(\vartheta_{1}^0,\delta)\times\Theta_2}
\hspace{-3pt}
n \big(\mathcal{L}_{n}^{(h)}(\vartheta)-\mathcal{L}_{n}^{(h)}(\vartheta^0) \big)+o_p(1)\\
&\hspace*{-3cm}\geq \inf_{\vartheta\in\overline{N}_{n,\gamma}(\vartheta_{1}^0,\delta)\times\Theta_2}
\hspace{-3pt}
n \mathcal{L}_{n,1}^{(h)}(\vartheta)
+\inf_{\vartheta\in\Theta_2}n\big(\mathcal{L}_{n,2}^{(h)}(\vartheta_2)-\mathcal{L}_{n,2}^{(h)}(\vartheta_2^0)\big)+o_p(1)
\end{align*}
holds. %
We investigate now the second term. Note that $\mathcal{L}_{n,2}^{(h)}(\vartheta_2)$ depends only on the short-run parameters. Therefore, we take the infeasible estimator $\widehat{\vt}_{n,2}^{st}:=\argmin_{\vt_2\in\Theta_2} \mathcal{L}_{n,2}^{(h)}(\vartheta_2) $ for the short-run parameter
$\vt_2^0$  minimizing $\mathcal{L}_{n,2}^{(h)}(\vartheta_2)$.  For this reason, we can interpret this as a \glqq classical\grqq\ stationary estimation problem. Applying a
Taylor-expansion of $n\mathcal{L}_{n,2}^{(h)}$ around $\vt_2^0$
yields $$n\cdot\big(\mathcal{L}_{n,2}^{(h)}(\widehat{\vt}_{n,2}^{st})
 -\mathcal{L}_{n,2}^{(h)}(\vartheta_2^0)\big)
 =\big(\sqrt{n}\nabla_{\vt_2}\mathcal{L}_{n,2}^{(h)}(\underline{\vartheta}_{n,2}) \big)\cdot \big(\sqrt{n}(\widehat{\vt}_{n,2}^{st}-\vt_2^0)\big)$$ for an appropriate intermediate value $\underline{\vartheta}_{n,2}\in\Theta_2$ with $\Vert \underline{\vartheta}_{n,2}-\vt_2^0\Vert\leq \Vert \widehat{\vt}_{n,2}^{st}-\vt_2^0\Vert $.
 Since $\sqrt{n}\nabla_{\vt_2}\mathcal{L}_{n,2}^{(h)}(\underline{\vartheta}_{n,2}) $ and $\sqrt{n}(\widehat{\vt}_{n,2}^{st}-\vt_2^0)$
are asymptotically normally distributed
 (these are special and easier calculations as in \Cref{Sec.NormalShort}) we can conclude $n\cdot\big(\mathcal{L}_{n,2}^{(h)}(\widehat{\vt}_{n,2}^{st})
 -\mathcal{L}_{n,2}^{(h)}(\vartheta_2^0)\big)=\mathcal{O}_p(1)$.
 Finally,
 \begin{align*}
&\inf_{\vartheta\in\overline{N}_{n,\gamma}(\vartheta_{1}^0,\delta)\times\Theta_2}
\hspace{-3pt}
n\cdot \big(\mathcal{\widehat{L}}_{n}^{(h)}(\vartheta)-\mathcal{\widehat{L}}_{n}^{(h)}(\vartheta^0) \big)\geq\inf_{\vartheta\in\overline{N}_{n,\gamma}(\vartheta_{1}^0,\delta)\times\Theta_2}
\hspace{-3pt}
 n\cdot \mathcal{L}_{n,1}^{(h)}(\vartheta)+\mathcal{O}_p(1).
%\label{eqStep2ConsSuffCond}
\end{align*}
Thus, if we can show that
\begin{eqnarray} \label{eqConsistRateTheta1}
 \sup_{\vartheta\in\overline{N}_{n,\gamma}(\vartheta_{1}^0,\delta)\times \Theta_2 }n\mathcal{L}_{n,1}^{(h)}(\vartheta)\cip\infty,
%\lim_{n\to\infty} \P \bigg(
%\inf_{\vartheta\in\overline{N}_{n,\gamma}(\vartheta_{1}^0,\delta)\times\Theta_2}
%\hspace{-3pt}
% n\cdot\mathcal{L}_{n,1}^{(h)}(\vartheta)
%>c\bigg)=1
\end{eqnarray}
then for any $\tau>0$
\begin{align*}
&\hspace*{-1cm}\lim_{n\to\infty}\P\bigg(\inf_{\vartheta\in\overline{N}_{n,\gamma}(\vartheta_{1}^0,\delta)\times\Theta_2}
\mathcal{\widehat{L}}_n^{(h)}(\vartheta)-\mathcal{\widehat{L}}_n^{(h)}(\vartheta^0)>0\bigg)\\
&\geq
\lim_{n\to\infty}\P\bigg(\inf_{\vartheta\in\overline{N}_{n,\gamma}(\vartheta_{1}^0,\delta)\times\Theta_2}
n\left(\mathcal{\widehat{L}}_n^{(h)}(\vartheta)-\mathcal{\widehat{L}}_n^{(h)}(\vartheta^0)\right)>\tau\bigg)=1.
\end{align*}
Before we prove \eqref{eqConsistRateTheta1} we first note that due to \eqref{eqQgeqc} we only have to
consider the set
\begin{align*}
\overline{M}_{n,\gamma}(\vartheta_{1}^0,\delta_1):= \overline{N}_{n,\gamma}(\vartheta_{1}^0,\delta_1)\cap \mathcal{B}(\vartheta_{1}^0,\delta_1)\subseteq \Theta_1\cap \mathcal{B}(\vartheta_{1}^0,\delta_1)
\end{align*}
for $n$ large enough instead of the whole set $\overline{N}_{n,\gamma}(\vartheta_{1}^0,\delta_1)$ in the infimum.
Note that  $\inf_{\vt\in\Theta}\sigma_{\min}((V_\vt^{(h)})^{-1})>0$ by Lemma~\ref{Lemma 2.3}(c).
Then, Lemma~\ref{Lemma 3.5a} and Lemma~\ref{Lemma 3.5b} give the lower bound
\begin{eqnarray*}
    \mathcal{L}_{n,1}^{(h)}(\vartheta)&\geq& \mathcal{L}_{n,1,1}^{(h)}(\vartheta)-|\mathcal{L}_{n,1,2}^{(h)}(\vartheta)|\\
        &\geq & \mathfrak{C}\|\vt_1-\vt_1^0\|^2\sigma_{\min}\left(\frac{1}{n}\sum_{k=1}^nB_1L_{k-1}^{(h)}
[B_1L_{k-1}^{(h)}]^{\mathsf{T}}\right)-\mathfrak{C}\|\vt_1-\vt_1^0\| U_n\\
        &\geq & \mathfrak{C} n\|\vt_1-\vt_1^0\|^2\underbrace{\left(\sigma_{\min}\left(\frac{1}{n^2}\sum_{k=1}^nB_1L_{k-1}^{(h)}
[B_1L_{k-1}^{(h)}]^{\mathsf{T}}\right)-\frac{1}{n \|\vt_1-\vt_1^0\|} U_n\right)}_{=:Z_n(\vt)}.
\end{eqnarray*}
%An application of \eqref{eqIntermedStep} results in
Finally,
\begin{eqnarray*}
    \inf_{\vartheta\in\overline{M}_{n,\gamma}(\vartheta_{1}^0,\delta)\times \Theta_2 }n\mathcal{L}_{n,1}^{(h)}(\vartheta)
        &\geq& \left(\inf_{\vartheta \in\overline{M}_{n,\gamma}(\vartheta_{1}^0,\delta)\times \Theta_2 } \mathfrak{C} n^2\|\vt_1-\vt_1^0\|^2 \right) \left(\inf_{\vartheta\in\overline{M}_{n,\gamma}(\vartheta_{1}^0,\delta)\times \Theta_2 }Z_n(\vt)\right)\\
       %  &\geq& \inf_{(\vartheta_1,\vt_2)\in\overline{M}_{n,\gamma}(\vartheta_{1}^0,\delta)\times \Theta_2 } C n^2\|\vt_1^0-\vt_1\|^2 %\sup_{(\vartheta_1,\vt_2)\in\overline{M}_{n,\gamma}(\vartheta_{1}^0,\delta)\times \Theta_2 }X_n(\vt)\\
        &\geq& \mathfrak{C} n^{2-2\gamma}\inf_{\vartheta\in\overline{M}_{n,\gamma}(\vartheta_{1}^0,\delta)\times \Theta_2 }Z_n(\vt).
\end{eqnarray*}
Due to Proposition~\ref{PropUniformConvRes1}(b) and Lemma~\ref{Lemma 3.5a}, we receive
$$\inf_{\vartheta\in\overline{M}_{n,\gamma}(\vartheta_{1}^0,\delta)\times \Theta_2 }Z_n(\vt)\ccid \sigma_{\min}\left(B_1\int_0^1W_3(r)W_3(r)^{\mathsf{T}}\dd r B_1^{\mathsf{T}}\right)>0\quad \P\text{-a.s.}$$ Thus,
finally  $ \sup_{\vartheta\in\overline{M}_{n,\gamma}(\vartheta_{1}^0,\delta)\times \Theta_2 }n\mathcal{L}_{n,1}^{(h)}(\vartheta)\cip\infty$ for $0\leq\gamma<1$.
\end{proof}

\subsection{Consistency of the short-run QML estimator}
\label{sec4:sub2}
 Next, we consider the consistency of the short-run parameter estimator  $\widehat{\vartheta}_{n,2}$  with the help of the order
 of consistency of the long-run parameter estimator  $\widehat{\vartheta}_{n,1}$  which we determined in Corollary~\ref{super-consist}.
 Similarly to Saikkonen~\cite[eq. (31)]{Saikkonen1995} we show a sufficient condition given by the next theorem.
 Therefore, define for $\delta>0$ the set $
 \mathcal{B}(\vartheta_2^0,\delta):= \{\vartheta_2\in\Theta_2:\|\vartheta_2-\vartheta_2^0\|\leq\delta\}
$
as closed ball with radius $\delta$ around $\vt_2^0$ and
$\overline{\mathcal{B}}(\vartheta_{2}^0,\delta):=\Theta_2\backslash {\mathcal{B}}(\vartheta_{2}^0,\delta)$ as its complement.

 \begin{theorem}
 \label{eqConsistRateTheta2} \label{ball}
Then,  for any $\delta>0$ we have
\begin{align*}
\lim_{n\to\infty}\P \bigg(
\inf_{\vartheta\in \Theta_1\times \overline{\mathcal{B}}(\vartheta_2^0,\delta)}
\hspace{-3pt}
\mathcal{\widehat{L}}_{n}^{(h)}(\vartheta)-\mathcal{\widehat{L}}_{n}^{(h)}(\vartheta^0)> 0\bigg)=1.
\end{align*}
\end{theorem}

\begin{corollary} \label{corollary:3.9}
In particular, $
%\label{eqconsistency:b}
\widehat{\vt}_{n,2}-\vt_2^0=o_p(1).
$
\end{corollary}

\subsubsection{Proof of \Cref{eqConsistRateTheta2}}

Again we prove some auxiliary results before we state the proof of the theorem. Lemma~\ref{eqConsistRateTheta2Part1} corresponds
to Saikkonen~\cite[eq. (32)]{Saikkonen1995} and Lemma~\ref{eqConsistRateTheta2Part2} to Saikkonen~\cite[eq. (33)]{Saikkonen1995} for the regression model.

\begin{lemma} \label{eqConsistRateTheta2Part1}
 For $\frac{1}{2}<\gamma< 1$,  $\delta_1>0$ and $\tau>0$ we have
\begin{eqnarray*}
\lim_{n\to\infty}\P \bigg(\sup_{\vartheta\in N_{n,\gamma}(\vartheta_{1}^0,\delta_1)\times \Theta_2}
\hspace{-3pt}
\vert \mathcal{L}_{n,1}^{(h)}(\vartheta)\vert \leq \tau \, \bigg)=1.
\end{eqnarray*}
\end{lemma}
\begin{proof}
Due to Lemma~\ref{Lemma 3.5a} and Lemma~\ref{Lemma 3.5b} we have the
upper bound
\begin{eqnarray*}
    |\mathcal{L}_{n,1}^{(h)}(\vartheta)|\hspace*{-0.2cm}&\leq&\hspace*{-0.2cm} |\mathcal{L}_{n,1,1}^{(h)}(\vartheta)|+|\mathcal{L}_{n,1,2}^{(h)}(\vartheta)|\\
        &\leq&\hspace*{-0.2cm} \mathfrak{C}\|(V_\vt^{(h)})^{-1}\|\|\vt_1-\vt_1^0\|^2\tr\left(\frac{1}{n}\sum_{k=1}^nB_1L_{k-1}^{(h)}
[B_1L_{k-1}^{(h)}]^{\mathsf{T}}\right)+\mathfrak{C}\|\vt_1-\vt_1^0\| U_n.
\end{eqnarray*}
%Moreover, due to $\Pi(\vt_1^0,\vt_2)C_1=0$ (see \autoref{AssMUniquePosTriForm}) and the Lipschitz-continuity of $\Pi(\vt)$ (see Scholz~\cite[Lemma 5.9.3]{Scholz}) we have
%\begin{eqnarray*} \label{C.3}
%    \|\Pi(\vt)C_1\|^2=\|\Pi(\vt)-\Pi(\vt_1^0,\vt_2)C_1\|^2\leq \|\Pi(\vt)-\Pi(\vt_1^0,\vt_2)\|^2\|C_1\|^2\leq C\|\vt_1-\vt_1^0\|^2.
%\end{eqnarray*}
Then,  using Lemma~\ref{Lemma 2.3}(b) results in %on the one hand,
%\beao
%\sup_{\vartheta\in N_{n,\gamma}(\vartheta_{1}^0,\delta_1)\times \Theta_2}\|(V_\vt^{(h)})^{-1}\|\|\Pi(\vt)C_1\|^2
%%&\leq& \sup_{\vartheta\in N_{n,\gamma}(\vartheta_{1}^0,\delta_1)\times \overline{\mathcal{B}}(\vartheta_2^0,\delta)}\|\Pi(\vt)-\Pi(\vt_1^0,\vt_2)C_1\|^2\\
%%&\leq& \sup_{\vartheta\in N_{n,\gamma}(\vartheta_{1}^0,\delta_1)\times \overline{\mathcal{B}}(\vartheta_2^0,\delta)}C\|\Pi(\vt)-\Pi(\vt_1^0,\vt_2)\|^2\|C_1\|^2\\
%&\leq& \sup_{\vartheta\in N_{n,\gamma}(\vartheta_{1}^0,\delta_1)\times \Theta_2}C\|\vt_1-\vt_1^0\|^2\leq C\delta_1 n^{-2\gamma}.
%\eeao
%On the other hand,  $\sup_{\vartheta\in N_{n,\gamma}(\vartheta_{1}^0,\delta_1)\times \Theta_2}\|\vt_1-\vt_1^0\|\leq C \delta_1n^{-\gamma}$.
%Hence,
    \begin{eqnarray} \label{A1}
    \sup_{\vartheta\in N_{n,\gamma}(\vartheta_{1}^0,\delta_1)\times \Theta_2}\mathcal{L}_{n,1}^{(h)}(\vartheta)
     \leq \mathfrak{C}\delta_1^2n^{1-2\gamma}\tr\left(\frac{1}{n^2}\sum_{k=1}^nB_1L_{k-1}^{(h)}
[B_1L_{k-1}^{(h)}]^{\mathsf{T}}\right)+\mathfrak{C}\delta_1n^{-\gamma} U_n.
\end{eqnarray}
Since $U_n=\mathcal{O}_p(1)$ by Lemma~\ref{Lemma 3.5a} and
\beao
    \tr\left(\frac{1}{n^2}\sum_{k=1}^nB_1L_{k-1}^{(h)}
[B_1L_{k-1}^{(h)}]^{\mathsf{T}}\right)\ccid\tr\left(B_1\int_0^1W_3(r)W_3(r)^{\mathsf{T}}\dd r B_1^{\mathsf{T}}\right)>0 \quad\P\text{-a.s.}
\eeao
 by Proposition~\ref{PropUniformConvRes1}(b) and the continuous mapping theorem, the right hand side of \eqref{A1} converges to $0$ in probability if $\frac{1}{2}<\gamma< 1$.
 This proves the lemma.
\end{proof}

\begin{lemma} \label{eqConsistRateTheta2Part2}
For any $\delta>0$ and $\tau>0$ we have
\begin{eqnarray*}
\lim_{n\to\infty}\P \bigg(\inf_{\vartheta_2\in\overline{\mathcal{B}}(\vartheta_2^0,\delta)}\mathcal{L}_{n,2}^{(h)}(\vartheta_2)-\mathcal{L}_{n,2}^{(h)}(\vartheta_2^0)> \tau\bigg)=1.
\end{eqnarray*}
\end{lemma}
\begin{proof}
We have
\begin{eqnarray*}
\lefteqn{\inf_{\vartheta_2\in\overline{\mathcal{B}}(\vartheta_2^0,\delta)}
\hspace{-3pt}
\left(\mathcal{L}_{n,2}^{(h)}(\vartheta_2)-\mathcal{L}_{n,2}^{(h)}(\vartheta_2^0)\right)}\\
%&=\inf_{\vartheta_2\in\overline{\mathcal{B}}(\vartheta_2^0,\delta)} \left(\mathcal{L}_{n,2}^{(h)}(\vartheta_2)-\mathbfcal{L}_{2}^{(h)}(\vartheta_2)-\mathcal{L}_{n,2}^{(h)}(\vartheta_2^0)
%+\mathbfcal{L}_{2}^{(h)}(\vartheta_2^0)+\mathbfcal{L}_{2}^{(h)}(\vartheta_2)-\mathbfcal{L}_{2}^{(h)}(\vartheta_2^0)\right)
%\notag
%\\
&&\quad\geq\inf_{\vartheta_2\in\overline{\mathcal{B}}(\vartheta_2^0,\delta)} \left(\mathcal{L}_{n,2}^{(h)}(\vartheta_2)-\mathbfcal{L}_{2}^{(h)}(\vartheta_2)\right)
+\inf_{\vartheta_2\in\overline{\mathcal{B}}(\vartheta_2^0,\delta)}\left(-\mathcal{L}_{n,2}^{(h)}
(\vartheta_2^0)+\mathbfcal{L}_{2}^{(h)}(\vartheta_2^0)\right)
\notag
\\
&&\quad
\quad+\inf_{\vartheta_2\in\overline{\mathcal{B}}(\vartheta_2^0,\delta)}\left(\mathbfcal{L}_{2}^{(h)}(\vartheta_2)-\mathbfcal{L}_{2}^{(h)}(\vartheta_2^0)\right).
\end{eqnarray*}
On the one hand, the first two terms converge to zero in probability, due to Lemma~\ref{propConvEps2}(a) and the continuous mapping theorem.
On the other hand, $\inf_{\vartheta_2\in\overline{\mathcal{B}}(\vartheta_2^0,\delta)} \big(\mathbfcal{L}_{2}^{(h)}(\vartheta_2) -\mathbfcal{L}_{2}^{(h)}(\vartheta_2^0)\big)>0$ since $\mathbfcal{L}_{2}^{(h)}(\vartheta_2)$ has a unique minimum in $\vartheta_2^0$ by Lemma~\ref{LemUniqueMin}.
\end{proof}

\begin{proof}[Proof of \Cref{eqConsistRateTheta2}]
Let us assume that $\frac{1}{2}<\gamma<1$. Apparently, the parameter subspace $\Theta_1$ is the union of  $\Theta_1=\overline{N}_{n,\gamma}(\vartheta_{1}^0,\delta_1)\cup N_{n,\gamma}(\vartheta_{1}^0,\delta_1)$ and thus, we have already shown \Cref{eqConsistRateTheta2} for the set $\overline{N}_{n,\gamma}(\vartheta_{1}^0,\delta_1)\times\overline{\mathcal{B}}(\vartheta_2^0,\delta)$ instead of $\Theta_1\times\overline{\mathcal{B}}(\vartheta_2^0,\delta)$  in \Cref{eqSuffCondforCons}. It remains to investigate ${N}_{n,\gamma}(\vartheta_{1}^0,\delta_1)\times\overline{\mathcal{B}}(\vartheta_2^0,\delta)$.
 For any $\delta_1>0$ we obtain by Proposition~\ref{Lemma 2.8}
\begin{align*}
&\hspace*{-0.5cm}
\lim_{n\to\infty}
\P \bigg(
\inf_{\vartheta\in N_{n,\gamma}(\vartheta_{1}^0,\delta_1)\times \overline{\mathcal{B}}(\vartheta_2^0,\delta)}
\hspace{-3pt}
\big(\mathcal{\widehat{L}}_{n}^{(h)}(\vartheta)-\mathcal{\widehat{L}}_{n}^{(h)}(\vartheta^0)\big)> 0\bigg)
\\
&=
\lim_{n\to\infty}
\P \bigg(
\inf_{\vartheta\in N_{n,\gamma}(\vartheta_{1}^0,\delta_1)\times \overline{\mathcal{B}}(\vartheta_2^0,\delta)}
\hspace{-3pt}
\big(\mathcal{L}_{n}^{(h)}(\vartheta)-\mathcal{L}_{n}^{(h)}(\vartheta^0)\big)> 0\bigg)
\\
&
\geq
\lim_{n\to\infty}
\P \bigg(
\inf_{\vartheta\in N_{n,\gamma}(\vartheta_{1}^0,\delta_1)\times \overline{\mathcal{B}}(\vartheta_2^0,\delta)}
\hspace{-3pt}
 \mathcal{L}_{n,1}^{(h)}(\vartheta)
+\inf_{\vartheta_2\in\overline{\mathcal{B}}(\vartheta_2^0,\delta)}
\hspace{-3pt}
\big(\mathcal{L}_{n,2}^{(h)}(\vartheta_2)-\mathcal{L}_{n,2}^{(h)}(\vartheta_2^0)\big)> 0\bigg)
\\
&
\geq
\lim_{n\to\infty}
\P \bigg(
\sup_{\vartheta\in N_{n,\gamma}(\vartheta_{1}^0,\delta_1)\times \overline{\mathcal{B}}(\vartheta_2^0,\delta)}
\hspace{-3pt}
\vert \mathcal{L}_{n,1}^{(h)}(\vartheta)\vert \leq \tau \, ; \,\inf_{\vartheta_2\in\overline{\mathcal{B}}(\vartheta_2^0,\delta)} \mathcal{L}_{n,2}^{(h)}(\vartheta_2)-\mathcal{L}_{n,2}^{(h)}(\vartheta_2^0)> \tau\bigg) .
\end{align*}
Then, a consequence of Lemma~\ref{eqConsistRateTheta2Part1} and Lemma~\ref{eqConsistRateTheta2Part2} is
\begin{align*}
\lim_{n\to\infty}
\P \bigg(
\inf_{\vartheta\in N_{n,\gamma}(\vartheta_{1}^0,\delta_1)\times \overline{\mathcal{B}}(\vartheta_2^0,\delta)}
\big(\mathcal{\widehat{L}}_{n}^{(h)}(\vartheta)-\mathcal{\widehat{L}}_{n}^{(h)}(\vartheta^0)\big)> 0\bigg)\geq 1,
\end{align*}
which proves in combination with  \Cref{eqSuffCondforCons}  the claim.
\end{proof}

%% ==============================
\section{Asymptotic distributions of the QML estimator}
\label{sec:5}
%% ==============================
The aim of this section is to derive the asymptotic distributions of the  long-run parameter estimator  $\widehat{\vartheta}_{n,1}$ and the short-run parameter  estimator $\widehat{\vartheta}_{n,2}$. These two estimators have a different asymptotic behavior and a different convergence rate. On the one hand, we  prove the asymptotic normality of the short-run QML estimator and on the other hand, we show that the long-run QML estimator is asymptotically mixed normally distributed.

\begin{comment}
The Jacobian matrix of the $d\times d$ matrix function ${k}(z,\cdot)$ with respect to parameter vector $\vt_i$ in the interior of $\Theta_i$, $i=1,2$, is defined by \begin{align*}
\nabla_{\vt_i}{k}(z,\vt) :=\frac{\partial\vect({k}(z,\vt))}{\partial{\vt_i}^\mathsf{T}}\in\R^{d^2\times s_i}
\end{align*}
and analogous for the $d\times d$ matrix function $\Pi(\vt)=\alpha(\vt)\beta(\vt)^\mathsf{T}$ with respect to $\vt_i$ in the interior of $\Theta_i$
\begin{align*}
\nabla_{\vartheta_i}\Pi(\vt):&= \frac{\partial\vect(\Pi(\vt))}{\partial{\vartheta_i}^\mathsf{T}}
=
\left(I_d\otimes \alpha(\vt)\right)\frac{\partial\vect(\beta(\vt)^\mathsf{T})}{\partial{\vartheta_i}^\mathsf{T}}
+\left(\beta(\vt)\otimes I_d \right)\frac{\partial\vect(\alpha(\vt))}{\partial{\vartheta_i}^\mathsf{T}}\in\R^{d^2\times s_i},
 \end{align*}
 where the last equality holds due to the differentiation rules for matrix functions (see, e.g., L\"utkepohl \cite[Appendix A.13]{Luetkepohl2005}).
\end{comment}

\subsection{Asymptotic distribution of the long-run parameter estimator}
\label{sec:5:sub:1}

We derive in this section the asymptotic distribution of the long-run QML estimator $\widehat\vt_{n,1}$.
From Corollary~\ref{super-consist} we already know that
${\widehat{\vt}_{n,1}}-\vt_{1}^0=o_p(n^{-\gamma})$, for  $0\leq\gamma<1$.
 Since the true parameter $\vt^0=((\vt_{1}^0)^\mathsf{T},({\vt}_{2}^0)^\mathsf{T})^\mathsf{T}$ is an element of the interior of the compact parameter space $\Theta=\Theta_1\times\Theta_2$ due to \autoref{AssMBrownMot},
 the estimator $\widehat{\vt}_{n,1}$ is at some point also an element of the interior of $\Theta_1$ with probability one.
Because the parametrization is assumed to be threefold continuously differentiable, we can find the minimizing $
\widehat{\vt}_{n}=({\widehat{\vt}_{n,1}}^{\mathsf{T}},{\widehat{\vt}_{n,2}}^{\mathsf{T}})^{\mathsf{T}}$ via the first order condition $\nabla_{\vt_1}\mathcal{\widehat{L}}_n^{(h)}({\widehat{\vt}_{n,1}},{\widehat{\vt}_{n,2}})=0_{s_1}$.
We apply a Taylor-expansion of the score vector around the point $(\vt_1^0,{\widehat{\vt}_{n,2}})$ resulting in
\begin{align}
\label{eqTaylorExpansionDer1}
0_{s_1}=\nabla_{\vt_1} \mathcal{\widehat{L}}_n^{(h)}(\vt_1^0,\widehat\vt_{n,2})+n^{-1}\nabla_{\vt_1}^2 \underline{\mathcal{\widehat{L}}}_n^{(h)}(\underline{\vt}_{n,1},\widehat{\vt}_{n,2})n(\widehat{\vt}_{n,1}-\vt_1^0),
\end{align}
where $\nabla_{\vt_1}^2 \underline{\mathcal{\widehat{L}}}_n^{(h)}(\underline{\vt}_{n,1},\widehat{\vt}_{n,2})$ denotes the matrix whose $i^{th}$ row, $i=1,\ldots,s_1$, is equal to the $i^{th}$ row of $\nabla_{\vt_1}^2 \mathcal{\widehat{L}}_n^{(h)}(\underline{\vt}_{n,1}^i,\widehat{\vt}_{n,2})$ with $\underline{\vartheta}_{n,1}^i\in\Theta_1$ such that $\Vert \underline{\vartheta}_{n,1}^i-\vartheta_1^0\Vert\leq \Vert \widehat{\vartheta}_{n,1}-\vartheta_1^0\Vert$.
%where $\nabla_{\vt_1}^2 \mathcal{L}_n^{(h)}(\underline{\vt}_{n,1},\widehat{\vt}_{n,2})$ denotes the matrix whose $i^{th}$ row, for $i=1,\ldots,s_1$, is equal to the $i^{th}$ row of $\nabla_{\vt_1}^2 \mathcal{L}_n^{(h)}({\underline{\vt}_{n,1,i}},\widehat{\vt}_{n,2})$.
In the case \linebreak $\nabla_{\vt_1}^2 \underline{\mathcal{\widehat{L}}}_n^{(h)}(\underline{\vt}_{n,1},\widehat{\vt}_{n,2})$ is non-singular we receive
$$
n(\widehat{\vt}_{n,1}-\vt_1^0)=-\Big(n^{-1}\nabla_{\vt_1}^2 \underline{\mathcal{\widehat{L}}}_n^{(h)}(\underline{\vt}_{n,1},\widehat\vt_{n,2})\Big)^{-1}\nabla_{\vt_1} \mathcal{\widehat{L}}_n^{(h)}(\vt_1^0,\widehat\vt_{n,2}).
$$
Thus, we have to consider the asymptotic behavior of the score vector $\nabla_{\vt_1} \mathcal{\widehat{L}}_n^{(h)}(\vt) $ and the Hessian matrix
$\nabla_{\vt_1}^2 \mathcal{\widehat{L}}_n^{(h)}(\vt)$. Based on Proposition~\ref{Lemma 2.8} it is sufficient to consider $\nabla_{\vt_1}^2 \mathcal{{L}}_n^{(h)}(\vt) $
and $\nabla_{\vt_1} \mathcal{{L}}_n^{(h)}(\vt)$, respectively.

 % In the next step we have to derive some equicontinuity criteria because we evaluate the derivatives of the QML-function
%at some random time points and not a fixed value $\vt\in\Theta$.

\subsubsection{Asymptotic behavior of the score vector}
First, we  show the convergence of the gradient with respect to the long-run parameter $\vt_1$. For this, we consider the partial derivatives with respect to the ${i}^{th}$-component of the parameter vector $\vartheta$, ${i}=1,\ldots,s_1$, of the log-likelihood function. These partial derivatives are given due to differentiation rules for matrix functions (see, e.g., L\"utkepohl \cite[Appendix A.13]{Luetkepohl2005})
by
\begin{align}
\label{eqDerQ}
\partial_{i}\mathcal{L}_n^{(h)}(\vartheta)&=\tr\left(\big(V_\vt^{(h)}\big)^{-1} \partial_{i}V_\vt^{(h)}\right)
  -\frac{1}{n}\sum_{k=1}^n \tr\left(\big(V_\vt^{(h)}\big)^{-1} \varepsilon_k^{(h)}(\vartheta)\varepsilon_k^{(h)}(\vartheta)^\mathsf{T}\big(V_\vt^{(h)}\big)^{-1}\partial_{i}V_\vt^{(h)} \right)
  \notag
  \\
  &\quad
   +\frac{2}{n}\sum_{k=1}^n \big(\partial_{i}\varepsilon_k^{(h)}(\vartheta)^\mathsf{T}\big) \big(V_\vt^{(h)}\big)^{-1} \varepsilon_k^{(h)}(\vartheta).
\end{align}
From Section~\ref{Section: Properties linear innovation} we already know that the pseudo-innovations are indeed three times differentiable.

For reasons of brevity, we write $\partial_i^1:=\frac{\partial}{\partial \vt_{1i}}$ for the partial derivatives  with respect to the $i^{th}$-component of the long-run parameter vector $\vt_1\in\Theta_1$, $i\in\{1,\ldots,s_1\}$, and similarly  $\partial_j^{st}:=\frac{\partial}{\partial \vt_{2j}}$ for the partial derivatives with respect to the $j^{th}$-component of the short-run parameter vector $\vt_2\in\Theta_2$, $j\in\{1,\ldots,s_2\}$. Analogously we define $\partial_{i,j}^{1}$ and  $\partial_{i,j}^{st}$, respectively for the second partial derivatives.

%Let us now derive the asymptotic behavior of the score vector with respect to the long-run parameters.
\begin{proposition}
\label{PropConvergScoreVec2}
The score vector with respect to the long-run parameters $\vt_1$ satisfies
\begin{align*}
%\label{eqConvergenceScoreVec1}
\nabla_{\vartheta_1}\mathcal{L}_{n}^{(h)}(\vartheta^0)\cid \mathcal{J}_1(\vartheta^0):=\big(
\mathcal{J}_1^{(1)}(\vartheta^0)
\quad
\cdots
\quad
\mathcal{J}_1^{(s_1)}(\vartheta^0)
\big)^\mathsf{T},
\end{align*}
 where
 \begin{align*}
 \mathcal{J}_1^{(i)}(\vartheta^0)\,=\, &\,2\tr\bigg[\big(V_{\vt^0}^{(h)}\big)^{-1}\left(-\partial_i^1\Pi(\vartheta^0),0_{d\times d}\right) \int_0^1 W^{\#}(r)\,\dif W^{\#}(r)^\mathsf{T}
\left(\begin{array}{c}{\mathsf{k}}(1,\vartheta^0)\\ -\Pi(\vartheta^0)\end{array} \right)\bigg]\\\notag
&
+2\tr\left[\big(V_{\vt^0}^{(h)}\big)^{-1} \left(\Gamma_{\partial_i^1{\mathsf{k}}(\mathsf{B},\vartheta^0)\Delta Y^{(h)},\varepsilon^{(h)}(\vt^0)}(0)+
    \sum_{j=1}^\infty\Gamma_{-\partial_i^1\Pi(\vartheta^0)\Delta Y^{(h)},\varepsilon^{(h)}(\vt^0)}(j)\right) \right]
\end{align*}
and $(W^{\#}(r))_{0\leq r\leq 1}=((W_1(r)^{\mathsf{T}},W_2(r)^{\mathsf{T}})^{\mathsf{T}})_{0\leq r\leq 1}$ as defined on p.~\pageref{Brownian motion}.
\end{proposition}

\begin{proof}
Equation \eqref{eqDerQ} implies for $i=1,\ldots,s_1$ that
\begin{align*}
\partial_i^1 \mathcal{L}_{n}^{(h)}(\vartheta^0)
=&
\tr\bigg(\big(V_{\vt^0}^{(h)}\big)^{-1} \partial_{i}^1V_{\vt^0}^{(h)}\bigg)
  - \tr\bigg( \big(V_{\vt^0}^{(h)}\big)^{-1}\big(\partial_{i}^1V_{\vt^0}^{(h)}\big)\big(V_{\vt^0}^{(h)}\big)^{-1} \frac1n \sum_{k=1}^n\varepsilon_k^{(h)}(\vt^0)\varepsilon_k^{(h)}(\vt^0)^\mathsf{T} \bigg)
  \notag
  \\
  &
   +2\cdot\tr\bigg(\big(V_{\vt^0}^{(h)}\big)^{-1} \frac1n \sum_{k=1}^n  \big(\partial_{i}^1\varepsilon_k^{(h)}(\vt^0)\big)
   \varepsilon_k^{(h)}(\vt^0)^\mathsf{T}\bigg)=:I_{n,1}+I_{n,2}+I_{n,3}.
\end{align*}
Note that the second term $I_{n,2}$ converges due to the ergodicity of $( \varepsilon_k^{(h)}(\vt^0))_{k\in\N}$, \linebreak
 $\E(\varepsilon_k^{(h)}(\vt^0)\varepsilon_k^{(h)}(\vt^0)^\mathsf{T})=V_{\vt^0}^{(h)}$ (see Lemma~\ref{Lemma 2.5}(a,e)) and Birkhoff's Ergodic Theorem (see Bradley \cite[2.3 Ergodic
Theorem]{Bradley2007}) so that
\begin{align*}
I_{n,2}%& =-\tr\bigg( \big(V_{\vt^0}^{(h)}\big)^{-1}\big(\partial_{i}^1V_{\vt^0}^{(h)}\big)\big(V_{\vt^0}^{(h)}\big)^{-1} \frac1n \sum_{k=1}^n\varepsilon_k^{(h)}(\vt^0)\varepsilon_k^{(h)}(\vt^0)^\mathsf{T} \bigg)
%\\
&\cas
-\tr\left( \big(V_{\vt^0}^{(h)}\big)^{-1}\big(\partial_{i}^1V_{\vt^0}^{(h)}\big)\big(V_{\vt^0}^{(h)}\big)^{-1} V_{\vt^0}^{(h)} \right)=-\tr\left( \big(V_{\vt^0}^{(h)}\big)^{-1}\partial_{i}^1V_{\vt^0}^{(h)} \right).
\end{align*}
Hence, $I_{n,1}+I_{n,2}\cas 0$.
Thus, it only remains to show the convergence of the last term $I_{n,3}$. We obtain with  Proposition~\ref{PropLimitResultsFuncofDYY}(a,c)
and the continuous mapping theorem
\begin{align}
& \hspace*{-1cm} \frac1n \sum_{k=1}^n \big(\partial_{i}^1\varepsilon_k^{(h)}(\vt^0)\big)\varepsilon_k^{(h)}(\vt^0)^\mathsf{T} \notag
\\\notag
=&-
\frac1n \sum_{k=1}^n
[\left(\partial_i^1\Pi(\vartheta^0)\right)Y_{k-1}^{(h)}]
[{\mathsf{k}}(\mathsf{B},\vartheta^0)\Delta Y_k^{(h)}]^\mathsf{T}
+
\frac1n \sum_{k=1}^n
[\left(\partial_i^1\Pi(\vartheta^0)\right)Y_{k-1}^{(h)}]
[\Pi(\vartheta^0)Y_{st,k-1}^{(h)}]^\mathsf{T}
\\\notag
&
+
\frac1n \sum_{k=1}^n
[\left(\partial_i^1{\mathsf{k}}(\mathsf{B},\vartheta^0)\right)\Delta Y_k^{(h)}]
\varepsilon_k^{(h)}(\vt^0)^\mathsf{T}
\\\notag
%
%%%%%%%%%%%%%%%%%%%%%%%%%%%%%%%%%%%%%%%%%%%%%%%%%%%%%%%%%%%%%%%%%%
%
\cid &
-\left(\partial_i^1\Pi(\vartheta^0)\right)\int_0^1 W_1(r)\dif W_1(r)^\mathsf{T} {\mathsf{k}}(1,\vartheta^0)^\mathsf{T}
-\sum_{j=1}^\infty\Gamma_{\partial_i^1\Pi(\vartheta^0)\Delta Y^{(h)},{\mathsf{k}}(\mathsf{B},\vartheta^0)\Delta Y^{(h)}}(j)
\\\notag
&+\left(\partial_i^1 \Pi(\vartheta^0)\right)
\int_0^1 W_1(r)\,\dif W_2(r)^\mathsf{T}\Pi(\vartheta^0)^\mathsf{T}
+\sum_{j=1}^\infty\Gamma_{\partial_i^1\Pi(\vartheta^0)\Delta Y^{(h)}, \Pi(\vartheta^0)Y_{st}^{(h)}}(j)
\\
&+ \Gamma_{\partial_i^1{\mathsf{k}}(\mathsf{B},\vartheta^0)\Delta Y^{(h)},\varepsilon^{(h)}(\vt^0)}(0). \label{5.2}
\end{align}
Then, the continuous mapping theorem results in $I_{n,3}\cid  \mathcal{J}_1^{(i)}(\vartheta^0)$ which concludes the proof.
\end{proof}

\subsubsection{Asymptotic behavior of the Hessian matrix}

The second partial derivatives of the log-likelihood function $\mathcal{L}_n^{(h)}(\vartheta)$ are given
by
\begin{align}
\partial_{i,j}\mathcal{L}_n^{(h)}(\vartheta)&=
 \tr\left(\big(V_\vt^{(h)}\big)^{-1} \partial_{i,j}^2V_\vt^{(h)}-\big(V_\vt^{(h)}\big)^{-1} \big(\partial_{i}V_\vt^{(h)}\big)\big(V_\vt^{(h)}\big)^{-1} \big(\partial_{j}V_\vt^{(h)}\big)\right)
 \notag
 \\
 &\quad-\frac{1}{n}\sum_{k=1}^n \tr\left(\big(V_\vt^{(h)}\big)^{-1} \varepsilon_k^{(h)}(\vartheta)\varepsilon_k^{(h)}(\vartheta)^\mathsf{T}\big(V_\vt^{(h)}\big)^{-1}\partial_{i,j}^2V_\vt^{(h)} \right)
 \notag
 \\
 &\quad+\frac{1}{n}\sum_{k=1}^n \tr\left(\big(V_\vt^{(h)}\big)^{-1} \big(\partial_{j}V_\vt^{(h)}\big)\big(V^{(h)}\big)^{-1} \varepsilon_k^{(h)}(\vartheta)\varepsilon_k^{(h)}(\vartheta)^\mathsf{T} \big(V_\vt^{(h)}\big)^{-1}\partial_{i}V_\vt^{(h)} \right)
 \notag
  \\
 &\quad+\frac{1}{n}\sum_{k=1}^n \tr\left(\big(V_\vt^{(h)}\big)^{-1} \varepsilon_k^{(h)}(\vartheta)\varepsilon_k^{(h)}(\vartheta)^\mathsf{T} \big(V_\vt^{(h)}\big)^{-1}\big(\partial_{j}V_\vt^{(h)}\big)\big(V_\vt^{(h)}\big)^{-1} \partial_{i}V_\vt^{(h)} \right)
 \notag
 \\
  &\quad-\frac{1}{n}\sum_{k=1}^n \tr\left(\big(V_\vt^{(h)}\big)^{-1} \big(\partial_{j}\varepsilon_k^{(h)}(\vartheta)\varepsilon_k^{(h)}(\vartheta)^\mathsf{T}\big) \big(V_\vt^{(h)}\big)^{-1} \partial_{i}V_\vt^{(h)} \right)
 \notag
 \\
  &\quad
   +\frac{2}{n}\sum_{k=1}^n \left(\partial_{i,j}\varepsilon_k^{(h)}(\vartheta)^\mathsf{T}\right) \big(V_\vt^{(h)}\big)^{-1} \varepsilon_k^{(h)}(\vartheta)
   \notag
 \\
  &\quad-\frac{2}{n}\sum_{k=1}^n\tr\left(\big(V_\vt^{(h)}\big)^{-1} \varepsilon_k^{(h)}(\vartheta)
 \big(\partial_{i}\varepsilon_k^{(h)}(\vartheta)^\mathsf{T}\big) \big(V_\vt^{(h)}\big)^{-1}\partial_{j} V_\vt^{(h)}\right)
   \notag
   \\
 &\quad
 +\frac{2}{n}\sum_{k=1}^n\left(\partial_{i}\varepsilon_k^{(h)}(\vartheta)^\mathsf{T}\right) \big(V_\vt^{(h)}\big)^{-1} \left(\partial_{j}\varepsilon_k^{(h)}(\vartheta)\right) \notag\\
& =:\sum_{j=1}^8I_{n,j}.
\label{eqDerDerQ}
\end{align}

Since the Hessian matrix should be asymptotically positive definite we need an additional assumption.

\begin{assumptionletter}
 \label{AssMPosDefN}
The $((d-c)c\times s_1)$- dimensional gradient matrix $\left.\nabla_{\vartheta_1} \left(C_{1,\vartheta_1}^{\perp\mathsf{T}}C_1\right)\right|_{\vt_1=\vt_1^0}$ is of full column rank $s_1$.
\end{assumptionletter}

The asymptotic distribution of the Hessian matrix is given in the next proposition.
\begin{proposition}
\label{PropConvergHessianMat2}
Let \autoref{AssMPosDefN} additionally hold.
 Define the $(s_1\times s_1)$-dimensional random matrix $Z_1(\vt^0)$ as
\beao
[Z_1(\vt^0)]_{i,j}:=2\cdot\tr\bigg( \big(V^{(h)}_{\vt^0}\big)^{-1}\partial^1_i\Pi(\vartheta^0) \int_0^1 W_1(r)W_1(r)^\mathsf{T}\,\dif r\,\left(\partial^1_j\Pi(\vartheta^0) \right)^\mathsf{T}\bigg)%2\cdot\tr\Big(  \int_0^1\|\big(V^{(h)}_{\vt_0}\big)^{-1}\partial^1_i\Pi(\vartheta^0) W_1(r)\|^2\,\dif r\,\Big)
\eeao
for $i,j=1,\ldots,s_1$.
Then, $Z_1(\vt^0)$  is almost surely positive definite and % for any sequence $\overline{\vartheta}_n\to\vartheta^0$
%with $\overline{\vartheta}_{n,1}\in N_{n,\gamma}(\vartheta_{1}^0,\delta)$,
$$n^{-1}\nabla_{\vt_1}^2\mathcal{L}_{n}^{(h)}(\vt^0)\cid Z_1(\vt^0).$$
 % Moreover, $n^{-1}\partial^1_{i,j}\mathcal{L}_{n}^{(h)}(\vt)$  satisfies condition (SE)  for $0\leq\gamma<1$ .
\end{proposition}
\begin{proof} First, we prove the asymptotic behavior of the score vector and then, in the next step, that the limit is almost surely positive definite.\\
\textbf{Step 1:}
The first term $\frac{1}{n}I_{n,1}$ in \eqref{eqDerDerQ} converges to zero due to the additional normalizing rate of $n^{-1}$.
Due to Proposition~\ref{PropUniformConvRes1}~(a,c) we have for $j=2,\ldots,7$ that $I_{n,j}=\mathcal{O}_p(1)$
(see  exemplarily \eqref{5.2} for $I_{n,5}$) and hence,
$\frac{1}{n}\sum_{j=2}^7I_{n,j}$ converges in probability to zero.
 To summarize,
\begin{align*}
n^{-1}\partial^1_{i,j}\mathcal{L}_{n}^{(h)}({{\vartheta}^0})&=\frac{1}{n}I_{n,8}+o_p(1)=2\cdot\tr\bigg(\big(V^{(h)}_{\vt^0}\big)^{-1}\frac{1}{n^2} \sum_{k=1}^n\partial^1_{i}\varepsilon_k^{(h)}({{\vartheta}^0})\partial^1_{j}\varepsilon_k^{(h)}({{\vartheta}^0})^\mathsf{T}\bigg)+o_p(1).
\intertext{
Due to  Lemma~\ref{LemDerInnovaSeqProp} and Proposition~\ref{PropUniformConvRes1}~(a,c)  we receive
}
n^{-1}\partial^1_{i,j}\mathcal{L}_{n}^{(h)}({{\vartheta}^0})&=2\cdot\tr\bigg(\big(V^{(h)}_{\vt^0}\big)^{-1}\frac{1}{n^2} \sum_{k=1}^n\partial^1_{i}\Pi({{\vartheta}^0})Y_{k-1}^{(h)}Y_{k-1}^{(h)\mathsf{T}}\partial^1_{j}\Pi({{\vartheta}^0})^\mathsf{T}\bigg)+o_p(1).
\intertext{Then, Proposition~\ref{PropLimitResultsFuncofDYY}(b) and the continuous mapping theorem result in}
n^{-1}\partial^1_{i,j}\mathcal{L}_{n}^{(h)}({{\vartheta}^0})&\cid 2\cdot\tr\bigg( \big(V^{(h)}_{\vt^0}\big)^{-1}\partial^1_i\Pi(\vartheta^0) \int_0^1 W_1(r)W_1(r)^\mathsf{T}\,\dif r\,\left(\partial^1_j\Pi(\vartheta^0) \right)^\mathsf{T}\bigg).
\end{align*}
In particular, we have also the joint convergence of the partial derivatives.\\[2mm]
\textbf{Step 2:}
%Recall the representation of the $i,j$-th component $Z_1$ given by
%\begin{align*}
%2\cdot\tr\left( \big(V^{(h)}_{\vt^0}\big)^{-\frac12}\partial^1_i\Pi(\vartheta^0)C_1B_1 \int_0^1 W(r)W(r)^\mathsf{T}\,\dif r\,\left(\partial^1_j\Pi(\vartheta^0)C_1B_1 \right)^\mathsf{T}\big(V^{(h)}_{\vt^0}\big)^{-\frac12}\right).
%\end{align*}
%We ignore the prefactor 2 in \eqref{Z1} in the following and transform the matrix $Z_1(\vt^0)$.
%Let $(W_3(r))_{0\leq r \leq 1}$ be a Brownian motion with covariance matrix $\Sigma_L$ so that $(W_1(r))_{0\leq r\leq 1}\eqd (C_1B_1W_3(r))_{0\leq r\leq 1}$.
Note that we can take $W_1=C_1B_1W_3$ and define $M:=B_1\int_0^1 W_3(r)W_3(r)^\mathsf{T}\,\dif r B_1^\mathsf{T}$, which is a $\P$-a.s. positive definite $c\times c$ matrix. We apply the Cholesky decomposition $M=M_*M_*^\mathsf{T}$. By using properties of the $\text{vec}$ operator and the Kronecker product (see Bernstein~\cite[Chapter 7.1]{Bernstein2009}) we have
\begin{align} \label{KronVecMat}
\notag
[Z_1(\vt^0)]_{i,j}&=2\tr\Big(\big(V^{(h)}_{\vt^0}\big)^{-\frac12}\partial^1_i\Pi(\vartheta^0) C_1 M  \Big(\big(V^{(h)}_{\vt^0}\big)^{-\frac12}\partial^1_j\Pi(\vartheta^0)C_1\Big)^\mathsf{T}\Big)
\\
\notag
&=2\text{vec}\Big(\big(V^{(h)}_{\vt^0}\big)^{-\frac12}\alpha(\vt^0)\partial^1_iC_{1,\vartheta_1^0}^{\perp\mathsf{T}} C_1 M_*\Big)^\mathsf{T}\text{vec}\Big(\big(V^{(h)}_{\vt^0}\big)^{-\frac12}\alpha(\vt^0)\partial^1_j C_{1,\vartheta_1^0}^{\perp\mathsf{T}} C_1 M_*\Big)
\\
%\notag
%&=2\text{vec}\left(\partial^1_i C_{1,\vartheta_1^0}^{\perp\mathsf{T}} C_1\right)^\mathsf{T} \Big(M_*\otimes %\left(\alpha(\vt^0)^{\mathsf{T}}\big(V^{(h)}_{\vt^0}\big)^{-\frac12}\right)\Big)
%\Big(M_*^\mathsf{T}\otimes \left(\big(V^{(h)}_{\vt^0}\big)^{-\frac12}\alpha(\vt^0)\right) \Big)\text{vec}\left(\partial^1_j C_{1,\vartheta_1^0}^{\perp\mathsf{T}} C_1\right)
%\\
&=2\text{vec}\left(\partial^1_i C_{1,\vartheta_1^0}^{\perp\mathsf{T}} C_1\right)^\mathsf{T} \Big(M\otimes \left(\alpha(\vt^0)^{\mathsf{T}}\big(V^{(h)}_{\vt^0}\big)^{-1}\alpha(\vt^0)\right) \Big)
\text{vec}\left(\partial^1_j C_{1,\vartheta_1^0}^{\perp\mathsf{T}} C_1\right).
\end{align}
Furthermore,  $\rank\big(M\otimes \left(\alpha(\vt^0)^{\mathsf{T}}\big(V^{(h)}_{\vt^0}\big)^{-1}\alpha(\vt^0)\right)\big)=\rank(M)\cdot\rank\big(\alpha(\vt^0)^{\mathsf{T}}\big(V^{(h)}_{\vt^0}\big)^{-1}\alpha(\vt^0)\big)$ due to Bernstein~\cite[Fact 7.4.23]{Bernstein2009} and thus, $M\otimes \left(\alpha(\vt^0)^{\mathsf{T}}\big(V^{(h)}_{\vt^0}\big)^{-1}\alpha(\vt^0)\right)$  has full rank $c\cdot (d-c)$ a.s.
Now, if we consider the Hessian matrix $Z_1(\vt^0)$, we have %with  $v_i:=\text{vec}\left(\partial^1_i \Pi(\vartheta_{1}^0)^\mathsf{T} C_1\right)$
\begin{align*}
Z_1(\vt^0)=
%\begin{pmatrix}
%v_1^\mathsf{T} \left(M\otimes \big(V^{(h)}_{\vt^0}\big)^{-1} \right)  v_1 & \cdots & v_1^\mathsf{T}\left(M\otimes \big(V^{(h)}_{\vt^0}\big)^{-1} \right) v_{s_1}
%\\
%\vdots & & \vdots
%\\
%v_{s_1}^\mathsf{T}\left(M\otimes \big(V^{(h)}_{\vt^0}\big)^{-1} \right)  v_1 & \cdots & v_{s_1}^\mathsf{T} \left(M\otimes \big(V^{(h)}_{\vt^0}\big)^{-1} \right) v_{s_1}
% \end{pmatrix}
%\\& =
%\begin{pmatrix}
%v_1 & \cdots &  v_{s_1}
% \end{pmatrix}^\mathsf{T}\left(M\otimes \big(V^{(h)}_{\vt^0}\big)^{-1} \right) \begin{pmatrix}
%v_1 & \cdots &  v_{s_1}
% \end{pmatrix}
%\\
2\left[\nabla_{\vartheta_1} \left(C_{1,\vartheta_1}^{\perp\mathsf{T}} C_1\right)^\mathsf{T}\right]_{\vt_1=\vt_1^0} \left(M\otimes \left(\alpha(\vt^0)^{\mathsf{T}}\big(V^{(h)}_{\vt^0}\big)^{-1} \alpha(\vt^0) \right)\right)\left[\nabla_{\vartheta_1} \left(C_{1,\vartheta_1}^{\perp\mathsf{T}} C_1\right)\right]_{\vt_1=\vt_1^0}.
\end{align*}
%Thus, $Z_1(\vt^0)$ is obviously positive semi-definite.
Due to \autoref{AssMPosDefN} the $((d-c)c\times s_1)$-dimensional matrix $\nabla_{\vartheta_1}\left( C_{1,\vartheta_1^0}^{\perp\mathsf{T}} C_1\right)$
is of full column rank and hence, the product has full rank $s_1$. Therefore, we have  the positive definiteness almost surely.
\end{proof}

\subsubsection{Asymptotic mixed normality of the long-run QML estimator}

We are able now to show the weak convergence of the long-run QML estimator and thus, we have one main result.
\begin{theorem}
\label{thmAsymDisEstimators1}
Let \autoref{AssMPosDefN} additionally hold.
 Then, we have as $n\to\infty$
\begin{align*}
%\label{eqAsymDisTheta1}
n(\widehat{\vartheta}_{n,1}-\vartheta_{1}^0)
&\cid
-Z_1(\vt^0)^{-1}\cdot \mathcal{J}_1(\vt^0),
\end{align*}
where $\mathcal{J}_1(\vartheta^0)$ is defined as in Proposition~\ref{PropConvergScoreVec2} and
$Z_1(\vt^0)$ as in Proposition~\ref{PropConvergHessianMat2}, respectively.

\end{theorem}

\begin{proof}
From \eqref{eqTaylorExpansionDer1} we know that
\begin{align}
\label{eqTaylorExpansionDer12}
0_{s_1}=\nabla_{\vt_1} \mathcal{\widehat{L}}_n^{(h)}(\vt_1^0,\widehat\vt_{n,2})+n^{-1}\nabla_{\vt_1}^2 \underline{\mathcal{\widehat{L}}}_n^{(h)}(\underline{\vt}_{n,1},\widehat{\vt}_{n,2})n(\widehat{\vt}_{n,1}-\vt_1^0).
\end{align}
In Proposition~\ref{PropConvergScoreVec2} we already derived the asymptotic behavior of the score vector $\nabla_{\vt_1} \mathcal{L}_n^{(h)}(\vt_1^0,\vt_{2}^0)$ and in Proposition~\ref{PropConvergHessianMat2}
the asymptotic behavior of the Hessian matrix $n^{-1}\nabla_{\vt_1}^2 \mathcal{L}_n^{(h)}({\vt}_{1}^0,{\vt}_{2}^0)$.
However, for the proof of \Cref{thmAsymDisEstimators1} we require now  the asymptotic behavior of
$\nabla_{\vt_1} \mathcal{L}_n^{(h)}(\vt_1^0,\widehat\vt_{n,2})$ and $n^{-1}\nabla_{\vt_1}^2 \underline{\mathcal{L}}_n^{(h)}(\underline{\vt}_{n,1},\widehat{\vt}_{n,2})$.
 Therefore, we use a local stochastic equicontinuity condition on the family
 $\nabla_{\vt_1}\mathcal{L}_n^{(h)}(\vartheta_1^0,\cdot)$ in $\vt_2^0$ ($n\in\N$) and on the family
 $n^{-1}\nabla^2_{\vt_1}\mathcal{L}_n^{(h)}(\cdot)$ in $\vt^0$ ($n\in\N$).  %similarly to \cite[Proposition 3.1(ii)]{Saikkonen1995}.

\begin{lemma} \label{Proposition 6.1}
%Let $N_{n,\gamma}(\vartheta_{1}^0,\delta_1)$ be defined as in \eqref{N} and $\mathcal{B}(\vartheta_{2}^0,\delta_2)$ be defined as in \Cref{ball}.
For every $\tau>0$ and every $\eta>0$, there exist an integer $n(\tau,\eta)$ and  real numbers $\delta_1,\delta_2>0$ such that
for $\frac{1}{2}<\gamma<1$,
\begin{itemize}
%\item[(a)]
% $ \P\bigg(\sup_{\vartheta\in N_{n,\gamma}(\vartheta_{1}^0,\delta_1)\times \mathcal{B}(\vartheta_{2}^0,\delta_2)}|\mathcal{L}_n^{(h)}(\vartheta)-\mathcal{L}_n^{(h)}(\vartheta^0)|>\tau\bigg)\leq\eta \quad \text{ for } n\geq n(\tau,\eta);$
\item[(a)]
 $ \P\bigg(\sup_{\vartheta_2\in  \mathcal{B}(\vartheta_{2}^0,\delta_2)}\|\nabla_{\vt_1}\mathcal{L}_n^{(h)}(\vartheta_1^0,\vt_2)-\nabla_{\vt_1}\mathcal{L}_n^{(h)}(\vartheta^0_1,\vartheta_2^0)\|>\tau\bigg)\leq\eta \quad \text{ for } n\geq n(\tau,\eta),$
\item[(b)] $ \P\bigg(\sup_{\vartheta\in N_{n,\gamma}(\vartheta_{1}^0,\delta_1)\times \mathcal{B}(\vartheta_{2}^0,\delta_2)}\|n^{-1}\nabla^2_{\vt_1}\mathcal{L}_n^{(h)}(\vartheta)-n^{-1}\nabla^2_{\vt_1}\mathcal{L}_n^{(h)}(\vartheta^0)\|>\tau\bigg)\leq\eta \quad \text{ for } n\geq n(\tau,\eta).$
\end{itemize}
\end{lemma}
    The stochastic equicontinuity conditions $SE$ and  $SE_o$ in Saikkonen~\cite{Saikkonen1995} are global conditions where Lemma~\ref{thmAsymDisEstimators1} is weaker and
    presents only  a local
    stochastic equicontinuity condition for the standardized score  in $\vt^0_2$ and for the standardized Hessian matrix in $\vt^0$.

\begin{proof}[Proof of Lemma~\ref{Proposition 6.1}]
 $\mbox{}$\\
 (a) \,Note that on the one hand, $\nabla_{\vt_1}\mathcal{L}_{n,1,1}^{(h)}(\vartheta_1^0,\vt_2)=0$ since $\varepsilon_{n,1,1}^{(h)}(\vt_1^0,\vt_2)=0$
 and on the other hand,  $\nabla_{\vt_1}\mathcal{L}_{n,2}^{(h)}(\vt_2)=0$ clearly. Hence,
 \beao
   \lefteqn{\sup_{\vartheta_2\in  \mathcal{B}(\vartheta_{2}^0,\delta_2)} \|\nabla_{\vt_1}\mathcal{L}_n^{(h)}(\vartheta_1^0,\vt_2)-\nabla_{\vt_1}\mathcal{L}_n^{(h)}(\vartheta^0_1,\vartheta_2^0)\|}\\
   &&=\sup_{\vartheta_2\in  \mathcal{B}(\vartheta_{2}^0,\delta_2)} \|\nabla_{\vt_1}\mathcal{L}_{n,1,2}^{(h)}(\vartheta_1^0,\vt_2)-\nabla_{\vt_1}\mathcal{L}_{n,1,2}^{(h)}(\vartheta^0_1,\vartheta_2^0)\|.
 \eeao
 We can conclude with similar calculations as in Lemma~\ref{Lemma 3.5a}
applying \eqref{EQ3} and \eqref{EQ5} that
 \beao
    \sup_{\vartheta_2\in  \mathcal{B}(\vartheta_{2}^0,\delta_2)} \|\nabla_{\vt_1}\mathcal{L}_{n,1,2}^{(h)}(\vartheta_1^0,\vt_2)-\nabla_{\vt_1}\mathcal{L}_{n,1,2}^{(h)}(\vartheta^0_1,\vartheta_2^0)\|\leq \sup_{\vartheta_2\in  \mathcal{B}(\vartheta_{2}^0,\delta_2)} \mathfrak{C}\|\vt_2-\vt_2^0\|U_n\leq \mathfrak{C}\delta_2U_n.
 \eeao
 Since $U_n=\mathcal{O}_p(1)$ due to Lemma~\ref{Lemma 3.5a} we obtain the statement.\\
 (b) \, Due to $\nabla^2_{\vt_1}\mathcal{L}_{n,2}^{(h)}(\vt_2)=0$  we have
 \begin{align*}
    & \hspace*{-2cm} \sup_{\vartheta\in N_{n,\gamma}(\vartheta_{1}^0,\delta_1)\times \mathcal{B}(\vartheta_{2}^0,\delta_2)}\|n^{-1}\nabla^2_{\vt_1}\mathcal{L}_n^{(h)}(\vartheta)-n^{-1}\nabla^2_{\vt_1}\mathcal{L}_n^{(h)}(\vartheta^0)\|\\
    &\leq  \sup_{\vartheta\in N_{n,\gamma}(\vartheta_{1}^0,\delta_1)\times \mathcal{B}(\vartheta_{2}^0,\delta_2)}\|n^{-1}\nabla^2_{\vt_1}\mathcal{L}_{n,1,1}^{(h)}(\vartheta)-n^{-1}\nabla^2_{\vt_1}\mathcal{L}_{n,1,1}^{(h)}(\vartheta^0)\|\\
    & \quad\quad +\sup_{\vartheta\in N_{n,\gamma}(\vartheta_{1}^0,\delta_1)\times \mathcal{B}(\vartheta_{2}^0,\delta_2)}\|n^{-1}\nabla^2_{\vt_1}\mathcal{L}_{n,1,2}^{(h)}(\vartheta)-n^{-1}\nabla^2_{\vt_1}\mathcal{L}_{n,1,2}^{(h)}(\vartheta^0)\|.
\intertext{Then, the first term is bounded by \eqref{EQ1} and the second term by \eqref{EQ3} and \eqref{EQ5}, respectively. Hence, }
    & \leq \, \sup_{\vartheta\in N_{n,\gamma}(\vartheta_{1}^0,\delta_1)\times \mathcal{B}(\vartheta_{2}^0,\delta_2)}\left(\mathfrak{C}\|\vt-\vt^0\|\left\|\frac{1}{n^2}\sum_{k=1}^nL_{k-1}^{(h)}[L_{k-1}^{(h)}]^{\mathsf{T}}\right\| + \frac{1}{n}\mathfrak{C}\|\vt-\vt^0\|U_n\right)\\
    & \leq \, \mathfrak{C}\delta_2\left\|\frac{1}{n^2}\sum_{k=1}^nL_{k-1}^{(h)}[L_{k-1}^{(h)}]^{\mathsf{T}}\right\| + \frac{1}{n}\mathfrak{C}\delta_2U_n.
 \end{align*}
 Since $U_n=\mathcal{O}_p(1)$ due to Lemma~\ref{Lemma 3.5a} and $\frac{1}{n^2}\sum_{k=1}^nL_{k-1}^{(h)}[L_{k-1}^{(h)}]^{\mathsf{T}}=\mathcal{O}_p(1)$ due to
 Proposition~\ref{PropUniformConvRes1}(b), statement (b) follows.
\end{proof}

The weak convergence of  $\nabla_{\vt_1} \mathcal{\widehat{L}}_n^{(h)}(\vt_1^0,\widehat\vt_{n,2})$ to $\mathcal{J}_1(\vt^0)$  follows
then by  Proposition~\ref{Lemma 2.8}, Proposition~\ref{PropConvergScoreVec2} and  Lemma~\ref{Proposition 6.1}(a). Due to   Proposition~\ref{Lemma 2.8}, Proposition~\ref{PropConvergHessianMat2} and Lemma~\ref{Proposition 6.1}(b) we have that $n^{-1}\nabla_{\vt_1}^2 \underline{\mathcal{\widehat{L}}}_n^{(h)}(\underline{\vt}_{n,1},\widehat{\vt}_{n,2})$ converges weakly to the random matrix $Z_1(\vt^0)$.
In particular, Proposition~\ref{PropLimitResultsFuncofDYY} also guarantees the joint convergence of both terms. Finally,  the almost sure positive definiteness of $Z_1(\vt^0)$ allows us to take the inverse and reorder \eqref{eqTaylorExpansionDer1}  so that
\begin{align*}
%\label{eqAsymNormConv}
n(\widehat{\vt}_{n,1}-\vt^0_1)&=-\Big(n^{-1}\nabla_{\vt_1}^2 \underline{\mathcal{\widehat{L}}}_n^{(h)}(\underline{\vt}_{n,1},\widehat\vt_{n,2})\Big)^{-1}\nabla_{\vt_1} \mathcal{\widehat{L}}_n^{(h)}(\vt_1^0,\widehat\vt_{n,2})\cid
-Z_1(\vt^0)^{-1}\cdot \mathcal{J}_1(\vt^0).
\end{align*}
%$\mbox{}$
\end{proof}

\subsection{Asymptotic distribution of the short-run parameter estimator} \label{Sec.NormalShort}
Lastly, we derive the asymptotic normality of the short-run QML estimator $\widehat\vt_{n,2}$ which we prove by using a Taylor-expansion of the QML-function
similarly as in \Cref{sec:5:sub:1}. Before we start the proof we want to derive some mixing property of the process $(Y_{st,k}^{(h)},\Delta Y^{(h)}_k)_{k\in\Z}$
because this will be used throughout this section.

\begin{lemma} \label{Lemma 4.5}
    The process $(Y_{st,k}^{(h)},\Delta Y^{(h)}_k)_{k\in\Z}$ is strongly mixing with mixing coefficients \linebreak $\alpha_{\Delta Y^{(h)},Y_{st}^{(h)}}(l)\leq \mathfrak{C}\rho^l$
    for some $0<\rho<1$. In particular, for any $\delta>0$,
    \beao
         \sum_{l=1}^\infty \alpha_{\Delta Y^{(h)},Y_{st}^{(h)}}(l)^{\frac{\delta}{2+\delta}}<\infty.
    \eeao
\end{lemma}
\begin{proof}
Due to \eqref{statespacestationary} the process $Y_{st}^{(h)}$ has the state space representation
\beao
    Y_{st,k}^{(h)}=C_2X_{st,k}^{(h)} \quad\text{ with } \quad X_{st,k}^{(h)}=\e^{A_2 h}X_{st,k-1}^{(h)}+\int_{(k-1)h}^{kh}\e^{A_2(kh-u)}B_2\,dL_u, \quad k\in\N.
\eeao
Masuda~\cite[Theorem 4.3]{Masuda2004} proved that  $(X_{st,k}^{(h)})_{k\in\N}$ is $\beta$-mixing with an exponentially rate since \linebreak $\E\|X_{st,k}^{(h)}\|^2<\infty$.
Having $\E\|\Delta L_k^{(h)}\|^2<\infty$ in mind as well we can conclude on the same way as in Masuda~\cite[Theorem 4.3]{Masuda2004} that the Markov process
\beao
    \left(\begin{array}{cc}
        \Delta L_k^{(h)} \\
        X_{st,k}^{(h)}
    \end{array}\right)=\left(\begin{array}{cc}
        0_{m\times m} & 0_{m\times (N-c)} \\
        0_{(N-c)\times m} & \e^{A_2h}
    \end{array}\right)\left(\begin{array}{cc}
        \Delta L_{k-1}^{(h)} \\
        X_{st,k-1}^{(h)}
    \end{array}\right)+\int_{(k-1)h}^{kh}\left(\begin{array}{cc}
        I_m  \\
        \e^{A_2(kh-u)} B_2
    \end{array}\right)\,dL_u
\eeao
is $\beta$-mixing with mixing coefficient $\beta_{\Delta L^{(h)}, X^{(h)}}(l)\leq \mathfrak{C}\rho_1^l$ for some $0<\rho_1<1$.  Hence, it is as well
$\alpha$-mixing with mixing coefficient  $\alpha_{\Delta L^{(h)}, X^{(h)}}(l)\leq \beta_{\Delta L^{(h)}, X^{(h)}}(l) \leq \mathfrak{C}\rho_1 ^l$.
Finally, it is obvious of the definition of $\alpha$-mixing that
\beao
    \left(\begin{array}{cc} \Delta Y_k^{(h)} \\ Y_{st,k}^{(h)}\end{array}\right)
        =\left(\begin{array}{ccc} C_1B_1 & C_2 & -C_2\\ 0_{d\times m} & C_2 & 0_{d\times N-c}\end{array}\right)
         \left(\begin{array}{cc}
        \Delta L_k^{(h)} \\
        X_{st,k}^{(h)}\\
        X_{st,k-1}^{(h)}
    \end{array}\right)
\eeao
is $\alpha$-mixing with $\alpha_{\Delta Y^{(h)},Y_{st}^{(h)}}(l)\leq \alpha_{\Delta L^{(h)}, X^{(h)}}(l-1)\leq \mathfrak{C}\rho_1^{l-1}$.
\end{proof}

%We will prove this in the following steps:
%\begin{enumerate}
%\item The covariance of the score vector of $\mathcal{L}_{n,2}^{(h)}(\vt_1^0,\vartheta_2)$ converges for every $\vt_2\in\Theta_2$ (\Cref{LemmaConvergScoreVecAuxRes}).
%\item
%    The gradient of the log-likelihood function with respect to the short-run parameter at the true parameter $\vt_2^0$ is asymptotically normal (\Cref{PropConvergScoreVec1}).
%\item The Hessian matrix of $\mathcal{L}_{n}^{(h)}$ converges continuously weakly to $Z_{st}$, which is almost surely positive definite (\Cref{PropConvergHessianMat}).
%%\item The short-run estimator is asymptotically normal (\Cref{thmAsymDisEstimators}).
%\end{enumerate}
\subsubsection{Asymptotic behavior of the score vector}
First, we prove that the partial derivatives have finite variance.
\begin{lemma}
\label{LemFinVarianceQ2}
 For any $\vartheta_2\in\Theta_2$, $n\in\N$, and $i=1,\ldots,s_2$ the random variable $\E|\partial_i^{st}\mathcal{L}_{n}^{(h)}(\vartheta^0)|^2<\infty$.
\end{lemma}
\begin{proof}
We have due to  Lemma~\ref{Lemma 2.5}~(b) and the Cauchy-Schwarz inequality
\begin{align*}
&\E\Big\vert-\tr\Big(\big(V_{\vt^0}^{(h)}\big)^{-1} \varepsilon_{k}^{(h)}(\vt^0)\varepsilon_{k}^{(h)}(\vt^0)^\mathsf{T} \big(V_{\vt^0}^{(h)}\big)^{-1}\partial_{i}^{st} V_{\vt^0}^{(h)} \Big)
+2\cdot\big(\partial_i^{st}\varepsilon_{k}^{(h)}(\vt^0)^\mathsf{T}\big) \big(V_{\vt^0}^{(h)}\big)^{-1} \varepsilon_{k}^{(h)}(\vt^0) \Big\vert^2
 \\
    &\leq
       \mathfrak{C}\cdot\E\Vert \varepsilon_{k}^{(h)}(\vartheta^0) \Vert^4+  \mathfrak{C} \cdot \Big(\E\Vert \varepsilon_{k}^{(h)}(\vartheta^0) \Vert^4\E\Vert \partial_{i}^{st}\varepsilon_{k}^{(h)}(\vartheta^0) \Vert^4\Big)^{\frac12}<\infty,
\end{align*}
so that the statement follows with \eqref{eqDerQ}.
\end{proof}
Now we can prove the convergence of the covariance matrix of the score vector where we plug in the true parameter.
\begin{lemma}
\label{LemmaConvergScoreVecAuxRes}
We have for all $\vartheta_2\in\Theta_2$ and $\ell_{k,2}^{(h)}(\vartheta_{2}):=\varepsilon_{k,2}^{(h)}(\vartheta_2)^\mathsf{T} \big(V_{\vt_1^0,\vt_2}^{(h)}\big)^{-1}\varepsilon_{k,2}^{(h)}(\vartheta_2)$ that
\begin{align}
%\label{eqConvergScoreVecAuxRes}
\lim_{n\to\infty}\Var\big(\nabla_{\vartheta_2}\mathcal{L}_{n}^{(h)}(\vartheta^0)\big)= \Big[\sum_{l\in\Z}\Cov\left(\partial^{st}_i \ell_{1,2}^{(h)}(\vartheta_{2}^0),\partial^{st}_j\ell_{1+l,2}^{(h)}(\vartheta_{2}^0)\right)\Big]_{1\leq i,j\leq s_2}=:I(\vartheta_2^0).
\label{eqConvergScoreVecMat}
\end{align}
\end{lemma}

\begin{proof}
We can derive the result in a similar way as in Schlemm and Stelzer \cite[Lemma 2.14]{SchlemmStelzer2012a}. Hence, we only sketch the proof to show the differences.  A detailed proof can be found in Scholz \cite[Section 5.9]{Scholz}.
It is sufficient to show that for all $\vartheta_2\in\Theta_2$ and all $i,j=1,\ldots,s_2$ the sequence
\begin{align}
\label{eqDefInij}
[\Var\big(\nabla_{\vartheta_2}\mathcal{L}_{n}^{(h)}(\vartheta^0)\big)]_{i,j}=n^{-1}\sum^n_{k_1=1} \sum^n_{k_2=1} \Cov\left(\partial^{st}_i \ell_{k_1,2}^{(h)}(\vartheta_{2}^0),\partial^{st}_j\ell_{k_2,2}^{(h)}(\vartheta_{2}^0)\right)=:I^{(i,j)}_n(\vartheta_2^0)
\end{align}
converges as $n\to\infty$. By the representation of the partial derivatives in \eqref{eqDerQ} and \eqref{vareps} the sequence
\beao
    \partial^{st}_i \ell_{k,2}^{(h)}(\vartheta_{2}^0)
\hspace*{-0.3cm}&=&\hspace*{-0.3cm}-\tr\left(\big(V_{\vt^0}^{(h)}\big)^{-1} \varepsilon_{k,2}^{(h)}(\vartheta_2^0)\varepsilon_{k,2}^{(h)}(\vartheta_2^0)^\mathsf{T}\big(V_{\vt^0}^{(h)}\big)^{-1}\partial_{i}^{st}V_{\vt^0}^{(h)} \right)
+\big(\partial_{i}^{st}\varepsilon_{k,2}^{(h)}(\vartheta_2^0)^\mathsf{T}\big) \big(V_{\vt^0}^{(h)}\big)^{-1} \varepsilon_{k,2}^{(h)}(\vartheta_2^0)\\
&=&\hspace*{-0.3cm}-\tr\left(\big(V_{\vt^0}^{(h)}\big)^{-1} \varepsilon_{k}^{(h)}(\vartheta^0)\varepsilon_{k}^{(h)}(\vartheta^0)^\mathsf{T}\big(V_{\vt^0}^{(h)}\big)^{-1}\partial_{i}^{st}V_{\vt^0}^{(h)} \right)
+\big(\partial_{i}^{st}\varepsilon_{k}^{(h)}(\vartheta^0)^\mathsf{T}\big) \big(V_{\vt^0}^{(h)}\big)^{-1} \varepsilon_{k}^{(h)}(\vartheta^0)
\eeao
is stationary and the covariance in \eqref{eqDefInij} depends only on the difference $l=k_1-k_2$.
If we can show that
\begin{align}
\label{eqCovToShowAbsSumm}
\sum_{l\in\Z}\left|\Cov\left(\partial^{st}_i \ell_{1,2}^{(h)}(\vartheta_{2}^0),\partial^{st}_j\ell_{1+l,2}^{(h)}(\vartheta_{2}^0)\right)\right|<\infty,
\end{align}
 then the Dominated Convergence Theorem implies
\begin{eqnarray*}
I^{(i,j)}_n(\vartheta_2^0)&=&n^{-1}\sum_{l=-n}^{n}(n-\vert l\vert)\Cov\left(\partial^{st}_i \ell_{1,2}^{(h)}(\vartheta_{2}^0),\partial^{st}_j\ell_{1+l,2}^{(h)}(\vartheta_{2}^0)\right)
\\
&\cin& \sum_{l\in\Z}\Cov\left(\partial^{st}_i \ell_{1,2}^{(h)}(\vartheta_{2}^0),\partial^{st}_j\ell_{1+l,2}^{(h)}(\vartheta_{2}^0)\right).
\end{eqnarray*}
Due to Lemma~\ref{Lemma 4.5} and the uniformly exponentially bound of $(\mathsf{k}_j(\vt))$ and $(\partial_i^{st}\mathsf{k}_j(\vt))$  finding the dominant goes in the same vein as in the proof of
Schlemm and Stelzer~\cite[Lemma 2.14]{SchlemmStelzer2012a}.
\end{proof}
In the following, we derive the convergence of the score vector with respect to the short-run parameters by a truncation argument.
\begin{proposition}
\label{PropConvergScoreVec1}
For the gradient with respect to the short-run parameters the asymptotic behavior
\begin{align*}
\sqrt{n}\cdot \nabla_{\vartheta_2}\mathcal{L}_{n}^{(h)}(\vartheta^0)\cid\mathcal{N}(0,I(\vartheta_2^0))
\end{align*}
holds,
where $I(\vartheta_2^0)$ is the asymptotic covariance matrix given in \eqref{eqConvergScoreVecMat}.
\end{proposition}
\begin{proof}
First, we realize that representation~\eqref{eqDerQ} and Lemma~\ref{Lemma 2.5}~(c,d) result in $\E\big(\nabla_{\vartheta_2}\mathcal{L}_{n}^{(h)}(\vartheta^0)\big)=0_{s_2}$.
%
%The derivatives with respect to the stationary and non-stationary parameters of the pseudo-innovations are given by
%\begin{align*}
%\partial_i^1\varepsilon_k^{(h)}(\vartheta)
%&=\partial_i^1{k}(B,\vartheta)\Delta Y_k^{(h)}-\partial_i^1\Pi(\vartheta)Y_{k-1}^{(h)},\quad\text{  for } i\in\{1,\ldots,s_1\}
%\\
%\text{and }\quad\partial_i^{st}\varepsilon_k^{(h)}(\vartheta)
%&=\partial_i^{st}{k}(B,\vartheta)\Delta Y_k^{(h)}-\partial_i^{st}\Pi(\vartheta)Y_{k-1}^{(h)},\quad\text{ for } i\in\{1,\ldots,s_2\}.
%\end{align*}
%Note that $\partial_i^{st}\varepsilon_k^{(h)}(\vartheta^0)$ is stationary due to the fact that $\partial_i^{st}\Pi(\vartheta^0)Y^{(h)}=\left(\partial_i^{st}\alpha(\vt^0)\right)\beta(\vt_1^0)^\mathsf{T}Y_{st}^{(h)}$ and  $\partial_i^{1}\varepsilon_k^{(h)}(\vartheta^0)$ is non-stationary since
%$\partial_i^{1}\Pi(\vartheta^0)Y^{(h)}=\left(\partial_i^{1}\alpha(\vt^0)\right)\beta(\vt_1^0)^\mathsf{T}Y_{st}^{(h)}+\alpha(\vt^0)\partial_i^{1}\beta(\vt_1^0)^\mathsf{T}Y^{(h)}$.
%The first summand still cancels out the non-stationarity in contrast to the second one, as $\partial_i^{1}\beta(\vt_1^0)^\mathsf{T}$ does in general not lie in the cointegration space.
%
Due to \eqref{vareps} we can rewrite \eqref{eqDerQ} for $\mathsf{M}\in\N$ as
\begin{align}
  \partial_{i}^{st}\mathcal{L}_n^{(h)}(\vartheta^0)=&\frac{1}{n}\sum_{k=1}^n \left(Y_{\mathsf{M},k}^{(i)}-\E Y_{\mathsf{M},k}^{(i)}\right)+ \frac{1}{n}\sum_{k=1}^n \left(Z_{\mathsf{M},k}^{(i)}-\E Z_{\mathsf{M},k}^{(i)}\right),
  \label{Verweis1}
\\\text{where}\quad
    Y_{\mathsf{M},k}^{(i)}:=&   \tr\left(\big(V_{\vt^0}^{(h)}\big)^{-1} \partial_{i}^{st}V_{\vt^0}^{(h)}\right)
    -
  \tr\left(\big(V_{\vt^0}^{(h)}\big)^{-1} \Pi(\vt^0)Y_{st,k-1}^{(h)}Y_{st,k-1}^{(h)\mathsf{T}} \Pi(\vt^0)^\mathsf{T} \big(V_{\vt^0}^{(h)}\big)^{-1}\partial_{i}^{st}V_{\vt^0}^{(h)} \right)
    \notag
  \\
  &
    +
  \sum_{\iota_1=0}^\mathsf{M}
  \tr\left(\big(V_{\vt^0}^{(h)}\big)^{-1} {\mathsf{k}}_{\iota_1}(\vt^0) \Delta Y_{k-\iota_1}^{(h)} Y_{st,k-1}^{(h)\mathsf{T}} \Pi(\vt^0)^\mathsf{T} \big(V_{\vt^0}^{(h)}\big)^{-1}\partial_{i}^{st}V_{\vt^0}^{(h)} \right)
  \notag
    \\
  &
    +
  \sum_{\iota_2=0}^\mathsf{M}
  \tr\left(\big(V_{\vt^0}^{(h)}\big)^{-1} \Pi(\vartheta^0)Y_{st,k-1}^{(h)}  \Delta Y_{k-\iota_2}^{(h)\mathsf{T}} {\mathsf{k}}_{\iota_2}(\vt^0)^\mathsf{T} \big(V_{\vt^0}^{(h)}\big)^{-1}\partial_{i}^{st}V_{\vt^0}^{(h)} \right)
  \notag
  \\
  &
  -
  \sum_{\iota_1,\iota_2=0}^\mathsf{M}
  \tr\left(\big(V_{\vt^0}^{(h)}\big)^{-1}
  {\mathsf{k}}_{\iota_1}(\vt^0) \Delta Y_{k-\iota_1}^{(h)}
 \Delta Y_{k-\iota_2}^{(h)\mathsf{T}} {\mathsf{k}}_{\iota_2}(\vt^0)^\mathsf{T} \big(V_{\vt^0}^{(h)}\big)^{-1}\partial_{i}^{st}V_{\vt^0}^{(h)} \right)
  \notag
 %
 %%%%%%%%%%%
 %
  \\
  \notag
  &
   + 2\cdot\tr\left(
   \big(\partial_{i}^{st}\Pi(\vt^0)Y_{st,k-1}^{(h)}\big)
    \big(V_{\vt^0}^{(h)}\big)^{-1}
     Y_{st,k-1}^{(h)\mathsf{T}} \Pi(\vt^0)^\mathsf{T}\right)
     \\
     \notag
  &
   -2\cdot\sum_{\iota_1=0}^\mathsf{M} \tr\left(
   \big(\partial_{i}^{st}{\mathsf{k}}_{\iota_1}(\vt^0)  \Delta Y_{k-\iota_1}^{(h)}\big) \big(V_{\vt^0}^{(h)}\big)^{-1}
   Y_{st,k-1}^{(h)\mathsf{T}} \Pi(\vt^0)^\mathsf{T} \right)
     \\
     \notag
  &
   -2\cdot\sum_{\iota_2=0}^\mathsf{M}\tr\left(
   \big(\partial_{i}^{st}\Pi(\vt^0)Y_{st,k-1}^{(h)}\big) \big(V_{\vt^0}^{(h)}\big)^{-1}
   \Delta Y_{k-\iota_2}^{(h)\mathsf{T}}{\mathsf{k}}_{\iota_2}(\vt^0)^\mathsf{T}\right)
     \\
  &
   +2\cdot\sum_{\iota_1,\iota_2=0}^\mathsf{M} \tr\left(
   \big(\partial_{i}^{st}{\mathsf{k}}_{\iota_1}(\vt^0) \Delta Y_{k-\iota_1}^{(h)} \big) \big(V_{\vt^0}^{(h)}\big)^{-1}
   \Delta Y_{k-\iota_2}^{(h)\mathsf{T}}{\mathsf{k}}_{\iota_2}(\vt^0)^\mathsf{T}\right),
          %\label{eqANDefY}
          %
          %%%%%%%%
          %
             \notag
   \\
  %\label{eqANDefZ}
   Z_{\mathsf{M},k}^{(i)} :=&V_{\mathsf{M},k}^{(i)}+U_{\mathsf{M},k}^{(i)},
      \notag
  %
  %%%%
  %
   \\\text{and}\quad
     %
  %%%%
  %
  \notag
    V_{\mathsf{M},k}^{(i)}:= &
  \sum_{\iota_1=\mathsf{M}+1}^\infty
  \tr\left(\big(V_{\vt^0}^{(h)}\big)^{-1} {\mathsf{k}}_{\iota_1}(\vt^0) \Delta Y_{k-\iota_1}^{(h)} Y_{st,k-1}^{(h)\mathsf{T}} \Pi(\vt^0)^\mathsf{T} \big(V_{\vt^0}^{(h)}\big)^{-1}\partial_{i}^{st}V_{\vt^0}^{(h)} \right)
  \notag
    \\
     \notag
    &- \sum_{\iota_1=\mathsf{M}+1}^\infty \sum_{\iota_2=0}^\infty
  \tr\left(\big(V_{\vt^0}^{(h)}\big)^{-1}
  {\mathsf{k}}_{\iota_1}(\vt^0) \Delta Y_{k-\iota_1}^{(h)}
 \Delta Y_{k-\iota_2}^{(h)\mathsf{T}} {\mathsf{k}}_{\iota_2}(\vt^0)^\mathsf{T} \big(V_{\vt^0}^{(h)}\big)^{-1}\partial_{i}^{st}V_{\vt^0}^{(h)} \right)
     \\
     \notag
       &
   -2\cdot\sum_{\iota_1=\mathsf{M}+1}^\infty \tr\left(
   \big(\partial_{i}^{st}{\mathsf{k}}_{\iota_1}(\vt^0)  \Delta Y_{k-\iota_1}^{(h)}\big) \big(V_{\vt^0}^{(h)}\big)^{-1}
   Y_{st,k-1}^{(h)\mathsf{T}} \Pi(\vt^0)^\mathsf{T} \right)
     \\
     \notag
  &
   +2\cdot  \sum_{\iota_1=\mathsf{M}+1}^\infty \sum_{\iota_2=0}^\infty
   \tr\left(
   \big(\partial_{i}^{st}{\mathsf{k}}_{\iota_1}(\vt^0) \Delta Y_{k-\iota_1}^{(h)} \big) \big(V_{\vt^0}^{(h)}\big)^{-1}
   \Delta Y_{k-\iota_2}^{(h)\mathsf{T}}{\mathsf{k}}_{\iota_2}(\vt^0)^\mathsf{T}\right), \notag\\
    U_{\mathsf{M},k}^{(i)}:=
  &
  \sum_{\iota_2=\mathsf{M}+1}^\infty
  \tr\left(\big(V_{\vt^0}^{(h)}\big)^{-1} \Pi(\vartheta)Y_{st,k-1}^{(h)}  \Delta Y_{k-\iota_2}^{(h)\mathsf{T}} {\mathsf{k}}_{\iota_2}(\vt^0)^\mathsf{T} \big(V_{\vt^0}^{(h)}\big)^{-1}\partial_{i}^{st}V_{\vt^0}^{(h)} \right)
  \notag
  \\
  &
  -
  \sum_{\iota_1=0}^\infty \sum_{\iota_2=\mathsf{M}+1}^\infty
  \tr\left(\big(V_{\vt^0}^{(h)}\big)^{-1}
  {\mathsf{k}}_{\iota_1}(\vt^0) \Delta Y_{k-\iota_1}^{(h)}
 \Delta Y_{k-\iota_2}^{(h)\mathsf{T}} {\mathsf{k}}_{\iota_2}(\vt^0)^\mathsf{T} \big(V_{\vt^0}^{(h)}\big)^{-1}\partial_{i}^{st}V_{\vt^0}^{(h)} \right)
  \notag
 %
 %%%%%%%%%%%
 %
  \\
  \notag
  &
   -2\cdot\sum_{\iota_2=\mathsf{M}+1}^\infty\tr\left(
   \big(\partial_{i}^{st}\Pi(\vt^0)Y_{st,k-1}^{(h)}\big) \big(V_{\vt^0}^{(h)}\big)^{-1}
   \Delta Y_{k-\iota_2}^{(h)\mathsf{T}}{\mathsf{k}}_{\iota_2}(\vt^0)^\mathsf{T}\right)
     \\
     \notag
  &
   +2\cdot  \sum_{\iota_1=0}^\infty \sum_{\iota_2=\mathsf{M}+1}^\infty
   \tr\left(
   \big(\partial_{i}^{st}{\mathsf{k}}_{\iota_1}(\vt^0) \Delta Y_{k-\iota_1}^{(h)} \big) \big(V_{\vt^0}^{(h)}\big)^{-1}
   \Delta Y_{k-\iota_2}^{(h)\mathsf{T}}{\mathsf{k}}_{\iota_2}(\vt^0)^\mathsf{T}\right).
  \end{align}
We define
$
\mathcal{Y}_{\mathsf{M},k}:=(Y_{\mathsf{M},k}^{(1)\,\mathsf{T}}, \ldots, Y_{\mathsf{M},k}^{(s_2)\,\mathsf{T}})^\mathsf{T}$ as well as $\mathcal{Z}_{\mathsf{M},k}:=(Z_{\mathsf{M},k}^{(1)\,\mathsf{T}},\ldots, Z_{\mathsf{M},k}^{(s_2)\,\mathsf{T}})^\mathsf{T}$
and use a truncation argument analogous to Schlemm and Stelzer \cite[Lemma 2.16]{SchlemmStelzer2012a}.
The main difference to Schlemm and Stelzer~\cite{SchlemmStelzer2012a} is that in our case the definition of $Y_{\mathsf{M},k}^{(i)}$, $V_{\mathsf{M},k}^{(i)}$ and $U_{\mathsf{M},k}^{(i)}$ are more complex including the two stochastic
processes $\Delta Y^{(h)}$, $Y_{st}^{(h)}$ and additional summands.
We show the claim in three steps.
\newline
\textbf{Step 1:}
The process $\mathcal{Y}_{\mathsf{M},k}$ depends only on $\mathsf{M}+1$ past values of $\Delta Y^{(h)}$ and $Y_{st}^{(h)}$. Hence, it inherits the strong mixing property of $(\Delta Y^{(h)},Y_{st}^{(h)})$  and satisfies  $\alpha_{\mathcal{Y}_{\mathsf{M}}}(l)\leq\alpha_{\Delta Y^{(h)},Y_{st}^{(h)}}(\max\{0,l-\mathsf{M}-1\})$. Thus, by
Lemma~\ref{Lemma 4.5} we have $\sum_{l=1}^{\infty}\left(\alpha_{\mathcal{Y}_{\mathsf{M}}}(l)\right)^{\delta\slash (2+\delta)} <\infty$. Using the Cram\'er-Wold device and the univariate central limit theorem of Ibragimov \cite[Theorem 1.7]{Ibragimov1962} for strongly mixing random variables we obtain
\begin{align*}
\frac{1}{\sqrt{n}}\sum_{k=1}^n \left(\mathcal{Y}_{\mathsf{M},k}-\E \mathcal{Y}_{\mathsf{M},k}\right)\cid \mathcal{N}(0_{s_2},{I}_\mathsf{M}(\vt_2^0))
\end{align*}
as $n\to\infty$ where ${I}_\mathsf{M}(\vt_2^0):=\sum_{l\in\Z}\Cov(\mathcal{Y}_{\mathsf{M},1},\mathcal{Y}_{\mathsf{M},1+l})$.
Next, we need to show that
\begin{align}\label{eqAsymNormCov}
{I}_\mathsf{M}(\vt_2^0)\xrightarrow{\mathsf{M}\to\infty}I(\vartheta_2^0).
\end{align}
 Therefore, we first prove that $\Cov\big(Y_{\mathsf{M},k}^{(i)},Y_{\mathsf{M},k+l}^{(j)}\big) \to\Cov\big( \partial_{i}^{st} \ell_{1,2}^{(h)}(\vt^0),\partial_{j}^{st} \ell_{1+l,2}^{(h)}(\vt^0)\big)$ as $\mathsf{M}\to\infty$.
 Note that the bilinearity property of the covariance operator implies
 \begin{align}
 |\Cov\big(Y_{\mathsf{M},k}^{(i)},&Y_{\mathsf{M},k+l}^{(j)}\big) - \Cov\big( \partial_{i}^{st} \ell_{k,2}^{(h)}(\vt^0),\partial_{j}^{st} \ell_{k+l,2}^{(h)}(\vt^0)\big)|
 \notag
 \\
 &=
  |\Cov\big(Y_{\mathsf{M},k}^{(i)},Y_{\mathsf{M},k+l}^{(j)}-\partial_{j}^{st} \ell_{k+l,2}^{(h)}(\vt^0)\big) + \Cov\big(Y_{\mathsf{M},k}^{(i)}- \partial_{i}^{st} \ell_{k,2}^{(h)}(\vt^0),\partial_{j}^{st} \ell_{k+l,2}^{(h)}(\vt^0)\big)|\notag\\
  &\leq \Var\big(Y_{\mathsf{M},1}^{(i)}\big)^{1/2}\Var\big(Y_{\mathsf{M},1}^{(j)}-\partial_{j}^{st} \ell_{1,2}^{(h)}(\vt^0)\big)^{1/2}\notag\\
  &\quad  +\Var\big(Y_{\mathsf{M},1}^{(i)}- \partial_{i}^{st} \ell_{1,2}^{(h)}(\vt^0)\big)^{1/2}\Var\big(\partial_{j}^{st} \ell_{l,2}^{(h)}(\vt^0)\big)^{1/2},
\label{Verweis2}
\end{align}
where
\begin{align*}
Y_{\mathsf{M},k}^{(i)}- \partial_{i}^{st} \ell_{k,2}^{(h)}(\vt^0)=&-
 \sum_{\iota_1=\mathsf{M}+1}^\infty
  \tr\left(\big(V_{\vt^0}^{(h)}\big)^{-1} {\mathsf{k}}_{\iota_1}(\vt^0) \Delta Y_{k-\iota_1}^{(h)} Y_{st,k-1}^{(h)\mathsf{T}} \Pi(\vt^0)^\mathsf{T} \big(V_{\vt^0}^{(h)}\big)^{-1}\partial_{i}^{st}V_{\vt^0}^{(h)} \right)
  \notag
    \\
  &
    -
  \sum_{\iota_2=\mathsf{M}+1}^\infty
  \tr\left(\big(V_{\vt^0}^{(h)}\big)^{-1} \Pi(\vartheta^0)Y_{st,k-1}^{(h)}  \Delta Y_{k-\iota_2}^{(h)\mathsf{T}} {\mathsf{k}}_{\iota_2}(\vt^0)^\mathsf{T} \big(V_{\vt^0}^{(h)}\big)^{-1}\partial_{i}^{st}V_{\vt^0}^{(h)} \right)
  \notag
  \\
  &
  +
  \sum_{\substack{\iota_1,\iota_2\\\max\{\iota_1,\iota_2\}>\mathsf{M}}}
  \tr\left(\big(V_{\vt^0}^{(h)}\big)^{-1}
  {\mathsf{k}}_{\iota_1}(\vt^0) \Delta Y_{k-\iota_1}^{(h)}
 \Delta Y_{k-\iota_2}^{(h)\mathsf{T}} {\mathsf{k}}_{\iota_2}(\vt^0)^\mathsf{T} \big(V_{\vt^0}^{(h)}\big)^{-1}\partial_{i}^{st}V_{\vt^0}^{(h)} \right)
  \notag
 %
 %%%%%%%%%%%
 %
  \\
  \notag
  &
   +2\cdot\sum_{\iota_1=\mathsf{M}+1}^\infty
   \tr\left(
   \big(\partial_{i}^{st}{\mathsf{k}}_{\iota_1}(\vt^0)  \Delta Y_{k-\iota_1}^{(h)}\big) \big(V_{\vt^0}^{(h)}\big)^{-1}
   Y_{st,k-1}^{(h)\mathsf{T}} \Pi(\vt^0)^\mathsf{T} \right)
     \\
     \notag
  &
   +2\cdot\sum_{\iota_2=\mathsf{M}+1}^\infty
   \tr\left(
   \big(\partial_{i}^{st}\Pi(\vt^0)Y_{st,k-1}^{(h)}\big) \big(V_{\vt^0}^{(h)}\big)^{-1}
   \Delta Y_{k-\iota_2}^{(h)\mathsf{T}}{\mathsf{k}}_{\iota_2}(\vt^0)^\mathsf{T}\right)
     \\
  &
   -2\cdot\sum_{\substack{\iota_1,\iota_2\\\max\{\iota_1,\iota_2\}>\mathsf{M}}}
   \tr\left(
   \big(\partial_{i}^{st}{\mathsf{k}}_{\iota_1}(\vt^0) \Delta Y_{k-\iota_1}^{(h)} \big) \big(V_{\vt^0}^{(h)}\big)^{-1}
   \Delta Y_{k-\iota_2}^{(h)\mathsf{T}}{\mathsf{k}}_{\iota_2}(\vt^0)^\mathsf{T}\right).
\end{align*}
We obtain with the Cauchy-Schwarz inequality, the exponentially decreasing coefficients $({\mathsf{k}}_{j}(\vt^0))_{j\in\N}$ and the finite $4$th-moment of $Y_{st}^{(h)}$
and $\Delta Y^{(h)}$ due to Assumption~\ref{AssMBrownMot} that for some $0<\rho<1$, $$\Var\big(Y_{\mathsf{M},1}^{(i)}- \partial_{i}^{st} \ell_{1,2}^{(h)}(\vt^0)\big)\leq \mathfrak{C} \rho^\mathsf{M}.$$
Moreover, by the proof of Lemma~\ref{LemFinVarianceQ2} we have $\Var\big(\partial_{j}^{st} \ell_{1,2}^{(h)}(\vt^0)\big)<\infty$ and then, \linebreak
 $\Var\big(Y_{\mathsf{M},1}^{(i)})\leq 2\E\big(Y_{\mathsf{M},1}^{(i)}- \partial_{i}^{st} \ell_{1,2}^{(h)}(\vt^0)\big)^2+2\E\big(\partial_{j}^{st} \ell_{1,2}^{(h)}(\vt^0)\big)^2<\infty $ as well. Thus,
\eqref{Verweis2}
converges uniformly in $l$ at an exponential rate to zero as $\mathsf{M}\to\infty$
and
$$\Cov\big(Y_{\mathsf{M},k}^{(i)},Y_{\mathsf{M},k+l}^{(j)}\big) \xrightarrow{\mathsf{M}\to\infty} \Cov\big( \partial_{i}^{st} \ell_{1,2}^{(h)}(\vt^0),\partial_{j}^{st} \ell_{1+l,2}^{(h)}(\vt^0)\big).$$
 Then, the same arguments as in  Schlemm and Stelzer~\cite[Lemma 2.16]{SchlemmStelzer2012a} guarantee that there exists a dominant  (see Scholz~\cite[Section 5.9]{Scholz}) so that dominated convergence results in \eqref{eqAsymNormCov} (see proof of Lemma~\ref{LemmaConvergScoreVecAuxRes}).
\newline
\textbf{Step 2:}
In this step, we show that $\frac{1}{\sqrt{n}}\sum_{k=1}^n \left(\mathcal{Z}_{\mathsf{M},k}-\E \mathcal{Z}_{\mathsf{M},k}\right)$ is asymptotically negligible.
We have
\begin{align}
  &\tr \left(\Var\bigg(\frac{1}{\sqrt{n}}\sum_{k=1}^n \mathcal{Z}_{\mathsf{M},k}\bigg)\right)
 \leq 2 \cdot  \tr \left(\Var \bigg(\frac{1}{\sqrt{n}}\sum_{k=1}^n \mathcal{U}_{\mathsf{M},k}\bigg)\right)+ 2\cdot \tr\left(\Var \bigg(\frac{1}{\sqrt{n}}\sum_{k=1}^n \mathcal{V}_{\mathsf{M},k}\bigg)\right),
  \label{eqIneASNVAR}
\end{align}
where $\mathcal{U}_{\mathsf{M},k}$ and $\mathcal{V}_{\mathsf{M},k}$ are defined in the same vein as $\mathcal{Z}_{\mathsf{M},k}$.
Since both terms can be treated similarly we consider only the first one
\begin{align}
   \tr \left(\Var \bigg(\frac{1}{\sqrt{n}}\sum_{k=1}^n \mathcal{U}_{\mathsf{M},k}\bigg)\right)
   =\frac{1}{n}\sum_{k,k^\prime=1}^n \tr\left(\Cov ( \mathcal{U}_{\mathsf{M},k}, \mathcal{U}_{\mathsf{M},k^\prime})\right)
%\leq \frac{1}{n}\sum_{i,j=1}^{s_2}\sum_{l=-n+1}^{n-1}(n-\vert l\vert)|\Cov(U_{m,k}^{(i)},U_{m,k+l}^{(j)})| \notag \\
 \leq  \sum_{i,j=1}^{s_2}\sum_{l=-\infty}^\infty\vert \Cov(U_{\mathsf{M},1}^{(i)},U_{\mathsf{M},1+l}^{(j)})\vert.
\label{eqIneASNVAR2}
\end{align}
%We find an upper bound for $\vert\Cov(U_{m,k}^{(i)},U_{m,k+l}^{(j)})\vert$ independent of $i$ and $j$ using Davydov inequality (\cite[Lemma 2.13]{SchlemmStelzer2012a}), Cauchy-Schwarz inequality and the same arguments as in the proof of Lemma~\ref{LemmaConvergScoreVecAuxRes}. Namely, we obtain for $l/2>m$
%\begin{align*}
%\vert\Cov(U_{m,1}^{(i)},U_{m,1+l}^{(j)})\vert \leq C \rho^m \left( \left[\alpha_{\Delta Y^{(h)},Y_{st}^{(h)}}\left(\left\lfloor \frac{l}{2} \right\rfloor %\right)\right]^\frac{\delta}{\delta+2} +\rho^\frac{l}{2}\right).
%\end{align*}
With the same arguments as in Schlemm and Stelzer~\cite[Lemma 2.16]{SchlemmStelzer2012a} we obtain independent of $i$ and $j$ the upper bound
\begin{align*}
\sum_{l=0}^{\infty}\vert\Cov(U_{\mathsf{M},1}^{(i)},U_{\mathsf{M},1+l}^{(j)})\vert &\leq
\sum_{l=0}^{2\mathsf{M}}\vert\Cov(U_{\mathsf{M},1}^{(i)},U_{\mathsf{M},1+l}^{(j)})\vert+\sum_{l=2\mathsf{M}+1}^{\infty} \vert\Cov(U_{\mathsf{M},1}^{(i)},U_{\mathsf{M},1+l}^{(j)})\vert\\
&\leq \mathfrak{C} \rho^\mathsf{M} \bigg(\mathsf{M}+\sum_{l=0}^{\infty} \big[\alpha_{\Delta Y^{(h)},Y_{st}^{(h)}}\left( l \right)\big]^\frac{\delta}{\delta+2}\bigg),
\end{align*}
which implies $\tr \left(\Var \left(\frac{1}{\sqrt{n}}\sum_{k=1}^n \mathcal{U}_{\mathsf{M},k}\right)\right)\leq \mathfrak{C} \rho^\mathsf{M}(\mathsf{M}+ \overline{\mathfrak{C}})$, due to \eqref{eqIneASNVAR2}, for some constant $\overline{\mathfrak{C}}>0$.
With the same ideas one obtains an equivalent bound for $\tr \left(\Var \left(\frac{1}{\sqrt{n}}\sum_{k=1}^n \mathcal{V}_{\mathsf{M},k}\right)\right)$ and thus, we have with \eqref{eqIneASNVAR} that
\begin{align} \label{Step 2}
  \tr \left(\Var\bigg(\frac{1}{\sqrt{n}}\sum_{k=1}^n \mathcal{Z}_{\mathsf{M},k}\bigg)\right)\leq \mathfrak{C}\rho^\mathsf{M} (\mathsf{M}+ \overline{\mathfrak{C}}).
\end{align}
\textbf{Step 3:} With the multivariate Chebyshev inequality (see Schlemm \cite[Lemma 3.19]{SchlemmThesis2011}) and \eqref{Step 2} from Step 2 we obtain for every $\tau>0$ that
\begin{align*}
\lim_{\mathsf{M}\to\infty}&\limsup_{n\to\infty}\P\bigg(\bigg\Vert \sqrt{n}\nabla_{\vt_2} \mathcal{L}_{n}^{(h)}(\vartheta^0) - \frac{1}{\sqrt{n}}\sum_{k=1}^{n}\left[\mathcal{Y}_{\mathsf{M},k}-\E \mathcal{Y}_{\mathsf{M},k} \right]\bigg\Vert>\tau\bigg)
\\
&\leq \lim_{\mathsf{M}\to\infty}\limsup_{n\to\infty} \frac{s_2}{\tau^2}\tr\left(\Var\bigg(\frac{1}{\sqrt{n}}\sum_{k=1}^n \mathcal{Z}_{\mathsf{M},k}\bigg)\right)
\leq \lim_{\mathsf{M}\to\infty} \frac{s_2}{\tau^2}\mathfrak{C}\rho^\mathsf{M}(\mathsf{M}+ \overline{\mathfrak{C}})=0.
\end{align*}
 All in all, the results of Step 1 and Step 3 combined with Brockwell and Davis \cite[Proposition 6.3.9]{BrockwellDavis1998} yield the asymptotic normality in Lemma~\ref{PropConvergScoreVec1}.
\end{proof}

\subsubsection{Asymptotic behavior of the Hessian matrix}
We require an additional assumption for the Hessian matrix (with respect to the short-run parameters) to be positive definite. Therefore, we need some notation. We write shortly $F_\vt:=\mathrm{e}^{A_{\vt} h}-K_\vt^{(h)}C_\vt$. The function is similar to the function in Schlemm and Stelzer \cite[Assumption C11]{SchlemmStelzer2012a}. However, we define $F_\vt$  different since we do not have a moving average representation of $Y^{(h)}$ with respect to the innovations $\varepsilon^{(h)}$. Hence, we have to adapt the criterion in Schlemm and Stelzer~\cite{SchlemmStelzer2012a}  and define the function
\begin{align}
\label{eqdefPsi}
&\psi_{\vt,j}:=
\begin{pmatrix}
%\text{vec}(\Pi(\vt))\\
\left[I_{j+1}\otimes K_\vt^{(h)\,\mathsf{T}}\otimes C_\vt \right]\left[(\text{vec }I_N)^\mathsf{T} ~~~ (\text{vec }F_\vt)^\mathsf{T} ~\ldots~ (\text{vec }F_\vt^{j})^\mathsf{T}\right]^\mathsf{T}
\\
\text{vec }V_\vt^{(h)}
\end{pmatrix}.
\end{align}

\begin{assumptionletter}
\label{AssSSCov}
Let there exist a positive index $j_0$ such that the $[(j_0+2)d^2\times s_2]$ matrix $\nabla_{\vt_2}  \psi_{\vt^0,j_0}$
has rank $s_2$.
\end{assumptionletter}
\begin{proposition}
\label{PropConvergHessianMat}
Let \autoref{AssSSCov} additionally hold.
Then, $$\nabla_{\vt_2}^2\mathcal{L}_{n}^{(h)}({\vartheta}^0)\cip Z_{st}(\vt^0),$$
where the $(s_2\times s_2)$-dimensional matrix $Z_{st}(\vt^0)$ is given by
\begin{align*}
[Z_{st}(\vt^0)]_{i,j}=&\,2\E\big(\partial_{i}^{st}\varepsilon_1^{(h)}(\vartheta^0)^\mathsf{T}\big) \big(V^{(h)}_{\vartheta^0}\big)^{-1} \big(\partial_{j}^{st}\varepsilon_1^{(h)}(\vartheta^0)\big)
\\&+\tr\left(\big(V^{(h)}_{\vartheta^0}\big)^{-\frac12}\big(\partial_{i}^{st}V^{(h)}_{\vartheta^0}\big)\big(V^{(h)}_{\vartheta^0}\big)^{-1} \partial_{j}^{st}V^{(h)}_{\vartheta^0} \big(V^{(h)}_{\vartheta^0}\big)^{-\frac12}\right).
\end{align*}
  Moreover,   the limiting matrix $Z_{st}(\vt^0)$ is almost surely a non-singular deterministic matrix.
  \end{proposition}
\begin{proof}
 We proceed as in the proof of Proposition~\ref{PropConvergHessianMat2}.
\newline
\textbf{Step 1:}
%The first term in is a constant.
%Let us now consider all other terms in \eqref{eqDerDerQ} except the last one. They all follow in a similar way,
 %hence we consider again for example the relevant part of the fifth term.
% As in the proof of Proposition~\ref{PropConvergHessianMat2} we use a mean value expansion
%\begin{eqnarray} \label{5.15b}
%\lefteqn{\frac{1}{n} \sum_{k=1}^n\partial_{j}^{st}\varepsilon_k^{(h)}(\vartheta^0)\varepsilon_k^{(h)}(\vartheta^0)^\mathsf{T}} \notag \\
%&&=\frac{1}{n} \sum_{k=1}^n\partial_{j}^{st}\varepsilon_k^{(h)}(\vartheta^0) \varepsilon_k^{(h)\mathsf{T}}
%+\sum_{l=1}^{s}\left[\frac{1}{n}\sum_{k=1}^n \partial_{j}^{st}\varepsilon_k^{(h)}(\vartheta^0)
%\partial_{l} \varepsilon_k^{(h)}(\overline{\vartheta}_n)^\mathsf{T}\right]\big(\overline{\vartheta}_{n}(l)-\vartheta^0(l)\big)
%\end{eqnarray}
%where $\vartheta^0(l)$ is the $l^{th}$ component of $\vt^0$, analogously $\overline{\vartheta}_{n}(l)$ the $l^{th}$ component of $\overline{\vartheta}_{n}$ and
%$\overline\vartheta_n\in\Theta$ with $\|\overline \vartheta_n-\vartheta^0\|\leq \|\overline \vartheta_n-\vartheta^0\|$.
%
Since $(\partial_{i,j}^{st}\varepsilon_k^{(h)}(\vartheta^0),\partial_{j}^{st}\varepsilon_k^{(h)}(\vartheta^0), \varepsilon_k^{(h)}(\vt^0))_{k\in\N}$ is a stationary and an ergodic sequence with finite absolute moment (see Lemma~\ref{Lemma 2.5}(a)) we obtain with
Birkhoff's Ergodic Theorem
\begin{align}
&\partial^{st}_{i,j}\mathcal{L}_{n}^{(h)}({{\vartheta}^{0}})
\cip
 \tr\left(\big(V^{(h)}_{\vartheta^0}\big)^{-1} \partial_{i,j}^{st}V_{\vartheta^0}^{(h)}-\big(V_{\vartheta^0}^{(h)}\big)^{-1} \big(\partial_{i}^{st}V_{\vartheta^0}^{(h)}\big)\big(V_{\vartheta^0}^{(h)}\big)^{-1} \big(\partial_{j}^{st}V_{\vartheta^0}^{(h)}\big)\right)
 \notag
 \\
 &\qquad\qquad\qquad\qquad
 -\tr\left(\big(V_{\vartheta^0}^{(h)}\big)^{-1}\E\left[\varepsilon_1^{(h)}(\vartheta^0)\varepsilon_1^{(h)}(\vartheta^0)^\mathsf{T}\right]\big(V_{\vartheta^0}^{(h)}\big)^{-1}\partial_{i,j}^{st}V_{\vartheta^0}^{(h)} \right)
 \notag
 \\
 &\qquad\qquad\qquad\qquad
 +\tr\left(\big(V_{\vartheta^0}^{(h)}\big)^{-1} \big(\partial_{j}^{st}V_{\vartheta^0}^{(h)}\big) \big(V_{\vartheta^0}^{(h)}\big)^{-1}\E\left[\varepsilon_1^{(h)}(\vartheta^0)\varepsilon_1^{(h)}(\vartheta^0)^\mathsf{T}\right] \big(V_{\vartheta^0}^{(h)}\big)^{-1}\partial_{i}^{st}V_{\vartheta^0}^{(h)} \right)
 \notag
  \\
   &\qquad\qquad\qquad\qquad
   +\tr\left(\big(V_{\vartheta^0}^{(h)}\big)^{-1} \E\left[\varepsilon_1^{(h)}(\vartheta^0)\varepsilon_1^{(h)}(\vartheta^0)^\mathsf{T}\right] \big(V_{\vartheta^0}^{(h)}\big)^{-1}\big(\partial_{j}^{st}V_{\vartheta^0}^{(h)}\big)\big(V_{\vartheta^0}^{(h)}\big)^{-1} \partial_{i}^{st}V_{\vartheta^0}^{(h)} \right)
 \notag
 \\
  &\qquad\qquad\qquad\qquad
  -\tr\left(\big(V_{\vartheta^0}^{(h)}\big)^{-1} \E\left[\partial_{j}^{st}\varepsilon_1^{(h)}(\vartheta^0)\varepsilon_1^{(h)}(\vartheta^0)^\mathsf{T}\right] \big(V_{\vartheta^0}^{(h)}\big)^{-1} \partial_{i}^{st}V_{\vartheta^0}^{(h)} \right)
 \notag
 \\
  &\qquad\qquad\qquad\qquad
   +2\cdot\tr\left( \big(V_{\vartheta^0}^{(h)}\big)^{-1} \E\left[\varepsilon_1^{(h)}(\vartheta^0)\left(\partial_{i,j}^{st}\varepsilon_1^{(h)}(\vartheta^0)^\mathsf{T}\right)\right]\right)
   \notag
   \\
 &\qquad\qquad\qquad\qquad
 -2\cdot\tr\left(\big(V_{\vartheta^0}^{(h)}\big)^{-1} \E\left[\varepsilon_1^{(h)}(\vartheta^0)\partial_{i}^{st}\varepsilon_1^{(h)}(\vartheta^0)^\mathsf{T}\right] \big(V_{\vartheta^0}^{(h)}\big)^{-1}\partial_{j}^{st} V_{\vartheta^0}^{(h)}\right)
    \notag
 \\
  &\qquad\qquad\qquad\qquad
   +2\cdot\E\left[\left(\partial_{i}^{st}\varepsilon_1^{(h)}(\vartheta^0)^\mathsf{T}\right) \big(V_{\vartheta^0}^{(h)}\big)^{-1} \left(\partial_{j}^{st}\varepsilon_1^{(h)}(\vartheta^0)\right)\right].
 \notag
\end{align}
 Combining this with Lemma~\ref{Lemma 2.5}~(c,d) results  in
\begin{align*}
\partial^{st}_{i,j}\mathcal{L}_{n}^{(h)}({{\vartheta}^{0}})
\cip\tr\left(\big(V_{\vartheta^0}^{(h)}\big)^{-1}\big(\partial_{i}^{st}V_{\vartheta^0}^{(h)}\big)\big(V_{\vartheta^0}^{(h)}\big)^{-1} \partial_{j}^{st}V_{\vartheta^0}^{(h)} \right)
 +2\cdot\E\left[\big(\partial_{i}^{st}\varepsilon_1^{(h)}(\vartheta^0)^\mathsf{T}\big) \big(V_{\vartheta^0}^{(h)}\big)^{-1} \big(\partial_{j}^{st}\varepsilon_1^{(h)}(\vartheta^0)\big)\right]
 .
\end{align*}
\textbf{Step 2:}
Next we check that $Z_{st}(\vt^0)$ is positive definite with probability one. That we show by contradiction similarly to the corresponding proofs in Schlemm and Stelzer \cite[Lemma 3.22]{SchlemmStelzer2012a} or Boubacar and Francq \cite[Lemma 4]{Boubacar2011}, respectively.
From Step 2 we know that
\begin{align}
\label{eqHessLimit}
&\nabla_{\vt_2}^2\mathcal{L}_{n}^{(h)}(\vt^0)\cip Z_{st}(\vt^0)=Z_{st,1}(\vt^0)+Z_{st,2}(\vt^0),
\\
\text{where}\quad
\notag
&Z_{st,1}(\vt^0):=2\cdot\left[\E\left(\partial_{i}^{st}\varepsilon_1^{(h)}(\vartheta^0)^\mathsf{T}\right) \big(V_{\vartheta^0}^{(h)}\big)^{-1} \left(\partial_{j}^{st}\varepsilon_1^{(h)}(\vartheta^0)\right)\right]_{1\leq i,j\leq s_2}
\\\text{and }\quad
\notag
&Z_{st,2}(\vt^0):=\left[\tr\left(\big(V_{\vartheta^0}^{(h)}\big)^{-\frac12}\big(\partial_{i}^{st}V_{\vartheta^0}^{(h)}\big)\big(V_{\vartheta^0}^{(h)}\big)^{-1} \partial_{j}^{st}V_{\vartheta^0}^{(h)} \big(V_{\vartheta^0}^{(h)}\big)^{-\frac12}\right)\right]_{1\leq i,j\leq s_2}.
\end{align}
We can factorize
$Z_{st,2}(\vt^0)$ in the following way:
\begin{align*}
Z_{st,2}(\vt^0)=&
\begin{pmatrix}
a_1
&
\ldots
&
a_{s_2}
\end{pmatrix}^\mathsf{T}
\begin{pmatrix}
a_1
&
\ldots
&
a_{s_2}
\end{pmatrix}
\quad\text{with}\quad
a_i:=\left(\big(V_{\vartheta^0}^{(h)}\big)^{-\frac12}\otimes \big(V_{\vartheta^0}^{(h)}\big)^{-\frac12}\right) \text{vec}\big(\partial_{i}^{st}V_{\vartheta^0}^{(h)}\big).
\end{align*}
Thus, $Z_{st,2}(\vt^0)$ is obviously positive semi-definite. Similarly, $Z_{st,1}(\vt^0)$ is positive semi-definite.
It remains to check that for any $b\in\R^{s_2}\setminus \{0_{s_2}\}$ we have $b^\mathsf{T}Z_{st,1}(\vt^0)b+b^\mathsf{T}Z_{st,2}(\vt^0)b>0$.
We assume for the sake of contradiction that there exists a vector $b\in\R^{s_2}\backslash\{0_{s_2}\}$ such that \begin{align}
\label{eqAssContrPosDef}
\tag{$\Diamond$}
b^\mathsf{T}Z_{st,1}(\vt^0)b+b^\mathsf{T}Z_{st,2}(\vt^0)b=0.
\end{align}
In order to be zero, each summand $b^\mathsf{T}Z_{st,1}(\vt^0)b$ and $ b^\mathsf{T}Z_{st,2}(\vt^0)b$ must be zero, since $Z_{st,1}(\vt^0)$ as well as $Z_{st,2}(\vt^0)$ are positive semi-definite.
But $b^{\mathsf{T}}Z_{st,1}(\vt^0)b=0$ is only possible if
\beao
    0_d=(\nabla_{\vt_2}\varepsilon_1^{(h)}(\vt^0))b=-\sum_{j=1}^\infty  \left(\nabla_{\vartheta_2}\left[ C_{\vt^0} F_{\vt^0}^{j-1}K_{\vt^0}^{(h)}Y_{k-j}^{(h)}\right]\right)b \quad \Pas
\eeao
%where the operator $\widetilde\nabla_{\vt_2}$ is defined as  $\widetilde\nabla_{\vt_2}f(\vt^0)=(\partial_1^{st}f(\vt),\ldots,\partial_{s_2}^{st}f(\vt))|_{\vt=\vt^0}\in\R^{d\times s_2}$ for a matrix function $f(\vt)$ in $\R^d$.
Rewriting this equation yields
\beam \label{Rew}
    \left( \nabla_{\vartheta_2}\left[ C_{\vt^0} K_{\vt^0}^{(h)}Y_{k-1}^{(h)}\right]\right)b=-\sum_{j=2}^\infty \left(\nabla_{\vartheta_2} \left[
     C_{\vt^0} F_{\vt^0}^{j-1}K_{\vt^0}^{(h)} Y_{k-j}^{(h)}\right]\right)b \quad \Pas
\eeam
Hence, for every row $i=1,\ldots,d$ and $b=(b_1,\ldots,b_{s_2})^{\mathsf{T}}$ we obtain
\beao
    \sum_{u=1}^{s_2}\left[\sum_{l=1}^d \partial_u^{st} (C_{\vt^0} K_{\vt^0}^{(h)})_{i,l}Y_{k-1,l}^{(h)}\right]b_u=-\sum_{j=2}^\infty
    \sum_{u=1}^{s_2}\left[\sum_{l=1}^d \partial_u^{st}(C_{\vt^0} F_{\vt^0}^{j-1}K_{\vt^0}^{(h)})_{i,l} Y_{k-j,l}^{(h)}\right]b_u \quad \Pas,
\eeao
which is equivalent to
\beao
    \left(\nabla_{\vartheta_2}\left[e_i^{\mathsf{T}} C_{\vt^0} K_{\vt^0}^{(h)}\right]b\right)^{\mathsf{T}}Y_{k-1}^{(h)}=-\sum_{j=2}^\infty
    \left(\nabla_{\vartheta_2} \left[e_i^{\mathsf{T}}
     C_{\vt^0} F_{\vt^0}^{j-1}K_{\vt^0}^{(h)}\right]b\right)^{\mathsf{T}} Y_{k-j}^{(h)} \quad \Pas
\eeao
But then $\left(\nabla_{\vartheta_2}\left[e_i^{\mathsf{T}} C_{\vt^0} K_{\vt^0}^{(h)}\right]b\right)^{\mathsf{T}}Y_{k-1}^{(h)}$ lies in $\overline{\text{span}}\{Y_j^{(h)}:j\leq k-2\}$.
By the definition of the linear innovations, this is only possible if $\left(\nabla_{\vartheta_2}\left[e_i^{\mathsf{T}} C_{\vt^0} K_{\vt^0}^{(h)}\right]b\right)^{\mathsf{T}}\varepsilon_{k-1}^{(h)}=0$ $\Pas$ However, $V_{\vt^0}^{(h)}=\E(\varepsilon_{k-1}^{(h)}(\varepsilon_{k-1}^{(h)})^{\mathsf{T}})$ is non-singular due to Scholz~\cite[Lemma 5.9.1]{Scholz}
so that necessarily $\nabla_{\vartheta_2}\left[e_i^{\mathsf{T}} C_{\vt^0} K_{\vt^0}^{(h)}\right]b=0_d$ for $i=1,\ldots,d$. This is again equivalent
to $\nabla_{\vt_2}(C_{\vt^0} K_{\vt^0}^{(h)})b=0_{d^2}$. Plugging this in \eqref{Rew} gives
$$\nabla_{\vartheta_2} \left[
     C_{\vt^0} F_{\vt^0}^{}K_{\vt^0}^{(h)} Y_{k-2}^{(h)}\right]b=-\sum_{j=3}^\infty \nabla_{\vartheta_2} \left[
     C_{\vt^0} F_{\vt^0}^{j-1}K_{\vt^0}^{(h)} Y_{k-j}^{(h)}\right]b. $$ Then, we can show similarly
     $\nabla_{\vt_2}(C_{\vt^0} F_{\vt^0}K_{\vt^0}^{(h)})b=0_{d^2}$ and obtain recursively  that
\beam \label{5.17}
\nabla_{\vt_2}(C_{\vt^0} F_{\vt^0}^jK_{\vt^0}^{(h)})b=0_{d^2}, \quad \,j\in\N_0.
\eeam
On the other hand, we obtain due $b^\mathsf{T}Z_{st,2}(\vt^0)b=0$ under assumption \eqref{eqAssContrPosDef} that
\begin{align}
\label{5.18}
\big(\nabla_{\vt_2}V_{\vt^0}^{(h)}\big)b=0_{d^2}.
\end{align}
The definition of $\psi_{\vt,j}$ in \eqref{eqdefPsi}, \eqref{5.17} and \eqref{5.18} imply that $\big(\nabla_{\vt_2}\psi_{\vt^0,j}\big)b=0_{(j+2)d^2}$ holds for all $j\in\N$, which contradicts \autoref{AssSSCov}. Hence, $Z_{st}(\vt^0)$ is almost surely positive definite.
\end{proof}

\subsubsection{Asymptotic normality of the short-run QML estimator}

We conclude this section with the last main result of this paper, namely the asymptotic distribution of the short-run QML estimator.
\begin{theorem}
\label{thmAsymDisEstimators}
Let \autoref{AssSSCov} additionally hold.
Furthermore, suppose
\begin{align*}
I(\vt^0)= \lim_{n\to\infty}\Var \Big(\nabla_{\vt_2}\mathcal{L}_{n}^{(h)}(\vt^0)\Big)
\quad \text{and}\quad Z_{st}(\vt^0)=\lim_{n\to\infty}\nabla_{\vt_2}^2\mathcal{L}_n^{(h)}(\vt^0).
\end{align*} Then, as $n\to\infty$,
\begin{align}
\label{eqAsymDisTheta}
\sqrt{n}(\widehat{\vartheta}_{n,2}-\vartheta_{2}^0)
&\cid
 \mathcal{N}(0,Z_{st}(\vt^0)^{-1}I(\vt^0)Z_{st}(\vt^0)^{-1}).
\end{align}
\end{theorem}
Again we need the following auxiliary result for the proof.
\begin{lemma} \label{Proposition 6.9}
For every $\tau>0$ and every $\eta>0$, there exist an integer $n(\tau,\eta)$ and  real numbers $\delta_1,\delta_2>0$ such that for $\frac{3}{4}<\gamma<1$,
\begin{itemize}
\item[(a)]
 $ \P\bigg(\sup_{\vartheta_1\in N_{n,\gamma}(\vartheta_{1}^0,\delta_1)}\|\sqrt{n}\nabla_{\vt_2}\mathcal{L}_n^{(h)}(\vartheta_1,\vt_2^0)-\sqrt{n}\nabla_{\vt_2}\mathcal{L}_n^{(h)}(\vartheta_1^0,\vt_2^0)\|>\tau\bigg)\leq\eta \quad \text{ for } n\geq n(\tau,\eta)$,
\item[(b)] $ \P\bigg(\sup_{\vartheta\in N_{n,\gamma}(\vartheta_{1}^0,\delta_1)\times \mathcal{B}(\vartheta_{2}^0,\delta_2)}\|\nabla^2_{\vt_2}\mathcal{L}_n^{(h)}(\vartheta)-\nabla^2_{\vt_2}\mathcal{L}_n^{(h)}(\vartheta^0)\|>\tau\bigg)\leq\eta \quad \text{ for } n\geq n(\tau,\eta).$
\end{itemize}
\end{lemma}
This local stochastic equicontinuity condition for the standardized score $\sqrt{n}\nabla_{\vt_2}\mathcal{L}_n^{(h)}(\cdot,\vt_2^0)$ in $\vt_1^0$ and for the standardized Hessian matrix $\nabla^2_{\vt_2}\mathcal{L}_n^{(h)}(\cdot)$ in $\vt^0$
do not hold for general $\vt_1$ and $\vartheta$, respectively. \label{pageref3} Accordingly the stochastic equicontinuity conditions of Saikkonen~\cite{Saikkonen1995} are not satisfied.

\begin{proof}[Proof of Lemma~\ref{Proposition 6.9}] $\mbox{}$\\
(a) \, We use the upper bound
\begin{align}
    &\hspace*{-1.5cm}\sup_{\vartheta_1\in N_{n,\gamma}(\vartheta_{1}^0,\delta_1)}\|\sqrt{n}\nabla_{\vt_2}\mathcal{L}_n^{(h)}(\vartheta_1,\vt_2^0)-\sqrt{n}\nabla_{\vt_2}\mathcal{L}_n^{(h)}(\vartheta_1^0,\vt_2^0)\|\nonumber\\
    &\leq\, \sup_{\vartheta_1\in N_{n,\gamma}(\vartheta_{1}^0,\delta_1)}\|\sqrt{n}\nabla_{\vt_2}\mathcal{L}_{n,1,1}^{(h)}(\vartheta_1,\vt_2^0)-\sqrt{n}\nabla_{\vt_2}\mathcal{L}_{n,1,1}^{(h)}(\vartheta_1^0,\vt_2^0)\| \nonumber\\
    & \, \quad+\sup_{\vartheta_1\in N_{n,\gamma}(\vartheta_{1}^0,\delta_1)}\|\sqrt{n}\nabla_{\vt_2}\mathcal{L}_{n,1,2}^{(h)}(\vartheta_1,\vt_2^0)-\sqrt{n}\nabla_{\vt_2}\mathcal{L}_{n,1,2}^{(h)}(\vartheta_1^0,\vt_2^0)\|.
    \label{C1}
\intertext{Since $\Pi(\vt^0)C_1=0_{d\times c}$ and $\nabla_{\vt_2}(\Pi(\vt^0)C_1)=0_{dc\times s_2}$ we can apply \eqref{EQ2} and receive}
   % \lefteqn{ \sup_{\vartheta_1\in N_{n,\gamma}(\vartheta_{1}^0,\delta_1)}\|\sqrt{n}\nabla_{\vt_2}\mathcal{L}_{n,1,1}^{(h)}(\vartheta_1,\vt_2^0)-\sqrt{n}\nabla_{\vt_2}\mathcal{L}_{n1,1}^{(h)}(\vartheta_1^0,\vt_2^0)\|} \nonumber\\
     &\hspace*{-1cm}\sup_{\vartheta_1\in N_{n,\gamma}(\vartheta_{1}^0,\delta_1)}\|\sqrt{n}\nabla_{\vt_2}\mathcal{L}_{n,1,1}^{(h)}(\vartheta_1,\vt_2^0)-\sqrt{n}\nabla_{\vt_2}\mathcal{L}_{n,1,1}^{(h)}(\vartheta_1^0,\vt_2^0)\| \nonumber\\&  \leq \, \mathfrak{C}
     \sup_{\vartheta_1\in N_{n,\gamma}(\vartheta_{1}^0,\delta_1)}n^{\frac{3}{2}}\|\vt_1-\vt_1^0\|^2\tr\left(\frac{1}{n^2}\sum_{k=1}^nL_{k-1}^{(h)}[L_{k-1}^{(h)}]^{\mathsf{T}}\right) \nonumber\\
     &\leq \, \mathfrak{C}n^{\frac{3}{2}-2\gamma}\tr\left(\frac{1}{n^2}\sum_{k=1}^nL_{k-1}^{(h)}[L_{k-1}^{(h)}]^{\mathsf{T}}\right)\cip 0,
\end{align}
where we used $\gamma>3/4$ and $\tr\left(\frac{1}{n^2}\sum_{k=1}^nL_{k-1}^{(h)}[L_{k-1}^{(h)}]^{\mathsf{T}}\right)=\mathcal{O}_p(1)$ due to Proposition~\ref{PropUniformConvRes1}(b).
For the second term we get by \eqref{EQ3} and \eqref{EQ5}, and similar calculations as in Lemma~\ref{Lemma 3.5a} that
\beam \label{C3}
   \lefteqn{\hspace*{-1cm} \sup_{\vartheta_1\in N_{n,\gamma}(\vartheta_{1}^0,\delta_1)}\|\sqrt{n}\nabla_{\vt_2}\mathcal{L}_{n,1,2}^{(h)}(\vartheta_1,\vt_2^0)-\sqrt{n}\nabla_{\vt_2}\mathcal{L}_{n,1,2}^{(h)}(\vartheta_1^0,\vt_2^0)\|}\nonumber\\
   &&\leq  \sup_{\vartheta_1\in N_{n,\gamma}(\vartheta_{1}^0,\delta_1)}\sqrt{n}\mathfrak{C}\|\vt_1-\vt_1^0\|U_n
   \leq \mathfrak{C}n^{\frac{1}{2}-\gamma}U_n\cip 0
\eeam
due to $\gamma>3/4$ and $U_n=\mathcal{O}_p(1)$. A combination of \eqref{C1}-\eqref{C3} proves (a). \\
(b) The proof is similar to (a).
\end{proof}

\begin{proof}[Proof of \Cref{thmAsymDisEstimators}]
The proof is similar to the proof of \Cref{thmAsymDisEstimators1} using  Proposition~\ref{PropConvergScoreVec1}, Proposition~\ref{PropConvergHessianMat}, Proposition~\ref{Lemma 2.8} and Lemma~\ref{Proposition 6.9}.
\end{proof}
%
%%%%%
%
%%%%%
%

\section{Simulation study} \label{Section: Simulation study}

In this section we want to demonstrate the validity of the proposed QML-method by a simulation study. The simulated state space processes are driven
either by a standard Brownian motion or by a NIG  (normal inverse Gaussian) Lévy process with mean value $0_m$.
The increment of an $m$-dimensional NIG L\'evy process $L(t)-L(t-1)$ has the density
\begin{align}
f_{NIG}(x;\mu,\alpha,\beta,\delta,\Delta)&= \frac{\delta \mathrm{e}^{ \delta\kappa}}{2\pi }\cdot\frac{\mathrm{e}^{ \langle \beta,x\rangle} (1+\alpha g(x))} {\mathrm{e}^{\alpha g(x)}g(x)^3}, \quad x\in\R^m, \notag
\\
\notag
\text{where }\quad g(x)&=\sqrt{\delta^2+\langle x-\mu,\Delta(x-\mu)\rangle}
\quad \text{ and }\quad \kappa^2=\alpha^2-\langle \beta, \Delta\beta \rangle >0,
\end{align}
$\mu\in\R^m$ is a location parameter, $\alpha\geq 0$ is a shape parameter, $\beta\in\R^m$ is a symmetry parameter, $\delta\geq 0$ is a scale parameter
and $\Delta\in\R^{m\times m}$ is positive semi-definite with $\det\Delta =1$ determining the dependence of the components of $(L(t))_{t\geq 0}$.
The covariance of the process is then
\begin{align*}
\Sigma_L=\delta(\alpha-\beta^\mathsf{T}\Delta\beta)^{-\frac12}\left(\Delta+ (\alpha^2-\beta^\mathsf{T}\Delta\beta)^{-1}\Delta\beta\beta^T\Delta \right).
\end{align*}
For more details on NIG Lévy processes see, e.g., Barndorff-Nielsen~\cite{barndorff1997}.
In all simulation studies we have simulated 350 independent replicates of a cointegrated state space process
on an equidistant time grid $0,0.01,\ldots, 2000$ by applying an Euler scheme to the stochastic differential equation \eqref{SPR} with
 initial value $X(0)=0_N$ and $h$ in the observation scheme is chosen as $1$.

Moreover, we use canonical representations of the state space models. On the one hand, $C_{1,\vt_1}$ are chosen on such a way that
$C_{1,\vt_1}$ are lower triangular matrices with $C_{1,\vt_1}^{\mathsf{T}}C_{1,\vt_1}=I_c$ and similarly
$C_{1,\vt_1}^{\perp}$ are lower triangular matrices with $C_{1,\vt_1}^{\perp \mathsf{T}}C_{1,\vt_1}^{\perp}=I_{d-c}$
satisfying \autoref{AssMBrownMot}, \autoref{AssMUniquePosTriForm}, and \autoref{AssMPosDefN} for a properly chosen parameter space $\Theta$.
 On the other hand, the parametrization of the stationary part $Y_{st,\vt}$ is based on the echelon canonical form as given in Schlemm and Stelzer~\cite{SchlemmStelzer2012a}
 such that as well \autoref{AssMBrownMot} and \autoref{Identifiability} are satisfied for the properly chosen parameter space $\Theta$.
The echelon canonical form is widely used in the VARMA context, see, e.g., L\"utkepohl and Poskitt \cite{Luetkepohl1996} and the textbooks of  L\"utkepohl \cite{Luetkepohl2005}, or Hannan and Deistler \cite{hannandeistler2012}. In the context of linear state space models canonical representations can also be found  in Guidorzi \cite{Guidorzi1975361}.
  For the asymptotic normality of the short-run parameters we require additionally \autoref{AssSSCov}. However, this condition cannot be checked analytically, this is only possible
  numerically.

%% ==============================
\subsection{Bivariate state space model}
\label{ch:chapter6:sec:Simulations:subsec:2dim}
%% ==============================
\setcounter{MaxMatrixCols}{15}
\begin{table}%[hp]
  \centering
\begin{tabular}{ c| c ||c | c | c || c | c | c ||}
  \cline{3-8}
   \cline{3-8}
 \multicolumn{2}{c||}{}  &  \multicolumn{3}{ c|| }{ NIG-driven } &  \multicolumn{3}{ c|| }{ Brownian-driven }\\
\cline{2-8}	
  \hline
  \hline
   \multicolumn{1}{ |c| }{} &\multicolumn{1}{ c|| }{True} &\multicolumn{1}{ c |}{Mean}  & \multicolumn{1}{ c| }{Bias} & \multicolumn{1}{ c|| }{Std.~dev.}
    &\multicolumn{1}{ c |}{Mean}  & \multicolumn{1}{ c| }{Bias} & \multicolumn{1}{ c|| }{Std.~dev.}
  \\
\hline	
\hline	        	
\multicolumn{1}{ |c| }{$\vartheta_1$} & -1 &-0.9857 &-0.0143 & 0.0515  &-0.9895 &-0.0105  & 0.0425\\
\multicolumn{1}{ |c| }{$\vartheta_2$} & -2 & -2.0025 & \phantom{-}0.0025 & 0.0573 & -1.9934 & -0.0066 & 0.0459\\
\multicolumn{1}{ |c| }{$\vartheta_3$} & \phantom{-}1 & \phantom{-}0.9919 & \phantom{-}0.0081 & 0.0749 & \phantom{-}0.9898 & \phantom{-}0.0102 & 0.0570 \\
\multicolumn{1}{ |c| }{$\vartheta_4$} & -2 & -1.9758 & -0.0242 & 0.1126 & -1.9701 & -0.0299 & 0.0872 \\
\multicolumn{1}{ |c| }{$\vartheta_5$} & -3 & -2.9774 & -0.0226 & 0.0497  & -2.9898 & -0.0102 & 0.0324\\
\multicolumn{1}{ |c| }{$\vartheta_6$} &\phantom{-}1 & \phantom{-}1.0129 & -0.0129  & 0.1071 &\phantom{-}1.0155 & -0.0155 & 0.0789 \\
\multicolumn{1}{ |c| }{$\vartheta_7$} &\phantom{-}2 & \phantom{-}2.0005 & -0.0005 & 0.0690 & \phantom{-}2.0068 & -0.0068 & 0.0441 \\
 \hline
 \multicolumn{1}{ |c| }{$\vartheta_8$} &\phantom{-}1 & \phantom{-}1.0078 & -0.0078 & 0.0684 &\phantom{-}1.0096 & -0.0096 & 0.0482  \\
\multicolumn{1}{ |c| }{$\vartheta_9$} &\phantom{-}1 & \phantom{-}0.9872 & \phantom{-}0.0128  & 0.0761 &\phantom{-}0.9777 & \phantom{-}0.0223 & 0.0599\\
\hline
\multicolumn{1}{ |c| }{$\vartheta_{10}$} & \phantom{-}0.4751 & \phantom{-}0.4715 & \phantom{-}0.0036 & 0.0678 & \phantom{-}0.5200 & -0.0449 & 0.0518\\
\multicolumn{1}{ |c| }{$\vartheta_{11}$} & -0.1622 & -0.1572 & -0.0050 & 0.0381 & -0.1283 & -0.0339 & 0.0266\\
\multicolumn{1}{ |c| }{$\vartheta_{12}$} & \phantom{-}0.3708 & \phantom{-}0.3698 & \phantom{-}0.0010 & 0.0314 & \phantom{-}0.3195 & \phantom{-}0.0513 & 0.0213\\
 \hline
\multicolumn{1}{ |c| }{$\vartheta_{13}$} & \phantom{-}3 & \phantom{-}2.9999 & \phantom{-}0.0001 & 0.0075 &\phantom{-}2.9981 & \phantom{-}0.0019 & 0.0068 \\
  \hline
\end{tabular}
  \caption{Sample mean, bias and sample standard deviation  of 350 replicates of QML of a two-dimensional NIG-driven and Brownian-driven cointegrated state space process.}\label{TableSimNIG2}
\end{table}

As canonical parametrization of the family of cointegrated state space models we take
\begin{comment}
  \begin{align}
  \notag
\dif X_\vt(t) &=
\begin{pmatrix}
\vartheta_1 &	\vartheta_2	&0&	0\\
0&	0&	1&	0\\
\vartheta_3&	\vartheta_4&	\vartheta_5	&0\\
0&	0&	0&	0
\end{pmatrix} X_\vt(t) \dif t +
\begin{pmatrix}
\vartheta_1&	\vartheta_2\\
\vartheta_6&	\vartheta_7\\
\vartheta_3+\vartheta_5\vartheta_6 &	\vartheta_4+\vartheta_5\vartheta_7\\
\vartheta_8	&\vartheta_9
\end{pmatrix}\dif L_\vt(t),
\intertext{and}
%\label{eqSimCanonParaDim2}
Y_\vt(t) &= \begin{pmatrix}
1&	0&	0&	\frac{\vartheta_{13}^2-1}{\vartheta_{13}^2+1}\\
0	&1	&0&	\frac{2\cdot\vartheta_{13}}{\vartheta_{13}^2+1}
\end{pmatrix} X_\vt(t).  \nonumber
 \end{align}
Hence,
\end{comment}
\begin{eqnarray*}
\begin{array}{rlrlrl}
    A_{2,\vt}& =\begin{pmatrix}\vartheta_1 &	\vartheta_2	&0	\\
0&	0&	1	\\
\vartheta_3&	\vartheta_4&	\vartheta_5	
\end{pmatrix}, \quad

& B_{2,\vt}& =\begin{pmatrix}
\vartheta_1&	\vartheta_2\\
\vartheta_6&	\vartheta_7\\
\vartheta_3+\vartheta_5\vartheta_6 &	\vartheta_4+\vartheta_5\vartheta_7
\end{pmatrix},

& C_{2,\vt} & =  \begin{pmatrix}
1&	0&	0\\
0	&1	&0 \end{pmatrix},\\

 B_{1,\vt}&  =\begin{pmatrix}
\vartheta_8	&\vartheta_9
\end{pmatrix}, \quad

&\text{vech}(\Sigma_\vt^L)&=(\vartheta_{10},\vartheta_{11},\vartheta_{12}), \quad

& C_{1,\vt}& =\begin{pmatrix}
	\frac{\vartheta_{13}^2-1}{\vartheta_{13}^2+1}\\
	\frac{2\cdot\vartheta_{13}}{\vartheta_{13}^2+1}
\end{pmatrix}.
\end{array}
\end{eqnarray*}
 This implies that we have one cointegration relation and the cointegration rank is equal to $1$. In total we have 13 parameters.
We use
\begin{align}
\label{eqtruepara2dim}
\vartheta^0=
\begin{pmatrix}
-1 & -2 & 1	&-2	&-3	&1	& 2 & 1 & 1 & 0.4751 & -0.1622 & 0.3708 &3
\end{pmatrix}^\mathsf{T}. \notag
\end{align}
In order to obtain the covariance matrix of the NIG Lévy process, we have to set the parameters of the NIG Lévy
process to
\begin{align*}
\delta=1,\quad \alpha=3,\quad \beta=\begin{pmatrix}1 \\1 \end{pmatrix},\quad \Delta=\begin{pmatrix} 1.2 & -0.5 \\ -0.5 & 1 \end{pmatrix}\quad  \text{and}\quad \mu=-\frac{1}{2\sqrt{31}}\begin{pmatrix}3 \\2 \end{pmatrix}.
\end{align*}
On this way the parameters of the stationary  process $Y_{st,\vt}$ are chosen  as in Schlemm and Stelzer \cite[Section 4.2]{SchlemmStelzer2012a},
who performed a simulation study for QML estimation of stationary state space processes.

In \autoref{TableSimNIG2} the sample mean, bias and sample standard deviation of the 350 replicates of the estimated parameters are summarized.
From this we see that in both the NIG-driven as well the Brownian motion driven model the bias and the sample standard deviation are quite low which reflect the consistency
of our estimator. Moreover, for the Brownian-motion driven model the sample standard deviation is for all parameters lower than for the NIG-driven model which
is not surprising since the Kalman filter as well as the quasi-maximum likelihood function are motivated from the Gaussian case.
In contrast, the bias in the NIG-driven model is often lower than in the Gaussian model. It attracts attention that in both models the cointegration parameter $\vt_{13}$
has the lowest bias and sample standard deviation of all estimated parameters. This is in accordance with the fact that the
consistency rate for the long-run parameters is faster than that for the short-run parameters.

Next, we investigate what happens if we use as underlying parameter space in the QML method a space which does not contain the true model.
In the first parameter space $\Theta_I$, we set $B_{2,\vt}=0_{3\times 2}$ and all other matrices as above. Hence, $Y_\vt$ for $\vt\in\Theta_I$
is integrated but not cointegrated. In the second parameter space  $\Theta_W$, we set $C_{1,\vt}=(0,1)^T$ and all other matrices
as above such that $Y_\vt$ for $\vt\in\Theta_W$ is cointegrated but the cointegration space does not model the true cointegration space. Finally, in the last parameter space
$\Theta_S$, we set $C_{1,\vt}=(0,0)^T$ and all other matrices as above such that $Y_\vt$
for $\vt\in\Theta_S$ is stationary and coincides with $Y_{\vt,st}$.
The sample mean, sample standard deviation, minimal value and maximal value of the minimum of the likelihood function for  100 replications of the Brownian motion
driven model
 in the four different spaces is presented in \autoref{TableLLF2sum}.
\begin{table}[hp]
  \centering
\begin{tabular}{ |c|c| c| c| c |}
 \hline
& $\Theta$   & $\Theta_I$ & $\Theta_W$   & $\Theta_S$   \\
& true pro.  & int. pro. & wrong coint. pro. & stat. pro.     \\
  \hline
Mean & 5.2303 & 5.2851 & 16.2713 & 23.8473   \\
  \hline
St. dev. & 0.0449 & 0.0956 & 11.3465 & 16.0159   \\
  \hline
  Min & 5.1226 & 5.1367 & \phantom{0}6.0526 & \phantom{0}9.4492  \\
  \hline
  Max & 5.3356 & 5.7509 & 79.3741 & 88.2747   \\
  \hline
\end{tabular}
  \caption{Minimum of the Likelihood function for the four different parameter spaces and the Brownian motion driven model.}\label{TableLLF2sum}
\end{table}
Of course in the space $\Theta$, containing the true model, the sample mean of the minimum of the likelihood function is lowest.
However, the sample mean for the space $\Theta_I$ is not to far away because there at least the long-run parameters can be estimated
more or less appropriate such that due to Proposition~\ref{Lemma 2.8} and Lemma~\ref{propConvEps2} we get $\inf_{\vt\in\Theta_I}\mathcal{\widehat{L}}_n^{(h)}(\vartheta)
=\inf_{\vt\in\Theta_I}\mathcal{{L}}_{n,2}^{(h)}(\vartheta_2)+o_p(1)\cip \inf_{\vt\in\Theta_I} \mathbfcal{L}_{2}^{(h)}(\vartheta_2) $.
However, the standard deviation is much lower in $\Theta$ than in $\Theta_I$.
In contrast to the spaces $\Theta_W$ and $\Theta_S$ where the likelihood function seems to diverge.
This is in accordance to the results of this paper because due to Proposition~\ref{Lemma 2.8}, Lemma~\ref{propConvEps2} and  \eqref{C.1}
we have $\inf\limits_{\vt\in\Theta_W}\mathcal{\widehat{L}}_n^{(h)}(\vartheta)\cip\infty$ and
$\inf\limits_{\vt\in\Theta_S}\mathcal{\widehat{L}}_n^{(h)}(\vartheta)\cip\infty$.

%% ==============================
\subsection{Three-dimensional state space model}
\label{ch:chapter6:sec:Simulations:subsec:3dim}
%% ==============================

The canonical parametrization of the cointegrated state space model has the form
\begin{eqnarray*}
\begin{array}{rlrlrl}
  A_{2,\vt}&=\begin{pmatrix}
\vartheta_1 &	\vartheta_2	&0&	\vartheta_3
\\
0&	0&	1&	0
\\
\vartheta_4&	\vartheta_5&	\vartheta_6	&\vartheta_7
\\
\vartheta_8&	\vartheta_9&	\vartheta_{10}	&\vartheta_{11}
\end{pmatrix},
\quad

&B_{2,\vt}&=\begin{pmatrix}
\vartheta_1 &	\vartheta_2	&	\vartheta_3
\\
\vartheta_{12} &	\vartheta_{13}	&	\vartheta_{14}
\\
\vartheta_4+\vartheta_6\vartheta_{12} &	\vartheta_5+\vartheta_6\vartheta_{13}&\vartheta_7+\vartheta_6\vartheta_{14}
\\
\vartheta_8+\vartheta_{10}\vartheta_{12} &	\vartheta_9+\vartheta_{10}\vartheta_{13}&\vartheta_{11}+\vartheta_{10}\vartheta_{14}
\end{pmatrix},
\quad \\

C_{2,\vt}&=\begin{pmatrix}
1&	0&	0&	0\\
0	&1	&0&	0\\
0&	0&	0&	1
\end{pmatrix}, \quad

& B_{1,\vt}&=\begin{pmatrix}
\vartheta_{15}	& \vartheta_{16} & \vartheta_{17}
\\
\vartheta_{18}	&\vartheta_{19} & \vartheta_{20}
\end{pmatrix}, \quad\\
\text{vech}(\Sigma_\vt^L)&=(\vartheta_{21},\vartheta_{22},\vartheta_{23},\vartheta_{24},\vartheta_{25},\vartheta_{26}),
\quad

&C_{1,\vt}&=\begin{pmatrix}
\frac{\vartheta_{27}^2+\vartheta_{28}^2-1}{\vartheta_{27}^2+\vartheta_{28}^2+1} &	0\\
\frac{2\cdot\vartheta_{27}}{\vartheta_{27}^2+\vartheta_{28}^2+1}	& \frac{\vartheta_{28}}{\sqrt{\vartheta_{27}^2+\vartheta_{28}^2}}\\
\frac{2\cdot\vartheta_{28}}{\vartheta_{27}^2+\vartheta_{28}^2+1}& -\frac{\vartheta_{27}}{\sqrt{\vartheta_{27}^2+\vartheta_{28}^2}}
\end{pmatrix}.
\end{array}
\end{eqnarray*}
The state space model has two common stochastic trends and the cointegration space is  a one-dimensional subspace of $\R^3$. In total we have $28$ parameters.
In \autoref{TableSimNIG3} the sample mean, bias and sample standard deviation of the estimated parameters  of 350 replicates
are summarized for both the NIG-driven as well the Brownian-motion driven model. In order to obtain the covariance matrix of the NIG Lévy process,
we have had to set the parameters of the NIG Lévy
process to
\begin{align*}
&\delta=1, \quad \alpha=3,\quad
\beta=\begin{pmatrix}1 \\1  \\1 \end{pmatrix},
\quad \Delta=\begin{pmatrix}
1.25 & -0.5 &\frac{1}{6}\sqrt{3}
\\
 -0.5 & 1  &  -\frac{1}{3}\sqrt{3}
\\
 \frac{1}{6}\sqrt{3} &  -\frac{1}{3}\sqrt{3}  & \frac{4}{3}
\end{pmatrix}
\quad \text{ and }\quad
 \mu=-\frac{1}{2\sqrt{31}}\begin{pmatrix}3 \\2 \end{pmatrix}.
\end{align*}
The results are very similar to the two-dimensional example. In most cases the sample standard deviation in the Brownian motion-driven model is lower
than in the NIG-driven model.
Moreover, the bias and the standard deviation of the long-run parameters $(\vt_{27},\vt_{28})$ are lower than the values of the other parameters.

\begin{table}
  \centering
\begin{tabular}{ c| c || c | c | c || c | c | c ||}
  \cline{3-8}
   \cline{3-8}
 \multicolumn{2}{c||}{}  &  \multicolumn{3}{ c|| }{ NIG-driven } &  \multicolumn{3}{ c|| }{ Brownian-driven }\\
\cline{2-8}	
  \hline
  \hline
   \multicolumn{1}{ |c| }{} &\multicolumn{1}{ c|| }{True} &\multicolumn{1}{ c |}{Mean}  & \multicolumn{1}{ c| }{Bias} & \multicolumn{1}{ c|| }{Std.~dev.}
    &\multicolumn{1}{ c |}{Mean}  & \multicolumn{1}{ c| }{Bias} & \multicolumn{1}{ c|| }{Std.~dev.}
  \\
\hline	
\hline	        			        	
\multicolumn{1}{ |c| }{$\vartheta_1$} & -2 & -1.9910 &-0.0090& 0.0583 & -1.9958 &-0.0042& 0.0475\\
\multicolumn{1}{ |c| }{$\vartheta_2$} & -3 & -3.0042 & \phantom{-}0.0042 & 0.0407 & -3.0005 & \phantom{-}0.0005 & 0.0339\\
\multicolumn{1}{ |c| }{$\vartheta_3$} & -3 & -3.0194 &  \phantom{-}0.0194 & 0.0456 & -3.0309 &  \phantom{-}0.0309 & 0.0401\\
\multicolumn{1}{ |c| }{$\vartheta_4$} & \phantom{-}1 & \phantom{-}0.9887  & \phantom{-}0.0113 & 0.0440 & \phantom{-}0.9987  & \phantom{-}0.0013 & 0.0381 \\
\multicolumn{1}{ |c| }{$\vartheta_5$} & \phantom{-}1 & \phantom{-}0.9977 & \phantom{-}0.0023  & 0.0351 & \phantom{-}0.9895 & \phantom{-}0.0105  & 0.0316\\
\multicolumn{1}{ |c| }{$\vartheta_6$} &-1 &  -0.9861 & -0.0139 &  0.0544 &  -0.9763  & -0.0237 & 0.0431\\
\multicolumn{1}{ |c| }{$\vartheta_7$} &\phantom{-}2 & \phantom{-}2.0122 & -0.0122  & 0.0396 & \phantom{-}2.0113 & -0.0113  & 0.0342 \\
\multicolumn{1}{ |c| }{$\vartheta_8$}& -1 &-1.0039 &  \phantom{-}0.0039 & 0.0442 &-1.0075 &  \phantom{-}0.0075 & 0.0399\\
\multicolumn{1}{ |c| }{$\vartheta_9$} & -3 & -2.9937 & -0.0063 & 0.0342  & -2.9896 & -0.0104 & 0.0348\\
\multicolumn{1}{ |c| }{$\vartheta_{10}$} &-3 & -2.9904 & -0.0096 & 0.0490  & -2.9892 & -0.0108 & 0.0444 \\
\multicolumn{1}{ |c| }{$\vartheta_{11}$} & -1 & -1.0055 & \phantom{-}0.0055  & 0.0449 & -1.0097 & \phantom{-}0.0097  & 0.0461 \\
\multicolumn{1}{ |c| }{$\vartheta_{12}$} & -1 & -1.0023  &  \phantom{-}0.0023 & 0.0386 & -1.0242 &  \phantom{-}0.0242 & 0.0367 \\
\multicolumn{1}{ |c| }{$\vartheta_{13}$} & \phantom{-}2 & \phantom{-}1.9984 & \phantom{-}0.0016 & 0.0363 & \phantom{-}2.0077 & -0.0077 & 0.0295\\
\multicolumn{1}{ |c| }{$\vartheta_{14}$} & \phantom{-}1 &  \phantom{-}1.0034 &  -0.0034  &  0.0353 &  \phantom{-}0.9740 &  \phantom{-}0.0260 & 0.0353\\
 \hline
\multicolumn{1}{ |c| }{$\vartheta_{15}$} & \phantom{-}1 &  \phantom{-}0.9984  & \phantom{-}0.0016 & 0.0351 & \phantom{-}1.0175  & -0.0175 & 0.0284 \\
\multicolumn{1}{ |c| }{$\vartheta_{16}$} & \phantom{-}0 &  -0.0345 & \phantom{-}0.0345 & 0.0644  &  -0.0361 & \phantom{-}0.0361 & 0.0513\\
\multicolumn{1}{ |c| }{$\vartheta_{17}$} & \phantom{-}1 &  \phantom{-}0.9840 & \phantom{-}0.0160 & 0.0521 &  \phantom{-}0.9623 & \phantom{-}0.0377 & 0.0417\\
\multicolumn{1}{ |c| }{$\vartheta_{18}$} & \phantom{-}1 &  \phantom{-}1.0010 & -0.0010 & 0.0314 &  \phantom{-}0.9877 & \phantom{-}0.0123 & 0.0303 \\
\multicolumn{1}{ |c| }{$\vartheta_{19}$} & -2 &  -1.9841 & -0.0159 &  0.0388 &  -1.9868 & -0.0132 & 0.0306 \\
\multicolumn{1}{ |c| }{$\vartheta_{20}$} & \phantom{-}0 &  \phantom{-}0.0111 & -0.0111 & 0.0347 &  -0.0090 & \phantom{-}0.0090 & 0.0362\\
 \hline
\multicolumn{1}{ |c| }{$\vartheta_{21}$} & \phantom{-}0.5310 & \phantom{-}0.5279   & \phantom{-}0.0031 & 0.0605 & \phantom{-}0.5849   & -0.0539 & 0.0478 \\
\multicolumn{1}{ |c| }{$\vartheta_{22}$} & -0.1934 & -0.1870 & -0.0064 & 0.0385  & -0.2037 & \phantom{-}0.0103 & 0.0328 \\
\multicolumn{1}{ |c| }{$\vartheta_{23}$} & \phantom{-}0.1678 & \phantom{-}0.1678 & \phantom{-}0.0000 & 0.0467  & \phantom{-}0.1513 & \phantom{-}0.0165 & 0.0396\\
\multicolumn{1}{ |c| }{$\vartheta_{24}$} & \phantom{-}0.3784 &  \phantom{-}0.3816 & -0.0032 & 0.0293 &  \phantom{-}0.4209 & -0.0425 & 0.0259\\
\multicolumn{1}{ |c| }{$\vartheta_{25}$} & -0.2227 & -0.2127 &\phantom{-}0.0100 & 0.0334  & -0.2209 &-0.0018 & 0.0300\\
\multicolumn{1}{ |c| }{$\vartheta_{26}$} & \phantom{-}0.5632 & \phantom{-}0.5585 & \phantom{-}0.0047 & 0.0476  & \phantom{-}0.4814 & \phantom{-}0.0818 & 0.0356 \\
 \hline
\multicolumn{1}{ |c| }{$\vartheta_{27}$} & \phantom{-}1 & \phantom{-}1.0002 & \phantom{-}0.0002 & 0.0030  & \phantom{-}0.9995 & \phantom{-}0.0005 & 0.0033 \\
\multicolumn{1}{ |c| }{$\vartheta_{28}$} & \phantom{-}2 & \phantom{-}2.0000 & \phantom{-}0.0000 & 0.0079 & \phantom{-}2.0004 & -0.0004 & 0.0091\\
  \hline
\end{tabular}
  \caption{Sample mean, bias and sample standard deviation  of 350 replicates of QML estimators of a three-dimensional NIG-driven and Brownian-driven cointegrated state space process.}
  \label{TableSimNIG3}
\end{table}

\section{Conclusion} \label{sec:conclusion}
The main contribution of the present paper is the development of a QML estimation procedure for the parameters of cointegrated solutions of continuous-time
linear state space models sampled equidistantly allowing flexible margins.
%To the best of our knowledge it is the first paper estimating the parameters of
%a cointegrated output $Y$ of a linear state space and MCARMA model, respectively.
We showed that the QML estimator for the long-run parameter is super-consistent and that of the short-run parameter is consistent.
Moreover,  the QML estimator for the long-run parameter converges with a $n$-rate to a mixed normal distribution, whereas the
short-run parameter converges with a $\sqrt{n}$-rate to a normal distribution. In the simulation study, we saw that the estimator works quite well in practice.

In this paper, we lay the mathematical basis for QML for cointegrated solutions of state-space models. In a separate paper Fasen-Hartmann and Scholz~\cite{FasenScholz17}
we present an algorithm to construct canonical forms for the state space model satisfying  the assumptions of this paper, which is necessary to apply the method to data.
 We decided to split the paper
because the introduction into a canonical form  is quite lengthy and would blow up
the present paper.
Moreover, a drawback of our estimation procedure is that we assume that the cointegration rank is known in advance which is not the
case in reality. First, we have  to estimate and test the cointegration rank. For this it is possible to incorporate some
 well-known results for estimating and testing the cointegration rank for cointegrated VARMA processes
as, e.g., presented in Bauer and Wagner~\cite{BauerWagner2002b}, Lütkepohl and Claessen~\cite{LuetkepohlClaessen}, Saikkonen~\cite{Saikkonen92}, Yap and Reinsel~\cite{YapReinsel95a}. This will
also be considered in Fasen-Hartmann and Scholz~\cite{FasenScholz17}. Some parts of Fasen-Hartmann and Scholz~\cite{FasenScholz17} can already be found in Scholz~\cite{Scholz}.

\subsection*{Acknowledgement}
We would like to thank Eckhard Schlemm and Robert Stelzer, who kindly provided the MATLAB code for the simulation and parameter estimation
of stationary state space  processes. We adapted the code in order to include not only stationary state space models, but also cointegrated state space
models.

\begin{appendix}

\section{Auxiliary results} \label{Appendix: Auxiliary results}

\subsection{Asymptotic results}
For the derivation of the asymptotic behavior of our estimators we require
the asymptotic behavior of the standardized score vector and the standardized  Hessian matrix. To obtain these asymptotic results
 we use the next proposition.
 %The proofs of this section are moved to the Appendix.

\begin{proposition}
\label{PropUniformConvRes1} \label{PropLimitResultsFuncofDYY}
Let \autoref{AssMBrownMot} hold.
Furthermore, let $(L_k^{(h)})_{k\in\N_0}:=(L(kh))_{k\in\N_0}$ be the Lévy process sampled at distance $h$ and $\Delta L_k^{(h)}=L(kh)-L((k-1)h)$.
 Define for $n\in\N,k\in\N_0$,
\beao
    \xi_k^{(h)}:= \left( \begin{array}{l}
        \Delta Y_k^{(h)}\\
        Y_{st,k}^{(h)}\\
        \Delta L_k^{(h)}
    \end{array}
    \right) \quad \text{ and } \quad  S_n^{(h)}:=\sum_{k=1}^n\xi_k^{(h)}.
\eeao
Let $\overline{\mathsf{l}}(z,\vartheta):=\sum_{i=0}^{\infty}\overline{\mathsf{l}}_i(\vartheta)z^i$ and $\underline{\mathsf{l}}(z,\vartheta):=\sum_{i=0}^{\infty}\underline{\mathsf{l}}_i(\vartheta)z^i$, $\vartheta\in\Theta$, $z\in\C$, where $(\overline{\mathsf{l}}_i(\vartheta))_{i\in\N_0}$ is a deterministic uniformly exponentially bounded continuous
 matrix sequence in $\R^{\overline{d}\times (2d+m)}$ and similarly $(\underline{\mathsf{l}}_i(\vartheta))_{i\in\N_0}$ is a nonstochastic uniformly exponentially bounded continuous
 matrix sequence in $\R^{\underline{d}\times (2d+m)}$. Moreover, $\overline{\Pi}(\vt)\in \R^{\overline{d}\times (2d+m)},\,\underline{\Pi}(\vt)\in \R^{\underline{d}\times (2d+m)}$ are continuous matrix functions as well. We write \linebreak
$\underline{\mathsf{l}}(\mathsf{B},\vartheta)\xi^{(h)}=(\underline{\mathsf{l}}(\mathsf{B},\vartheta)\xi_k^{(h)})_{k\in\N_0}$ with $\underline{\mathsf{l}}(\mathsf{B},\vartheta)\xi_k^{(h)}=
\sum_{i=0}^{\infty}\overline{\mathsf{l}}_i(\vartheta)\xi_{k-i}^{(h)}$ and similarly $\overline{\mathsf{l}}(\mathsf{B},\vartheta)\xi^{(h)}$. \linebreak
Let $(W(r))_{0\leq r\leq 1}=((W_1(r)^{\mathsf{T}},W_2(r)^{\mathsf{T}},W_3(r)^{\mathsf{T}})^{\mathsf{T}})_{0\leq r\leq 1}$ be the Brownian motion as defined on p.~\pageref{Brownian motion}.
\begin{comment}
Finally, let $(W(r))_{0\leq r\leq 1}=((W_1(r)^{\mathsf{T}},W_2(r)^{\mathsf{T}},W_3(r)^{\mathsf{T}})^{\mathsf{T}})_{0\leq r\leq 1}$ be a $(2d+m)$-dimensional Brownian motion  with
covariance matrix
\begin{eqnarray*}
    \Sigma_W=\psi(1)\int_0^h\left(\begin{array}{cc}
        \Sigma_L &  \Sigma_L\e^{A_2^Tu} \\
        \e^{A_2u}B_2\Sigma_L & \e^{A_2u}B_2\Sigma_L B_2^T\e^{A_2^Tu}
    \end{array}\right)\,du\psi(1)^{\mathsf{T}}
\end{eqnarray*}
where
\beao
       && \psi_0:= \left( \begin{array}{cc}
        C_1B_1 & C_2\\
        0_{d\times m}      & C_2\\
        I_{m\times m} & 0_{m\times N-C}
    \end{array}
    \right),  \quad\quad
    \psi_j= \left( \begin{array}{cc}
      0_{d\times m}   & C_2(\e^{A_2hj}-\e^{A_2h(j-1)})\\
        0_{d\times m}      & C_2\e^{A_2hj}\\
        I_{m\times m} & 0_{m\times N-C}
    \end{array}
    \right), \, j\geq 1, \quad
    \text{  }\\
    &&  \psi(z)=\sum_{j=0}^\infty\psi_jz^j, \quad z\in\C,
\eeao
and $(W_i(r))_{0\leq r\leq 1}$, $i=1,2$,
are $d$-dimensional Brownian motions and $(W_3(r))_{0\leq r\leq 1}$ is an $m$-dimensional Brownian motion.
\end{comment}
Then, %Denote by $\Gamma_{\overline{\xi}\underline{\xi}}(h):=\E\overline{\xi}_k\underline{\xi}_{k+h}^\mathsf{T}$.
\begin{itemize}
\item[(a)] ${\displaystyle
  \sup_{\vt\in\Theta}\left\|\frac{1}{n}\sum_{k=1}^n[\overline{\mathsf{l}}(\mathsf{B},\vartheta) {\xi}_k^{(h)}] [\underline{\mathsf{l}}(\mathsf{B},\vartheta){\xi}_{k+j}^{(h)}]^\mathsf{T}   - \E\left[[\overline{\mathsf{l}}(\mathsf{B},\vartheta) {\xi}_1^{(h)}] [ \underline{\mathsf{l}}(\mathsf{B},\vartheta){\xi}_{1+j}^{(h)}]^\mathsf{T}   \right]\right\|\ccip 0
}$, \quad $j\in\N_0$,
%and condition \eqref{eqStochEquicon} holds for $\mathcal{X}_n(\vartheta)=\frac{1}{n}\sum_{k=1}^n\overline{L}(\mathsf{B},\vartheta) {\xi}_k {\xi}_{k+h}^\mathsf{T}   \underline{L}(\mathsf{B},\vartheta)^\mathsf{T}$.
\item[(b)] ${\displaystyle
n^{-2}\sum\limits_{k=1}^{n}
    \overline{\Pi}(\vartheta)S_{k-1}^{(h)}[S_{k-1}^{(h)}]^{\mathsf{T}}\underline{\Pi}(\vartheta)^\mathsf{T}
    %\label{eqContConvRes1}
    \ccid
    \overline{\Pi}(\vartheta) \int_0^1 W(r)W(r)^\mathsf{T}\,\dif r\,\underline{\Pi}(\vartheta)^\mathsf{T},
     }$
 %    and condition \eqref{eqStochEquicon} holds for  $\mathcal{X}_n(\vartheta) =n^{-1}\sum\limits_{k=1}^{n}
 %   \overline{\Pi}(\vartheta)S_{k-1}S_{k-1}^{\mathsf{T}}\underline{\Pi}(\vartheta)^\mathsf{T}$ with $\gamma_1>\frac12,\gamma_2>0$;
\item[(c)]
${\displaystyle
   n^{-1}\sum\limits_{k=1}^{n}\overline{\Pi}(\vartheta)S_{k-1}^{(h)}[\underline{\mathsf{l}}(\mathsf{B},\vartheta)\xi_{k}^{(h)}]^\mathsf{T}
      \ccid \overline{\Pi}(\vartheta)\int_0^1 W(r)\dif W(r)^\mathsf{T} \underline{\mathsf{l}}(1,\vartheta)^{\mathsf{T}}
 +\sum_{j=1}^\infty\Gamma_{\overline{\Pi}(\vt)\xi^{(h)},\underline{\mathsf{l}}(\mathsf{B},\vartheta)\xi^{(h)}}(j).
}$
 %    where $\Gamma_{\overline{\Pi}(\vt)\xi,\underline{\mathsf{l}}(\mathsf{B},\vartheta)\xi}(j)=\E([\overline{\Pi}(\vt)\xi_1][\underline{\mathsf{l}}(\mathsf{B},\vartheta)\xi_{1+j}]^{\mathsf{T}})$, $j\in\N$.
  %   and condition \eqref{eqStochEquicon} holds for \linebreak $\mathcal{X}_n(\vartheta)=n^{-1}\sum\limits_{k=1}^{n}\overline{\Pi}(\vartheta)S_{k-1}[\underline{\mathsf{l}}(\mathsf{B},\vartheta)\xi_{k}]^{\mathsf{T}}$  with $\gamma_1,\gamma_2>0$.
\end{itemize}
The stated weak convergence results  also hold jointly.
\end{proposition}
Before we state the proof of Proposition~\ref{PropLimitResultsFuncofDYY} we need some  auxiliary results.

\begin{lemma}  \label{Lemma auxiliary}
Let $\psi$ be defined as in \eqref{def psi}. Then, the following statements hold.
\begin{itemize}
    \item[(a)] $\E(\xi_k^{(h)})=0_{2d+m}$ and $\E\|\xi_k^{(h)}\|^{4}<\infty$.
    \item[(b)] $\frac{1}{n}\sum_{k=1}^n\xi_k^{(h)}\cip 0_{2d+m}$ \quad and \quad
                $\frac{1}{n}\sum_{k=1}^n\xi_k^{(h)}[\xi_{k+l}^{(h)}]^{\mathsf{T}}\cip \E(\xi_1^{(h)}[\xi_{1+l}^{(h)}]^{\mathsf{T}})=:\Gamma_{\xi^{(h)}}(l)$, \quad $l\in\N_0$.
    \item[(c)] $\sum_{l=0}^\infty \E\|\xi_1^{(h)}[\xi_{1+l}^{(h)}]^{\mathsf{T}}\|<\infty$.
    \item[(d)] $\left(\frac{1}{\sqrt{n}}S_{\lfloor nr\rfloor }^{(h)}\right)_{0\leq r\leq 1}\cid (\psi(1)W^*(r))_{0\leq r\leq 1}$ where
        $(W^*(r))_{0\leq r\leq 1}$ is a $(m+(N-c))$-dimensional Brownian motion with covariance matrix
        \beao
            \Sigma_{W^*}=\int_0^h\left(\begin{array}{cc}
        \Sigma_L & \Sigma_LB_2^T\e^{A_2^Tu} \\
         \e^{A_2u}B_2\Sigma_L & \e^{A_2u}B_2\Sigma_L B_2^T\e^{A_2^Tu}
    \end{array}\right)\,du
        \eeao
        and $\psi$ is defined as in \eqref{def psi}.
    \item[(e)] $\frac{1}{n}\sum_{k=2}^nS_{k-1}^{(h)}[\xi_k^{(h)}]^{\mathsf{T}}\cid \psi(1) \int_0^1W^*(r)\dif W^*(r)^{\mathsf{T}}\psi(1)^{\mathsf{T}}+\sum_{l=1}^\infty\Gamma_{\xi^{(h)}}(l)$.
\end{itemize}
\end{lemma}

\begin{proof}
We shortly sketch the proof. The sequence $(\xi_k^{(h)})_{k\in\N}$ has the MA-representation
\beam \label{MA}
    \left( \begin{array}{l}
        \Delta Y_k^{(h)}\\
        Y_{st,k}^{(h)}\\
        \Delta L_k^{(h)}
    \end{array}
    \right)=\xi_k^{(h)}=\sum_{j=-\infty}^k\psi_{k-j}\eta_j^{(h)}
\eeam
with the iid sequence
$
    \eta_k^{(h)}:=\left( \begin{array}{l}
        \Delta L_k^{(h)\,\mathsf{T}},\,\,
        R_k^{(h)\,{\mathsf{T}}}
    \end{array}
    \right)^{\mathsf{T}}
$
and $R_k^{(h)}:=\int_{(k-1)h}^{kh}\e^{A_2(kh-u)}B_2\,\dif L_u$.   Hence,  $(\xi_k^{(h)})_{k\in\N}$ is  stationary and ergodic
 as a  measurable map of a stationary ergodic process (see  Krengel~\cite[Theorem 4.3 in Chapter 1]{Krengel1985}). \\[2mm]
(a) \, is due to \autoref{AssMBrownMot}.\\[2mm]
(b) \, is a direct consequence of Birkhoff's ergodic theorem.\\[2mm]
(c) \, follows from $\E\|\eta_k^{(h)}\|^4<\infty$, $\|\psi_j\|\leq \mathfrak{C}\rho^j$ for some $\mathfrak{C}>0$, $0<\rho<1$ and the MA-representation \eqref{MA}. \\[2mm]
(d,e) \, are conclusions of Johansen \cite[Theorem B.13]{Johansen1995} and the MA-representation \eqref{MA}.
\end{proof}

\begin{proof}[Proof of Proposition~\ref{PropLimitResultsFuncofDYY}]
(a) \, The proof   follows directly by Theorem 4.1  of Saikkonen \cite{Saikkonen1993} and the comment of Saikkonen~\cite[p.163, line 4]{Saikkonen1993} if we can show that Assumption 4.1 and 4.2 of that paper are satisfied.
Since we have uniformly exponentially bounded families of matrix sequences,  Saikkonen \cite[Assumption 4.1]{Saikkonen1993} is obviously satisfied.
Saikkonen \cite[Assumption 4.2]{Saikkonen1993} is satisfied due to Lemma~\ref{Lemma auxiliary}.

Note that we have two different coefficient matrices, whereas the results in Saikkonen \cite{Saikkonen1993} are proved for the same coefficient matrix.
However, Saikkonen~\cite[Theorem 4.1]{Saikkonen1993} also holds if the coefficient matrices are different as long as each sequence of matrix coefficients satisfies the necessary conditions as mentioned in the paper of Saikkonen~\cite[p. 163]{Saikkonen1993}.\\
(b,c) \, Due to Lemma~\ref{Lemma auxiliary}, Saikkonen~\cite[Assumption 4.3]{Saikkonen1993} is satisfied as well. Hence, we can conclude
the weak convergence result from Saikkonen~\cite[Theorem 4.2(iii)]{Saikkonen1993} and \cite[Theorem 4.2(iv)]{Saikkonen1993}, respectively.
 \end{proof}

\subsection{Lipschitz continuity results}

A kind of local Lipschitz continuity in $\vt^0$ for the processes in Proposition~\ref{PropUniformConvRes1} is presented next. The local Lipschitz continuity in $\vt^0$  implies, in particular,
local stochastic equicontinuity in $\vt^0$.
However, this kind of local Lipschitz continuity in $\vt^0$ is stronger than local stochastic equicontinuity in $\vt^0$ so that we are
 not able to apply the stochastic equicontinuity results of Saikkonen~\cite{Saikkonen1993,Saikkonen1995} directly. The stochastic equicontinuity
 of the process in Proposition~\ref{Proposition3.3}(a) and (c) can be deduced with some effort from Saikkonen~\cite{Saikkonen1993,Saikkonen1995} but the process in
 Proposition~\ref{Proposition3.3}(b)
 is not covered in these papers.
 %But on the other hand,
%a local Lipschitz condition in $\vt^0$ is sufficient whereas Saikkonen~\cite{Saikkonen1993,Saikkonen1995} presents global stochastic equicontinuity conditions.

%
\begin{proposition} \label{Proposition3.3}
Let the assumption and notation of Proposition~\ref{PropUniformConvRes1} hold. Assume further that $\overline{\Pi}(\vt)$,
$\underline{\Pi}(\vt)$ are Lipschitz-continuous and the sequence of matrix functions
 $(\nabla_\vt(\overline{\mathsf{l}}_i(\vartheta)))_{i\in\N_0}$  and \linebreak $(\nabla_\vt(\underline{\mathsf{l}}_i(\vartheta)))_{i\in\N_0}$ are uniformly
 exponentially bounded.
\begin{itemize}
  \item[(a)] Define $\mathcal{X}_n(\vt)=\sum\limits_{k=1}^{n}
    \overline{\Pi}(\vartheta)S_{k-1}^{(h)}[S_{k-1}^{(h)}]^{\mathsf{T}}\underline{\Pi}(\vartheta)^\mathsf{T}$.
    Then,     \beam \label{EQ1}
        \|\mathcal{X}_n(\vt)-\mathcal{X}_n(\vt^0)\|\leq \mathfrak{C}\|\vt-\vt^0\|\left\|\sum\limits_{k=1}^{n}S_{k-1}^{(h)}[S_{k-1}^{(h)}]^{\mathsf{T}}\right\|.
    \eeam
    If additionally  $\underline{\Pi}(\vt^0)=0_{\underline{d}\times(2d+m)}$ and $\overline{\Pi}(\vt^0)=0_{\overline{d}\times(2d+m)}$ then
     \beam \label{EQ2}
        \|\mathcal{X}_n(\vt)-\mathcal{X}_n(\vt^0)\|\leq \mathfrak{C}\|\vt-\vt^0\|^2\left\|\sum\limits_{k=1}^{n}S_{k-1}^{(h)}[S_{k-1}^{(h)}]^{\mathsf{T}}\right\|.
    \eeam
  \item[(b)] Define $\mathcal{X}_n(\vt)=\sum\limits_{k=1}^{n}\overline{\Pi}(\vartheta)S_{k-1}^{(h)}[\underline{\mathsf{l}}(\mathsf{B},\vartheta)\xi_{k}^{(h)}]^\mathsf{T}$.
  Then,   \beam \label{EQ3}
    \|\mathcal{X}_n(\vt)-\mathcal{X}_n(\vt^0)\| \leq \mathfrak{C} n\|\vt-\vt^0\| V_n
  \eeam
  where
  \beam \label{EQ4}
        V_n&=& \left\|\frac{1}{n}\sum_{k=1}^n S_{k-1}^{(h)}[\xi_{k}^{(h)}]^T\right\|+\|S_{n}^{(h)}\|[\mathsf{k}_\rho(\mathsf{B})\|\xi_n^{(h)}\|]
        +\frac{1}{n}\sum_{k=1}^n\|\Delta S_{k}^{(h)}\|[\mathsf{k}_\rho(\mathsf{B})\|\xi_k^{(h)}\|]\nonumber\\
        &&\quad+\left\|\frac{1}{n}\sum\limits_{k=1}^nS_{k-1}^{(h)}\left[\underline{\mathsf{l}}(\mathsf{B},\vartheta^0)\xi_k^{(h)}\right]^\mathsf{T}\right\|\nonumber\\
        &=&\mathcal{O}_p(1), \quad
  \eeam
  $\mathsf{k}_\rho(z)=\mathfrak{c}\sum_{j=0}^\infty\rho^j z^j$ for some $0<\rho<1$, $\mathfrak{c}>0$, and $\mathsf{k}_\rho(\mathsf{B})\|\xi_k^{(h)}\|:=\mathfrak{c}\sum_{j=0}^\infty\rho^j\|\xi_{k-j}^{(h)}\|$.
  \item[(c)] Define $\mathcal{X}_n(\vt)=\sum\limits_{k=1}^{n}[\overline{\mathsf{l}}(\mathsf{B},\vartheta)\xi_{k}^{(h)}][\underline{\mathsf{l}}(\mathsf{B},\vartheta)\xi_{k}^{(h)}]^\mathsf{T}$.
  Then, there exists a random variable  \linebreak $ Q_n=\mathcal{O}_p(1)$ so that
  \beam \label{EQ5}
    \|\mathcal{X}_n(\vt)-\mathcal{X}_n(\vt^0)\| \leq \mathfrak{C} n\|\vt-\vt^0\| Q_n.
  \eeam
\end{itemize}
In particular, $V_n+Q_n=\mathcal{O}_p(1)$.
\end{proposition}

\begin{proof} $\mbox{}$\\ %[Proof of Proposition~\ref{Proposition3.3}] $\mbox{}$\\
(a) \, We have the upper bound
\beao
    \lefteqn{\|\mathcal{X}_n(\vt)-\mathcal{X}_n(\vt^0)\|}\\
    &&\leq
    \left\|\sum\limits_{k=1}^{n}
    [\overline{\Pi}(\vartheta)-\overline{\Pi}(\vartheta^0)]S_{k-1}^{(h)}[S_{k-1}^{(h)}]^{\mathsf{T}}\underline{\Pi}(\vartheta)^\mathsf{T}\right\|
    + \left\|\sum\limits_{k=1}^{n}\overline{\Pi}(\vartheta^0)S_{k-1}^{(h)}[S_{k-1}^{(h)}]^{\mathsf{T}}[\underline{\Pi}(\vartheta)-\underline{\Pi}(\vartheta^0)]^\mathsf{T}\right\|\\
    &&\leq\left(\|\overline{\Pi}(\vartheta)-\overline{\Pi}(\vartheta^0)\|\|\underline{\Pi}(\vartheta)\|+\|\underline{\Pi}(\vartheta)-\underline{\Pi}(\vartheta^0)\|\|\overline{\Pi}(\vartheta^0)\|\right)
    \left\|\sum\limits_{k=1}^{n}S_{k-1}^{(h)}[S_{k-1}^{(h)}]^{\mathsf{T}}\right\|.
\eeao
Since $\overline{\Pi}(\vt)$ and $\underline{\Pi}(\vt)$ are Lipschitz continuous, $$\max(\|\overline{\Pi}(\vartheta)-\overline{\Pi}(\vartheta^0)\|,\|\underline{\Pi}(\vartheta)-\underline{\Pi}(\vartheta^0)\|)\leq \mathfrak{C}\|\vt-\vt^0\|$$
and $\sup_{\vt\in\Theta}\max(\|\overline{\Pi}(\vartheta)\|,\|\underline{\Pi}(\vartheta)\|)\leq \mathfrak{C}$. Thus, \eqref{EQ1} holds.
Moreover, \eqref{EQ2} is valid because for $\underline{\Pi}(\vt^0)=0_{\underline{d},2d+m}$, $\overline{\Pi}(\vt^0)=0_{\overline{d},2d+m}$ the upper bound
\beao
    \|\mathcal{X}_n(\vt)-\mathcal{X}_n(\vt^0)\|=\|\mathcal{X}_n(\vt)\|
        &\leq& \|\underline{\Pi}(\vt)\| \|\overline{\Pi}(\vt)\| \left\|\sum\limits_{k=1}^{n}S_{k-1}^{(h)}[S_{k-1}^{(h)}]^{\mathsf{T}}\right\|\\
        &\leq& \|\underline{\Pi}(\vt)-\underline{\Pi}(\vt^0)\| \|\overline{\Pi}(\vt)-\overline{\Pi}(\vt^0)\| \left\|\sum\limits_{k=1}^{n}S_{k-1}^{(h)}[S_{k-1}^{(h)}]^{\mathsf{T}}\right\|\\
        &\leq &\mathfrak{C}\|\vt-\vt^0\|^2 \left\|\sum\limits_{k=1}^{n}S_{k-1}^{(h)}[S_{k-1}^{(h)}]^{\mathsf{T}}\right\|
\eeao
is valid.\\[2mm]
(b) \, Using  a Taylor expansion leads to
\begin{eqnarray*}
    \text{vec}(\underline{\mathsf{l}}(z,\vartheta))-\text{vec}(\underline{\mathsf{l}}(z,\vartheta^0))=\sum_{j=0}^\infty[ \text{vec}(\underline{\mathsf{l}}_j(\vt))- \text{vec}(\underline{\mathsf{l}}_j(\vt^0))]z^j
        =\sum_{j=0}^\infty \nabla_{\vt}\text{vec}(\underline{ \mathsf{l}}_j^*(\underline{\vt(j)}))(\vt-\vt^0) z^j,
\end{eqnarray*}
 where $\text{vec}(\underline{ \mathsf{l}}_j^*(\underline{\vt(j)}))$ denotes the matrix whose $i^{th}$ row is equal to the $i^{th}$ row of $\text{vec}(\underline{\mathsf{l}}_j(\underline{\vt^{i}(j)}))$ with $\underline{\vartheta^i(j)}\in\Theta$ such that $\Vert \underline{\vartheta^i(j)}-\vartheta^0\Vert \leq\Vert \vartheta-\vartheta^0\Vert$ for $i=1,\ldots,\underline{d}(2d+m)$.
 Due to assumption, \linebreak  $\|\nabla_{\vt}\text{vec}(\underline{ \mathsf{l}}_j^*(\underline{\vt(j)}))\|\leq \mathfrak{C} \rho^j$ for $j\in\N_0$ and some $0<\rho<1$ so that
\begin{eqnarray} \label{C.6}
    \|\underline{\mathsf{l}}(z,\vartheta)-\underline{\mathsf{l}}(z,\vartheta^0)\|\leq \mathsf{k}_\rho(|z|)\|\vt-\vt^0\|.
\end{eqnarray}
 Hence,
\begin{eqnarray} \label{C.4}
 \|\left(\underline{\mathsf{l}}(\mathsf{B},\vartheta)-\underline{\mathsf{l}}(\mathsf{B},\vartheta^0)\right)\xi_k^{(h)}\|\leq \|\vt-\vt^0\|\left[\mathsf{k}_\rho(\mathsf{B})\|\xi_k^{(h)}\|\right].
 \end{eqnarray}
 Define $\underline{\mathsf{l}}^{\nabla}(z,\vartheta,\vartheta^0):=\underline{\mathsf{l}}(z,\vartheta)-\underline{\mathsf{l}}(z,\vartheta^0)
=:\sum_{j=0}^\infty \underline{\mathsf{l}}_j^{\nabla}(\vt,\vt^0)z^j$. Then, we apply a Beveridge-Nelson decomposition (see Saikkonen \cite[(9)]{Saikkonen1993}) to get $$\underline{\mathsf{l}}^{\nabla}(\mathsf{B},\vt,\vt^0)\xi_{k}^{(h)}=\underline{\mathsf{l}}^{\nabla}(1,\vt,\vt^0)\xi_{k}^{(h)}+\eta_k(\vt,\vt^0)-\eta_{k-1}(\vt,\vartheta^0)$$ with $\eta_k(\vt,\vartheta^0):=-\sum_{j=0}^\infty \sum_{i=j+1}^\infty \underline{\mathsf{l}}_i^{\nabla}(\vt,\vt^0)\xi_{k-j}^{(h)}$. Thus,
\begin{eqnarray*}
        &&\hspace*{-0.8cm}\frac{1}{n}\sum\limits_{k=1}^n\overline{\Pi}(\vt)S_{k-1}^{(h)}\left[\left(\underline{\mathsf{l}}(\mathsf{B},\vartheta)-\underline{\mathsf{l}}(\mathsf{B},\vartheta^0)\right)\xi_k^{(h)}\right]^\mathsf{T}\\
        &&\hspace*{-0.8cm}\quad=\overline{\Pi}(\vt)\frac{1}{n}\sum_{k=1}^n S_{k-1}^{(h)} [ \xi_{k}^{(h)}]^{\mathsf{T}}\underline{\mathsf{l}}^{\nabla}(1,\vt,\vt^0)^{\mathsf{T}}+\overline{\Pi}(\vt)\frac{1}{n}\sum_{k=1}^nS_{k-1}^{(h)}[\eta_k(\vt,\vt^0)]^{\mathsf{T}}
        -\overline{\Pi}(\vt)\frac{1}{n}\sum_{k=1}^nS_{k-1}^{(h)}[\eta_{k-1}(\vt,\vt^0)]^\mathsf{T}\\
        &&\hspace*{-0.8cm}\quad=\overline{\Pi}(\vt)\frac{1}{n}\sum_{k=1}^n S_{k-1}^{(h)}[\xi_{k}^{(h)}]^{\mathsf{T}}\underline{\mathsf{l}}^{\nabla}(1,\vt,\vt^0)^\mathsf{T}+\overline{\Pi}(\vt)S_{n}^{(h)}[\eta_{n}(\vt,\vt^0)]^T
        -\overline{\Pi}(\vt)\frac{1}{n}\sum_{k=1}^n\Delta S_{k}^{(h)}[\eta_k(\vt,\vt^0)]^\mathsf{T}.
\end{eqnarray*}
Due to \eqref{C.6}, $\|\underline{\mathsf{l}}^{\nabla}(1,\vt,\vt^0)\|\leq \mathfrak{C}\|\vt-\vt^0\|$ and $\|\underline{\mathsf{l}}_j^{\nabla}(\vt,\vt^0)\|\leq \mathfrak{C} \|\vt-\vt^0\|\rho^j$ so that \linebreak
$\|\eta_k(\vt,\vartheta^0)\|\leq \|\vt-\vt^0\|\left[\mathsf{k}_\rho(\mathsf{B})\|\xi_k^{(h)}\|\right]$ as well.  Finally,  we receive
\beam \label{B1}
    \lefteqn{\left\|\overline{\Pi}(\vt)\frac{1}{n}\sum\limits_{k=1}^nS_{k-1}^{(h)}\left[\left(\underline{\mathsf{l}}(\mathsf{B},\vartheta)-\underline{\mathsf{l}}(\mathsf{B},\vartheta^0)\right)\xi_k^{(h)}\right]^\mathsf{T}\right\|}\\
    &&\leq \mathfrak{C}\|\vt-\vt^0\|\left[ \left\|\frac{1}{n}\sum_{k=1}^n S_{k-1}^{(h)}[\xi_{k}^{(h)}]^{\mathsf{T}}\right\|+\|S_{n}^{(h)}\|\left[\mathsf{k}_\rho(\mathsf{B})\|\xi_n^{(h)}\|\right]
        +\frac{1}{n}\sum_{k=1}^n\|\Delta S_{k}^{(h)}\|\left[\mathsf{k}_\rho(\mathsf{B})\|\xi_k^{(h)}\|\right]\right],  \nonumber \quad  \quad
\eeam
and
\beam \label{B2}
    {\left\|[\overline{\Pi}(\vt)-\overline{\Pi}(\vt^0)]\frac{1}{n}\sum\limits_{k=1}^nS_{k-1}^{(h)}\left[\underline{\mathsf{l}}(\mathsf{B},\vartheta^0)\xi_k^{(h)}\right]^{\mathsf{T}}\right\|}
        \leq \mathfrak{C}\|\vt-\vt^0\|\left\|\frac{1}{n}\sum\limits_{k=1}^nS_{k-1}^{(h)}\left[\underline{\mathsf{l}}(\mathsf{B},\vartheta^0)\xi_k^{(h)}\right]^\mathsf{T}\right\|. \quad
\eeam
Then, \eqref{B1} and \eqref{B2} result in \eqref{EQ3}.
 \\[2mm]
It remains to prove \eqref{EQ4}. The first term $\frac{1}{n}\sum_{k=1}^n S_{k-1}^{(h)}[\xi_{k}^{(h)}]^T$ in the definition of $V_n$ is $\mathcal{O}_p(1)$ by Lemma~\ref{Lemma auxiliary}.
Moreover, $\frac{1}{n}S_n^{(h)}=\frac{1}{n}\sum_{k=1}^n\xi_k^{(h)}=\mathcal{O}_p(1)$ by Birkhoff's Ergodic Theorem; similarly the third term
 $\frac{1}{n}\sum_{k=1}^n\|\Delta S_{k}^{(h)}\|\left[\mathsf{k}_\rho(\mathsf{B})\|\xi_k^{(h)}\|\right]$ is $\mathcal{O}_p(1)$  by Birkhoff's Ergodic Theorem.
Finally, the last term $\frac{1}{n}\sum\limits_{k=1}^nS_{k-1}^{(h)}\left[\underline{\mathsf{l}}(\mathsf{B},\vartheta^0)\xi_k^{(h)}\right]^\mathsf{T}$ is $\mathcal{O}_p(1)$ by Proposition~\ref{PropUniformConvRes1}(c). \\[2mm]
(c) \, The proof is similarly to the proof of (b).
\end{proof}

\end{appendix}

\section{Properties of the pseudo-innovations} \label{Section: Properties linear innovation}

In this section we present some probabilistic properties of the pseudo-innovations. Therefore, we start with an auxiliary lemma on
the functions defining the pseudo-innovations and the prediction covariance matrix
which we require for the  pseudo-innovations to be partial differentiable.

\begin{lemma} \label{Lemma 2.3}
Let \autoref{AssMBrownMot} hold.
\begin{itemize}
    \item[(a)] The matrix functions $\Pi(\vt)$, $\mathsf{k}(z,\vt)$, $V_\vt^{(h)}$ and $(V_\vt^{(h)})^{-1}$ are Lipschitz continuous on $\Theta$
        and three times partial differentiable.
    \item[(b)] $\sup_{\vt\in\Theta}\|(V_\vt^{(h)})^{-1}\|\leq \mathfrak{C}$.
    \item[(c)] $\inf_{\vt\in\Theta}\sigma_{\min}((V_\vt^{(h)})^{-1})>0$.
\end{itemize}
\end{lemma}
\begin{proof}
(a)  is a consequence of \autoref{AssMBrownMot} and Scholz~\cite[Lemma 5.9.3]{Scholz}. However, Scholz~\cite[Lemma 5.9.3]{Scholz} shows only the twice continuous differentiability but the proof of the existence of the third partial differential is analog.\\
(b) follows from (a) and the compactness of $\Theta$.\\
(c) Due to  Scholz~\cite[Lemma 5.9.1]{Scholz} the matrix $(V_\vt^{(h)})^{-1}$ is non-singular. Hence, we can conclude the statement from (a) and the compactness of $\Theta$.
\end{proof}
%so that on the
%one hand, $(\Pi(\vartheta^0)Y_t)_{t\geq 0}=(\Pi(\vartheta^0)C_1X_1(t_0)+\Pi(\vartheta^0)Y_{st,t})_{t\geq 0}$ is stationary and on the other hand, $( \varepsilon_k^{(h)}(\vartheta^0))_{k\in\N}$
%is stationary as well.
%\end{remark}
%
%
%The derivatives with respect to the stationary and non-stationary parameters of the pseudo-innovations are given by
%\begin{align*}
%\partial_i^1\varepsilon_k^{(h)}(\vartheta)
%&=\partial_i^1\overline{k}(B,\vartheta)\Delta Y_k^{(h)}-\partial_i^1\Pi(\vartheta)Y_{k-1}^{(h)},\quad\text{  for } i\in\{1,\ldots,s_1\}
%\\
%\text{and }\quad\partial_i^{st}\varepsilon_k^{(h)}(\vartheta)
%&=\partial_i^{st}\overline{k}(B,\vartheta)\Delta Y_k^{(h)}-\partial_i^{st}\Pi(\vartheta)Y_{k-1}^{(h)},\quad\text{ for } i\in\{1,\ldots,s_2\}.
%\end{align*}
%Note that $\partial_i^{st}\varepsilon_k^{(h)}(\vartheta^0)$ is stationary due to the fact that $\partial_i^{st}\Pi(\vartheta^0)Y^{(h)}=\left(\partial_i^{st}\alpha(\vt^0)\right)\beta(\vt_1^0)^\mathsf{T}Y_{st}^{(h)}$ and  $\partial_i^{1}\varepsilon_k^{(h)}(\vartheta^0)$ is non-stationary since
%$\partial_i^{1}\Pi(\vartheta^0)Y^{(h)}=\left(\partial_i^{1}\alpha(\vt^0)\right)\beta(\vt_1^0)^\mathsf{T}Y_{st}^{(h)}+\alpha(\vt^0)\partial_i^{1}\beta(\vt_1^0)^\mathsf{T}Y^{(h)}$.
%The first summand still cancels out the non-stationarity in contrast to the second one, as $\partial_i^{1}\beta(\vt_1^0)^\mathsf{T}$ does in general not lie in the cointegration space.
%
Thus, the pseudo-innovations are three times differentiable and we receive
an analog version of Lemma~\ref{LemInnovaSeqProp}.
\begin{lemma}
\label{LemDerInnovaSeqProp}
Let \autoref{AssMBrownMot} hold and let $u,v\in\{1,\ldots,s\}$. Then, the following results hold.
\begin{enumerate}
\item[(a)] \label{LemDerInnovaSeqPropii}
  The  matrix sequence $(\partial_v{\mathsf{k}}_j(\vartheta))_{j\in\N}$ in $\R^{d\times d}$ is uniformly exponentially bounded such that \linebreak
$
    \partial_v\varepsilon_k^{(h)}(\vartheta)=-\partial_v\Pi(\vartheta)^\mathsf{T}Y_{k-1}^{(h)}-\sum_{j=1}^\infty\partial_v{\mathsf{k}}_j(\vartheta) \Delta Y_{k-j}^{(h)}.
$%    \end{align*}
%    , i.e. there exist a positive constants $c$ and $\rho<1$, such that $\sup_{\vartheta\in\Theta}\|{K}_i^{(v)}(\vartheta)\|\leq c\rho^i$, $i\in\N$.
\item[(b)] \label{LemDerInnovaSeqPropiv} The matrix sequence $(\partial_{u,v}{\mathsf{k}}_j(\vartheta))_{j\in\N}$ in $\R^{d\times d}$ is uniformly exponentially bounded  such that
%\begin{align*}
 $   \partial_{u,v}\varepsilon_k^{(h)}(\vartheta)
    =-\partial_{u,v}\Pi(\vartheta)^\mathsf{T}Y_{k-1}^{(h)}-\sum_{j=1}^\infty \partial_{u,v} {\mathsf{k}}_j(\vartheta) \Delta Y_{k-j}^{(h)}.
$
%    \end{align*}
%    , i.e. there exist a positive constants $c$ and $\rho<1$, such that $\sup_{\vartheta\in\Theta}\|\overline{K}_i^{(u,v)}(\vartheta)\|\leq c\rho^i$, $i\in\N$.
\end{enumerate}
\end{lemma}
\begin{proof}
   Recall the representation given in Lemma~\ref{LemInnovaSeqProp} where  $({\mathsf{k}}_j(\vartheta))_{j\in\N}$ is an uniformly exponentially bounded matrix sequence.
   Then, the proof is analog to Schlemm and Stelzer~\cite[Lemma 2.11]{SchlemmStelzer2012a}.
\end{proof}

\begin{lemma} \label{Lemma 2.5}
Let \autoref{AssMBrownMot} hold and $i,j\in\{1,\ldots,s_2\}$. Then, \begin{itemize}
    \item[(a)]  $( \varepsilon_k^{(h)}(\vt^0)^{\mathsf{T}},\partial_{j}^{st}\varepsilon_k^{(h)}(\vartheta^0)^{\mathsf{T}},\partial_{i,j}^{st}\varepsilon_k^{(h)}(\vartheta^0)^{\mathsf{T}})_{k\in\N}$  is a stationary and ergodic sequence.
    \item[(b)] $\E\|\varepsilon_k^{(h)}(\vt^0)\|^4<\infty$, \quad $\E\|\partial_{j}^{st}\varepsilon_k^{(h)}(\vartheta^0)\|^4<\infty$ \quad and \quad $\E\|\partial_{i,j}^{st}\varepsilon_k^{(h)}(\vartheta^0)\|^4<\infty$.
    \item[(c)] $\E(\partial_i^{st}\varepsilon_{k}^{(h)}(\vartheta^0)\varepsilon_{k}^{(h)}(\vartheta^0)^{\mathsf{T}})=0_{d\times d}$
    \quad and \quad $\E(\partial_{i,j}^{st}\varepsilon_{k}^{(h)}(\vartheta^0)\varepsilon_{k}^{(h)}(\vartheta^0)^{\mathsf{T}})=0_{d\times d}$.
    \item[(d)] $\E(\varepsilon_{k}^{(h)}(\vartheta^0)\varepsilon_{k}^{(h)}(\vartheta^0)^{\mathsf{T}})=V_{\vt_0}^{(h)}$.
\end{itemize}
\end{lemma}
\begin{proof} $\mbox{}$\\
(a) \, Representation \eqref{eqRepContCointSSMSepar} and Lemma~\ref{Lemma 2.3} yield
\begin{align*}
%\label{3.4}
\Pi(\vartheta^0)Y^{(h)}_k&=\alpha(\vartheta_1^0,\vt_2^0)(C_{1,\vartheta_1^0}^{\perp})^\mathsf{T}Y^{(h)}_k=\Pi(\vartheta^0)Y_{st,k}^{(h)}, \\
\partial_i^{st}\Pi(\vartheta^0)Y^{(h)}_k&=\left(\partial_i^{st}\alpha(\vt^0)\right)(C_{1,\vartheta_1^0}^{\perp})^\mathsf{T}Y_{st,k}^{(h)}, \\
\partial_{i,j}^{st}\Pi(\vartheta^0)Y^{(h)}_k&=\left(\partial_{i,j}^{st}\alpha(\vt^0)\right)(C_{1,\vartheta_1^0}^{\perp})^\mathsf{T}Y_{st,k}^{(h)},
\end{align*}
and hence,
\beam \label{vareps}
    \varepsilon_k^{(h)}(\vartheta^0)&=&-\Pi(\vartheta^0)Y_{st,k-1}^{(h)}+ {\mathsf{k}}(\mathsf{B},\vt)\Delta Y_k^{(h)}, \notag\\
    \partial_i^{st}\varepsilon_k^{(h)}(\vartheta^0)&=&-\partial_i^{st}\Pi(\vartheta^0)^\mathsf{T}Y_{st,k-1}^{(h)}-\sum_{l=1}^\infty
    \partial_i^{st}{\mathsf{k}}_{l}(\vartheta^0) \Delta Y_{k-l}^{(h)},\\
    \partial_{i,j}^{st}\varepsilon_k^{(h)}(\vartheta^0)&=&-\partial_{i,j}^{st}\Pi(\vartheta^0)^\mathsf{T}Y_{st,k-1}^{(h)}-\sum_{l=1}^\infty \partial_{i,j}^{st}{\mathsf{k}}_{l}(\vartheta^0) \Delta Y_{k-l}^{(h)}. \notag
\eeam
These are obviously stationary processes.
Fasen-Hartmann and Scholz \cite[Proposition 5.8]{FasenScholz1} state already that $(\varepsilon_k^{(h)}(\vt^0))_{k\in\N}$ is  ergodic with finite second moments.
The same arguments lead to the  ergodicity of  $( \varepsilon_k^{(h)}(\vt^0)^{\mathsf{T}},\partial_{j}^{st}\varepsilon_k^{(h)}(\vartheta^0)^{\mathsf{T}},\partial_{i,j}^{st}\varepsilon_k^{(h)}(\vartheta^0)^{\mathsf{T}})_{k\in\N}$  . \\
(b) \, The finite fourth moment
of $(\varepsilon_k^{(h)}(\vt^0))_{k\in\N}$ and its partial derivatives are consequences of their series representation \eqref{vareps} with uniformly exponentially bounded coefficient matrices and the finite fourth moment of $Y_{st,k}^{(h)}$ and $\Delta Y_{k}^{(h)}$  due to Assumption~(A3) and Marquardt and Stelzer~\cite[Proposition 3.30]{MarquardtStelzer2007}.\\
(c) \, A consequence of \eqref{vareps} is that both  $\partial_i^{st}\varepsilon_{k}^{(h)}(\vartheta^0)$ and $\partial_{i,j}^{st}\varepsilon_k^{(h)}(\vartheta^0)$ are elements
of the Hilbert space generated by $\{Y_l^{(h)},-\infty<l<k\}$. But $\varepsilon_k^{(h)}(\vartheta^0)$ is orthogonal to the Hilbert space generated by $\{Y_l^{(h)},-\infty<l<k\}$ so that the statements follow. \\
(d) \, is a conclusion of the construction of the linear innovations by the Kalman filter.
\end{proof}

\section{Proof of Proposition~\ref{Lemma 2.8}} \label{Appendix B}

First, we present some auxiliary results for the proof of Proposition~\ref{Lemma 2.8}.

\begin{lemma} \label{Lemma2.9}
Let \autoref{AssMBrownMot} and \ref{Ass:int} hold. Define
$X^{(h)}_1(\vt)=\sum_{j=0}^{\infty} (\e^{A_\vt h}-K_\vt^{(h)} C_\vt)^{j}K_\vt^{(h)}
     Y_{-j}^{(h)}$. Then, \begin{itemize}
    \item[(a)] $\E\left(\sup_{\vt\in\Theta}\| X_1^{(h)}(\vt)\|^2\right)<\infty $ \quad and \quad $\max_{k\in\N}\left\{\frac{1}{(1+k)}\E\left(\sup_{\vt\in\Theta}\| \widehat\varepsilon_k^{(h)}(\vt)\|^2\right)\right\}<\infty$.
    \item[(b)] $\E\left(\sup_{\vt\in\Theta}\|\partial_u X_1^{(h)}(\vt)\|^2\right)<\infty$
    \quad and \quad  $\max_{k\in\N}\left\{\frac{1}{(1+k)}\E\left(\sup_{\vt\in\Theta}\|\partial_u \widehat\varepsilon_k^{(h)}(\vt)\|^2\right)\right\}<\infty$.
    \item[(c)] $\E\left(\sup_{\vt\in\Theta}\|\partial_{u,v} X_1^{(h)}(\vt)\|^2\right)<\infty$ \quad and \quad $\max_{k\in\N}\left\{\frac{1}{(1+k)}\E\left(\sup_{\vt\in\Theta}\|\partial_{u,v} \widehat\varepsilon_k^{(h)}(\vt)\|^2\right)\right\}<\infty.$
\end{itemize}
\end{lemma}
\begin{proof}
We  prove (a) exemplary for (b) and (c).
First, remark that $\E\|Y_j^{(h)}\|^2\leq \mathfrak{C}(1+|j|)$ for $j\in\Z$.
Since all eigenvalues of $(\e^{A_\vt h}-K_\vt^{(h)} C_\vt)$ lie inside the unit circle
(see Scholz~\cite[Lemma 4.6.7]{Scholz}) and all matrix functions are continuous on the compact set $\Theta$ and, hence,  bounded, we receive for some $0<\rho<1$
that $\sup_{\vt\in\Theta}\|\e^{A_\vt h}-K_\vt^{(h)} C_\vt\|\leq \rho$ and
$
    \sup_{\vt\in\Theta}\| X_1^{(h)}(\vt)\|\leq \mathfrak{C}\sum_{j=0}^{\infty}\rho^{j}\|Y_{-j}^{(h)}\|.
$
Thus,
\beao
    \E\left(\sup_{\vt\in\Theta}\| X_1^{(h)}(\vt)\|^2\right)\leq \mathfrak{C}\left(\sum_{j=0}^{\infty}\rho^{j}(\E\|Y_{-j}^{(h)}\|^2)^{1/2}\right)^2
        \leq \mathfrak{C}\left(\sum_{j=0}^{\infty}\rho^{j}(1+j)^{1/2}\right)^2<\infty.
\eeao
Similarly,
\beao
    \sup_{\vt\in\Theta}\|\widehat\varepsilon_k^{(h)}(\vt)\|\leq \|Y_k\| +\mathfrak{C}\rho^{k-1}\sup_{\vt\in\Theta}\|\widehat X_1^{(h)}(\vt)\| +\sum_{j=1}^{k-1}  \mathfrak{C}\rho^j\|Y_{k-j}^{(h)}\|,
\eeao
such that
\beao
    \E\left(\sup_{\vt\in\Theta}\|\widehat\varepsilon_k^{(h)}(\vt)\|^2\right)&\leq&
        3\E\|Y_k\|^2 +3\mathfrak{C}^2\rho^{2k-2}\E\left(\sup_{\vt\in\Theta}\|\widehat X_1^{(h)}(\vt)\|^2\right) +3
            \left(\sum_{j=1}^{k-1}  \mathfrak{C}\rho^j(\E\|Y_{k-j}^{(h)}\|^2)^{1/2}\right)^2\\
   %    &\leq& C\left((1+k)+\rho^{2k-2}\E\left(\sup_{\vt\in\Theta}\|\widehat X_1^{(h)}(\vt)\|^2\right)+\left(\sum_{j=1}^{k-1} \rho^j(1+k-j)^{1/2}\right)^2\right)\\
        &\leq&\mathfrak{C}\left((1+k)+\rho^{2k-2}\E\left(\sup_{\vt\in\Theta}\|\widehat X_1^{(h)}(\vt)\|^2\right)+k\left(\sum_{j=0}^{\infty} \rho^j\right)^2\right).
\eeao
Finally, due to \autoref{Ass:int}
\beao
    \max_{k\in\N}\left\{\frac{1}{(1+k)}\E\left(\sup_{\vt\in\Theta}\|\widehat\varepsilon_k^{(h)}(\vt)\|^2\right)\right\}\leq
    \mathfrak{C}\left(1+\E\left(\sup_{\vt\in\Theta}\|\widehat X_1^{(h)}(\vt)\|^2\right)+\left(\sum_{j=0}^{\infty} \rho^j\right)^2\right)<\infty.
\eeao
\end{proof}

\begin{lemma} \label{Lemma2.8}
Let \autoref{AssMBrownMot} and \ref{Ass:int} hold. Furthermore, let $u,v\in\{1,\ldots,s\}$.
\begin{itemize}
    \item[(a)] Then, there exists a positive random variable $\zeta$ with $\E(\zeta^2)<\infty$ and a constant $0<\rho<1$ so that
    $
        \sup_{\vt\in\Theta}\|\widehat\varepsilon_k^{(h)}(\vt)-\varepsilon_k^{(h)}(\vt)\|\leq \mathfrak{C}\rho^{k-1}\zeta
    $  for any $k\in\N$.
    \item[(b)] Then, there exists a positive random variable $\zeta^{(u)}$ with $\E(\zeta^{(u)})^2<\infty$ and a constant $0<\rho<1$ so that
    $
        \sup_{\vt\in\Theta}\|\partial_u\widehat\varepsilon_k^{(h)}(\vt)-\partial_u\varepsilon_k^{(h)}(\vt)\|\leq \mathfrak{C}\rho^{k-1}\zeta^{(u)}
    $  for any $k\in\N$.
    \item[(c)] Then, there exists a positive random variable $\zeta^{(u,v)}$ with $\E(\zeta^{(u,v)})^2<\infty$ and a constant $0<\rho<1$ so that
    $
        \sup_{\vt\in\Theta}\|\partial_{u,v}\widehat\varepsilon_k^{(h)}(\vt)-\partial_{u,v}\varepsilon_k^{(h)}(\vt)\|\leq \mathfrak{C}\rho^{k-1}\zeta^{(u,v)}
    $ for any $k\in\N$.
\end{itemize}
\end{lemma}
\begin{proof}
(a) \, We use the representation
\beao
    \widehat\varepsilon_k^{(h)}(\vt)-\varepsilon_k^{(h)}(\vt)=C_\vt(\e^{A_\vt h}-K_\vt^{(h)} C_\vt)^{k-1} (\widehat X_1^{(h)}(\vt)-X_1^{(h)}(\vt))
\eeao
and define $\zeta:=\sup_{\vt\in\Theta}\|\widehat X_1^{(h)}(\vt)\|+\sup_{\vt\in\Theta}\|X_1^{(h)}(\vt)\|$. Due to \autoref{Ass:int} and Lemma~\ref{Lemma2.9}(a)
we know that $\E(\zeta^2)<\infty$.  Since all eigenvalues of $(\e^{A_\vt h}-K_\vt^{(h)} C_\vt)$ lie inside the unit circle  and $C_\vt$ is bounded
as a continuous function on the compact set $\Theta$ there exists constants $\mathfrak{C}>0$ and $0<\rho<1$ so that $\sup_{\vt\in\Theta}\|C_\vt(\e^{A_\vt h}-K_\vt^{(h)} C_\vt)^{k-1}\|\leq \mathfrak{C}\rho^{k-1} $ and
\beao
        \sup_{\vt\in\Theta}\|\widehat\varepsilon_k^{(h)}(\vt)-\varepsilon_k^{(h)}(\vt)\|\leq
        \sup_{\vt\in\Theta}\|C_\vt(\e^{A_\vt h}-K_\vt^{(h)} C_\vt)\|^{k-1}\sup_{\vt\in\Theta} \|\widehat X_1^{(h)}(\vt)-X_1^{(h)}(\vt)\|\leq \mathfrak{C}\rho^{k-1}\zeta.
\eeao
(b,c) \, can be proven similarly.
\end{proof}

\begin{proof}[Proof of Proposition~\ref{Lemma 2.8}] $\mbox{}$\\
(a) \, First,
\beao
\lefteqn{\mathcal{\widehat{L}}_n^{(h)}(\vartheta)-\mathcal{L}_n^{(h)}(\vartheta)}\\
%    &=& \frac{1}{n}\sum_{k=1}^n\left[\widehat\varepsilon_k^{(h)}(\vartheta)^\mathsf{T}\big(V_\vt^{(h)}\big)^{-1}\widehat\varepsilon_k^{(h)}(\vartheta)
 %       -\varepsilon_k^{(h)}(\vartheta)^\mathsf{T}\big(V_\vt^{(h)}\big)^{-1}\varepsilon_k^{(h)}(\vartheta)\right]\\
    &=& \frac{1}{n}\sum_{k=1}^n\left[(\widehat\varepsilon_k^{(h)}(\vartheta)-\varepsilon_k^{(h)}(\vartheta))^\mathsf{T}\big(V_\vt^{(h)}\big)^{-1}\widehat\varepsilon_k^{(h)}(\vartheta)
        -\varepsilon_k^{(h)}(\vartheta)^\mathsf{T}\big(V_\vt^{(h)}\big)^{-1}(\varepsilon_k^{(h)}(\vartheta)-\widehat\varepsilon_k^{(h)}(\vartheta))\right].
\eeao
Then, \begin{align*}
    &n\sup_{\vt\in\Theta}|\mathcal{\widehat{L}}_n^{(h)}(\vartheta)-\mathcal{L}_n^{(h)}(\vartheta)|\\
        &\leq
        \sup_{\vt\in\Theta}\|(V_\vt^{(h)})^{-1}\|
        \sum_{k=1}^n\left[\sup_{\vt\in\Theta}\|\widehat\varepsilon_k^{(h)}(\vartheta)-\varepsilon_k^{(h)}(\vartheta)\|\left(\sup_{\vt\in\Theta}\|\widehat\varepsilon_k^{(h)}(\vartheta)\|
        +\sup_{\vt\in\Theta}\|\varepsilon_k^{(h)}(\vartheta)\|\right)\right].
\intertext{Due to Lemma~\ref{Lemma 2.3} and Lemma~\ref{Lemma2.8}(a)}
&\leq
        \mathfrak{C}\zeta
        \sum_{k=1}^n\rho^{k}\left[\sup_{\vt\in\Theta}\|\widehat\varepsilon_k^{(h)}(\vartheta)\|
        +\sup_{\vt\in\Theta}\|\varepsilon_k^{(h)}(\vartheta)\|\right]
\end{align*}
with $\E(\zeta^2)<\infty$. From this and Cauchy-Schwarz inequality we can conclude that
\begin{align*}
    &\hspace*{-1cm}n\E\left(\sup_{\vt\in\Theta}|\mathcal{\widehat{L}}_n^{(h)}(\vartheta)-\mathcal{L}_n^{(h)}(\vartheta)|\right)\\
    &\leq \mathfrak{C}(\E\zeta^2)^{1/2}\sum_{k=1}^n\rho^k \left[\E\left(\sup_{\vt\in\Theta}\|\widehat\varepsilon_k^{(h)}(\vartheta)\|^2\right)^{1/2}+
        \E\left(\sup_{\vt\in\Theta}\|\varepsilon_k^{(h)}(\vartheta)\|^2\right)^{1/2}\right].
\intertext{An application of Lemma~\ref{Lemma2.9}(a) yields}
    &\leq  \mathfrak{C}(\E\zeta^2)^{1/2}\sum_{k=1}^\infty\rho^k(1+k)^{1/2}<\infty.
\end{align*}
This proves, $n\sup_{\vt\in\Theta}|\mathcal{\widehat{L}}_n^{(h)}(\vartheta)-\mathcal{L}_n^{(h)}(\vartheta)|=\mathcal{O}_p(1)$ so that (a) follows.
Again (b) and (c) can be proven similarly.
\end{proof}

\end{document}